\documentclass[reqno]{amsart}
\include{packages}
\bibliography{sources}
\numberwithin{equation}{section}

\DeclareMathAlphabet{\mathpzc}{OT1}{pzc}{m}{it}

\makeatletter
\DeclareFieldFormat[article]{title}{{\it #1}}
\DeclareFieldFormat{journaltitle}{{\rm #1}}
\renewbibmacro{in:}{%
  \ifentrytype{article}{}{\printtext{\bibstring{in}\intitlepunct}}}
\DeclareFieldFormat[incollection]{title}{{\it #1}}
\DeclareFieldFormat{journaltitle}{{\rm #1}}

\begin{document}

\def\subsectionautorefname{Section}
\def\subsubsectionautorefname{Section}
\def\sectionautorefname{Section}
\def\equationautorefname~#1\null{(#1)\null}

\newcommand{\mynewtheorem}[4]{
  \if\relax\detokenize{#3}\relax 
    \if\relax\detokenize{#4}\relax 
      \newtheorem{#1}{#2}
    \else
      \newtheorem{#1}{#2}[#4]
    \fi
  \else
    \newaliascnt{#1}{#3}
    \newtheorem{#1}[#1]{#2}
    \aliascntresetthe{#1}
  \fi
  \expandafter\def\csname #1autorefname\endcsname{#2}
}

\mynewtheorem{theorem}{Theorem}{}{section}
\mynewtheorem{lemma}{Lemma}{theorem}{}
\mynewtheorem{lem}{Lemma}{theorem}{}
\mynewtheorem{rem}{Remark}{lemma}{}
\mynewtheorem{prop}{Proposition}{lemma}{}
\mynewtheorem{cor}{Corollary}{lemma}{}
\mynewtheorem{definition}{Definition}{lemma}{}
\mynewtheorem{question}{Question}{lemma}{}
\mynewtheorem{assumption}{Assumption}{lemma}{}
\mynewtheorem{example}{Example}{lemma}{}
\mynewtheorem{exm}{Example}{lemma}{}
\mynewtheorem{rmk}{Remark}{lemma}{}
\mynewtheorem{pb}{Problem}{lemma}{}
\newtheorem*{rmk*}{Remark}
\newtheorem*{pb*}{Problem}

\newtheorem{conjecture}[theorem]{Conjecture}
\newtheorem{condition}[theorem]{Condition}
\newtheorem{setup}[theorem]{Setup}

\def\defbb#1{\expandafter\def\csname b#1\endcsname{\mathbb{#1}}}
\def\defcal#1{\expandafter\def\csname c#1\endcsname{\mathcal{#1}}}
\def\deffrak#1{\expandafter\def\csname frak#1\endcsname{\mathfrak{#1}}}
\def\defop#1{\expandafter\def\csname#1\endcsname{\operatorname{#1}}}
\def\defbf#1{\expandafter\def\csname b#1\endcsname{\mathbf{#1}}}

\makeatletter
\def\defcals#1{\@defcals#1\@nil}
\def\@defcals#1{\ifx#1\@nil\else\defcal{#1}\expandafter\@defcals\fi}
\def\deffraks#1{\@deffraks#1\@nil}
\def\@deffraks#1{\ifx#1\@nil\else\deffrak{#1}\expandafter\@deffraks\fi}
\def\defbbs#1{\@defbbs#1\@nil}
\def\@defbbs#1{\ifx#1\@nil\else\defbb{#1}\expandafter\@defbbs\fi}
\def\defbfs#1{\@defbfs#1\@nil}
\def\@defbfs#1{\ifx#1\@nil\else\defbf{#1}\expandafter\@defbfs\fi}
\def\defops#1{\@defops#1,\@nil}
\def\@defops#1,#2\@nil{\if\relax#1\relax\else\defop{#1}\fi\if\relax#2\relax\else\expandafter\@defops#2\@nil\fi}
\makeatother

\defbbs{ZHQCNPALRVWS}
\defcals{ABCDOPQMNXYLTRAEHZKCIV}
\deffraks{apijklgmnopqueRCc}
\defops{IVC, PGL,SL,mod,Spec,Re,Gal,Tr,End,GL,Hom,PSL,H,div,Aut,rk,Mod,R,T,Tr,Mat,Vol,MV,Res,Hur, vol,Z,diag,Hyp,hyp,hl,ord,Im,ev,U,dev,c,CH,fin,pr,Pic,lcm,ch,td,LG,id,Sym,Aut,Log,tw,irr,discrep,BN,NF,NC,age,hor,lev,ram,NH,av,app,Quad,Stab,Per,Nil,Ker,EG,CEG,PB,Conf,MCG,Diff}
\defbfs{uvzwp} 

\def\ep{\varepsilon}
\def\ve{\varepsilon}
\def\abs#1{\lvert#1\rvert}
\def\dd{\mathrm{d}}
\def\WP{\mathrm{WP}}
\def\inj{\hookrightarrow}
\def\eq{=}

\def\i{\mathrm{i}}
\def\e{\mathrm{e}}
\def\st{\mathrm{st}}
\def\ct{\mathrm{ct}}
\def\rel{\mathrm{rel}}
\def\odd{\mathrm{odd}}
\def\even{\mathrm{even}}

\def\uC{\underline{\bC}}
\def\ol{\overline}

\def\Vrel{\bV^{\mathrm{rel}}}
\def\Wrel{\bW^{\mathrm{rel}}}
\def\twolev{\mathrm{LG_1(B)}}

\def\be{\begin{equation}}   \def\ee{\end{equation}}     \def\bes{\begin{equation*}}    \def\ees{\end{equation*}}
\def\ba{\be\begin{aligned}} \def\ea{\end{aligned}\ee}   \def\bas{\bes\begin{aligned}}  \def\eas{\end{aligned}\ees}
\def\={\;=\;}  \def\+{\,+\,} \def\m{\,-\,}

\newcommand*{\proj}{\mathbb{P}}
\newcommand{\IVCst}[1][\mu]{{\mathcal{IVC}}({#1})}
\newcommand{\barmoduli}[1][g]{{\overline{\mathcal M}}_{#1}}
\newcommand{\moduli}[1][g]{{\mathcal M}_{#1}}
\newcommand{\omoduli}[1][g]{{\Omega\mathcal M}_{#1}}
\DeclareDocumentCommand{\qmoduli}{ O{g} O{}}{{\mathrm{Quad}^{#2}_{#1}}}

\newcommand{\modulin}[1][g,n]{{\mathcal M}_{#1}}
\newcommand{\omodulin}[1][g,n]{{\Omega\mathcal M}_{#1}}
\newcommand{\zomoduli}[1][]{{\mathcal H}_{#1}}
\newcommand{\barzomoduli}[1][]{{\overline{\mathcal H}_{#1}}}
\newcommand{\pomoduli}[1][g]{{\proj\Omega\mathcal M}_{#1}}
\newcommand{\pomodulin}[1][g,n]{{\proj\Omega\mathcal M}_{#1}}
\newcommand{\pobarmoduli}[1][g]{{\proj\Omega\overline{\mathcal M}}_{#1}}
\newcommand{\pobarmodulin}[1][g,n]{{\proj\Omega\overline{\mathcal M}}_{#1}}
\newcommand{\potmoduli}[1][g]{\proj\Omega\tilde{\mathcal{M}}_{#1}}
\newcommand{\obarmoduli}[1][g]{{\Omega\overline{\mathcal M}}_{#1}}
\newcommand{\qbarmoduli}[1][g]{{\Omega^2\overline{\mathcal M}}_{#1}}
\newcommand{\obarmodulio}[1][g]{{\Omega\overline{\mathcal M}}_{#1}^{0}}
\newcommand{\otmoduli}[1][g]{\Omega\tilde{\mathcal{M}}_{#1}}
\newcommand{\pom}[1][g]{\proj\Omega{\mathcal M}_{#1}}
\newcommand{\pobarm}[1][g]{\proj\Omega\overline{\mathcal M}_{#1}}
\newcommand{\pobarmn}[1][g,n]{\proj\Omega\overline{\mathcal M}_{#1}}
\newcommand{\princbound}{\partial\mathcal{H}}
\newcommand{\omoduliinc}[2][g,n]{{\Omega\mathcal M}_{#1}^{{\rm inc}}(#2)}
\newcommand{\obarmoduliinc}[2][g,n]{{\Omega\overline{\mathcal M}}_{#1}^{{\rm inc}}(#2)}
\newcommand{\pobarmoduliinc}[2][g,n]{{\proj\Omega\overline{\mathcal M}}_{#1}^{{\rm inc}}(#2)}
\newcommand{\otildemoduliinc}[2][g,n]{{\Omega\widetilde{\mathcal M}}_{#1}^{{\rm inc}}(#2)}
\newcommand{\potildemoduliinc}[2][g,n]{{\proj\Omega\widetilde{\mathcal M}}_{#1}^{{\rm inc}}(#2)}
\newcommand{\omoduliincp}[2][g,\lbrace n \rbrace]{{\Omega\mathcal M}_{#1}^{{\rm inc}}(#2)}
\newcommand{\obarmoduliincp}[2][g,\lbrace n \rbrace]{{\Omega\overline{\mathcal M}}_{#1}^{{\rm inc}}(#2)}
\newcommand{\obarmodulin}[1][g,n]{{\Omega\overline{\mathcal M}}_{#1}}
\newcommand{\LTH}[1][g,n]{{K \overline{\mathcal M}}_{#1}}
\newcommand{\PLS}[1][g,n]{{\bP\Xi \mathcal M}_{#1}}

\DeclareDocumentCommand{\LMS}{ O{\mu} O{g,n} O{}}{\Xi\overline{\mathcal{M}}^{#3}_{#2}(#1)}
\DeclareDocumentCommand{\Romod}{ O{\mu} O{g,n} O{}}{\Omega\mathcal{M}^{#3}_{#2}(#1)}

\newcommand*{\Tw}[1][\Lambda]{\mathrm{Tw}_{#1}}  
\newcommand*{\sTw}[1][\Lambda]{\mathrm{Tw}_{#1}^s}  

\newcommand{\HH}{{\mathbb H}}
\newcommand{\MM}{{\mathbb M}}
\newcommand{\bbC}{{\mathbb C}}
\newcommand{\TT}{{\mathbb T}}

\newcommand\PP{\mathbb P}
\renewcommand\R{\mathbb R}
\renewcommand\Z{\mathbb Z}
\newcommand\N{\mathbb N}
\newcommand\Q{\mathbb Q}
\renewcommand{\H}{\mathbb{H}}
\newcommand{\halfplane}{\mathbb{H}}
\newcommand{\chalfplane}{\overline{\mathbb{H}}}
\newcommand{\bk}{k}

\newcommand{\bfa}{{\bf a}}
\newcommand{\bfb}{{\bf b}}
\newcommand{\bfd}{{\bf d}}
\newcommand{\bfe}{{\bf e}}
\newcommand{\bff}{{\bf f}}
\newcommand{\bfg}{{\bf g}}
\newcommand{\bfh}{{\bf h}}
\newcommand{\bfm}{{\bf m}}
\newcommand{\bfn}{{\bf n}}
\newcommand{\bfp}{{\bf p}}
\newcommand{\bfq}{{\bf q}}
\newcommand{\bft}{{\bf t}}
\newcommand{\bfP}{{\bf P}}
\newcommand{\bfR}{{\bf R}}
\newcommand{\bfU}{{\bf U}}
\newcommand{\bfu}{{\bf u}}
\newcommand{\bfx}{{\bf x}}
\newcommand{\bfz}{{\bf z}}

\newcommand{\bfl}{{\boldsymbol{\ell}}}
\newcommand{\bfmu}{{\boldsymbol{\mu}}}
\newcommand{\bfeta}{{\boldsymbol{\eta}}}
\newcommand{\bftau}{{\boldsymbol{\tau}}}
\newcommand{\bfomega}{{\boldsymbol{\omega}}}
\newcommand{\bfsigma}{{\boldsymbol{\sigma}}}
\newcommand{\bfnu}{{\boldsymbol{\nu}}}
\newcommand{\bfrho}{{\boldsymbol{\rho}}}
\newcommand{\bfone}{{\boldsymbol{1}}}

\newcommand\cl{\mathcal}
\newcommand{\calH}{\mathcal{H}}

\newcommand{\calA}{\mathcal A}
\newcommand{\calK}{\mathcal K}
\newcommand{\calD}{\mathcal D}
\newcommand{\calC}{\mathcal C}
\newcommand\C{\mathcal C}
\newcommand{\calT}{\mathcal T}
\newcommand{\calB}{\mathcal B}
\newcommand{\calF}{\mathcal F}
\newcommand{\calV}{\mathcal V}
\newcommand{\calQ}{\mathcal Q}
\newcommand{\calX}{\mathcal X}
\newcommand{\calY}{\mathcal Y}
\newcommand{\calP}{\mathcal P}
\newcommand{\calJ}{\mathcal J}
\newcommand{\calS}{\mathcal S}
\newcommand{\calZ}{\mathcal Z}

\newcommand\pvd{\operatorname{pvd}}
\newcommand\per{\operatorname{per}}
\newcommand\thick{\operatorname{thick}}

\newcommand\perf{\operatorname{perf}}
\newcommand{\modules}{\operatorname{mod}}
\newcommand{\Loc}{\operatorname{Loc}}
\renewcommand{\Mod}{\operatorname{Mod}}
\renewcommand{\Hom}{\operatorname{Hom}}
\newcommand{\Ext}{\operatorname{Ext}}
\newcommand{\coker}{\operatorname{coker}}

\newcommand\Rep{\operatorname{Rep}}
\newcommand\ext{\operatorname{ext}}
\newcommand{\heart}{\heartsuit}

\newcommand{\sph}{\operatorname{sph}}
\newcommand{\Br}{\operatorname{Br}}
\renewcommand{\Tw}{\operatorname{Tw}}
\newcommand{\RHom}{\operatorname{RHom}}

\newcommand{\torsion}{\mathcal{T}}
\newcommand{\torsionfree}{\mathcal{F}}
\newcommand{\tstr}{\mathcal{L}}

\newcommand\cy{\mathrm{CY}}
\newcommand\CY{\mathrm{CY}}

\newcommand\DQI{{\mathcal{D}_{Q_I}}}
\newcommand\AQ{{\mathcal{A}_Q}}
\newcommand\AQI{{\mathcal{A}_{Q_I}}}
\newcommand{\DQ}{{\mathcal{D}_{Q}}}
\DeclareDocumentCommand{\DQQ}{ O{}O{}}{{\mathcal{D}^{#1}_{#2}}}
\DeclareDocumentCommand{\AQQ}{ O{q} }{{\mathcal{A}_{#1}}}

\newcommand{\sgn}{\operatorname{sgn}}
\renewcommand{\GL}{\operatorname{GL}}
\renewcommand{\rk}{\operatorname{rank}}
\newcommand{\opL}{\operatorname{L}}

\newcommand\stab{\operatorname{Stab}}
\newcommand\gstab{\operatorname{GStab}}

\newcommand{\spn}{\operatorname{span}}
\newcommand{\GStab}{\operatorname{GStab}}
\newcommand{\PGStab}{\bP\mathrm{GStab}}
\newcommand{\MStab}{\operatorname{MStab}}
\newcommand{\PMStab}{\bP\mathrm{MStab}}
\newcommand{\Tilt}{\operatorname{Tilt}}

\def\oQ{\overline{Q}}
\def\bY{\mathbf{Y}}
\def\bX{\mathbf{X}}
\newcommand{\fmu}{\mu^\sharp}
\newcommand{\bmu}{\mu^\flat}
\newcommand{\Cone}{\operatorname{Cone}}
\newcommand\add{\operatorname{Add}}
\newcommand\Irr{\operatorname{Irr}}

\newcommand{\EGs}{\EG^{s}} 
\newcommand{\SEG}{\operatorname{SEG}} 
\newcommand{\pSEG}{\operatorname{pSEG}} 
\newcommand{\EGp}{\EG^\circ}       
\newcommand{\SEGp}{\SEG^\circ}       
\newcommand{\EGb}{\EG^\bullet}       
\newcommand{\SEGb}{\SEG^\bullet}       
\newcommand{\SEGV}{\SEG_{{^\perp}\calV}} 
\newcommand{\pSEGV}{\pSEG_{{^\perp}\calV}}
\newcommand{\SEGVb}{\SEGb_{{^\perp}\calV}} 
\newcommand{\pSEGVb}{\pSEG^\bullet_{{^\perp}\calV}}

\newcommand{\EGT}{\EG^\circ}
\newcommand{\uEG}{\underline{\EG}} 
\newcommand{\uCEG}{\underline{\CEG}} 

\newcommand{\CA}{\operatorname{CA}}
\newcommand{\OA}{\operatorname{OA}}
\newcommand{\wOA}{\widetilde{\OA}}
\newcommand{\wCA}{\widetilde{\CA}}
\newcommand{\BT}{\operatorname{BT}}
\newcommand{\SBr}{\operatorname{SBr}}
\newcommand\Bt[1]{\operatorname{B}_{#1}}
\newcommand\bt[1]{\operatorname{B}_{#1}^{-1}}

\newcommand{\Sim}{\operatorname{Sim}}
\def\dual{\iota}
\def\ivc{\iota_v}
\newcommand\ind{\operatorname{index}}
\newcommand{\Qgrad}{\overline{Q}}
\newcommand\diff{\operatorname{d}}
\def\Dsow{\Dcol}
\def\pvc{e \Gamma e}
\def\Qvc{Q_V^c}
\def\Gwt{\FGamma_\wt}
\def\wX{\widetilde{X}}

\def\DD{\mathbf{D}}
\def\uD{\underline{\calD}}
\def\uh{\underline{\calH}}
\def\uZ{\underline{Z}}
\def\us{\underline{\sigma}}
\newcommand{\isom}{\cong}

\newcommand{\tilt}[3]{{#1}^{#2}_{#3}}
\newcommand\Stap{\Stab^\circ} 
\newcommand\Stas{\Stab^\bullet}
\newcommand{\cub}{\operatorname{U}} 
\newcommand{\skel}{\wp} 

\newcommand{\mai}{\mathbf{i}} 

\newcommand\wT{\widetilde{\TT}} 
\newcommand{\Tri}{\Delta}
\newcommand{\deco}{\Delta}

\def\w{\mathbf{w}}
\def\wtpt{\w_{+2}}
\newcommand\surf{\mathbf{S}}  
\newcommand\surfo{{\mathbf{S}}_\Tri}  
\newcommand\surfw{\surf^\w}  
\newcommand\sow{\surf_\w}  
\newcommand\subsur{\Sigma}  
\newcommand\colsur{\overline{\surf}_\w}  
\def\Dsan{\calD_3(\surfo)}
\def\Dsow{\calD(\sow)}
\def\Dsub{\calD_3(\subsur)}
\def\Dcol{\calD(\colsur)}

\def\ww{node[white]{$\bullet$}node[red]{$\circ$}}
\def\nn{node{$\bullet$}}

\newcommand{\uk}{\mathbf{k}}
\def\sing{\operatorname{Sing}}

\newcommand\jiantou{edge[->]}
\newcommand\AS{\mathbb{A}}

\def\grad{\lambda}
\def\gms{\surf^\grad}
\def\gmsw{\surf^{\grad,\wt}_{\Tri}}
\def\iT{\TT_0}

\def\weta{\widetilde{\eta}}
\def\wzeta{\widetilde{\zeta}}
\def\walpha{\widetilde{\alpha}}
\def\wbeta{\widetilde{\beta}}
\def\wgamma{\widetilde{\gamma}}

\def\dsan{\calD_3(\surfo)}
\def\psan{\per(\surfo)}
\def\dsow{\Dsow}
\def\psow{\per(\colsur)}

\def\RP{\operatorname{RP}}
\def\Rs{\operatorname{RS}}
\newcommand{\ST}{\operatorname{ST}}  
\newcommand{\STp}{\ST^\circ}  

\newcommand{\Int}{\operatorname{Int}}

\def\hori{_{\operatorname{H}}}

\newcommand{\Note}[1]{\textcolor{red}{#1}}
\newcommand{\note}[1]{\textcolor{Emerald}{#1}}
\newcommand{\qy}[1]{\textcolor{cyan}{#1}}

\newcommand\foli[5]{
	\foreach \m in {0,#1,...,#5}{
		\coordinate (#2#4) at ($(#2)!.5!(#4)$);
		\draw[Emerald!50]plot [smooth,tension=.5] coordinates
		{(#2) ($(#2#4)!\m!(#3)$) (#4)};
		\draw[thick,gray](#2)to(#3)to(#4);
}}

\def\iA{\AS_0}
\def\iA{\AS_0}

\newcommand\Ind{\operatorname{Ind}}

\newcommand\xx{\mathbf{X}} 
\newcommand\surp{\xx^\circ}
\newcommand{\Zer}{\operatorname{Zero}}
\newcommand{\Pol}{\operatorname{Pol}}
\newcommand{\Crit}{\operatorname{Crit}}
\newcommand{\FQuad}{\operatorname{FQuad}}
\newcommand{\FQuab}{\FQuad^{\bullet}}

\newcommand{\Imgy}{\operatorname{Im}}
\newcommand{\numarc}{n}


\newcommand{\Exch}{\operatorname{Exch}}
\newcommand{\arrowIn}{
	\tikz \draw[-stealth] (-1pt,0) -- (1pt,0);
}


\newcommand{\wh}{\widehat}
\newcommand{\wt}{\widetilde}

\newcommand{\whmu}{\widehat{\mu}}
\newcommand{\whrho}{\widehat{\rho}}
\newcommand{\whLa}{\widehat{\Lambda}}

\newcommand{\ps}{\mathrm{ps}}

\newcommand{\tdpm}[1][{\Gamma}]{\mathfrak{W}_{\operatorname{pm}}(#1)}
\newcommand{\tdps}[1][{\Gamma}]{\mathfrak{W}_{\operatorname{ps}}(#1)}
\newcommand{\cal}[1]{\mathcal{#1}}

\newlength{\halfbls}\setlength{\halfbls}{.8\baselineskip}

\newcommand*{\Teichmuller}{Teich\-m\"uller\xspace}

\DeclareDocumentCommand{\MSfun}{ O{\mu} }{\mathbf{MS}({#1})}
\DeclareDocumentCommand{\MSgrp}{ O{\mu} }{\mathcal{MS}({#1})}
\DeclareDocumentCommand{\MScoarse}{ O{\mu} }{\mathrm{MS}({#1})}
\DeclareDocumentCommand{\tMScoarse}{ O{\mu} }{\widetilde{\mathbb{P}\mathrm{MS}}({#1})}

\newcommand{\kmin}{\kappa_{(2g-2)}}
\newcommand{\ktop}{\kappa_{\mu_\Gamma^{\top}}}
\newcommand{\kbot}{\kappa_{\mu_\Gamma^{\bot}}}

\newcommand\bra{\langle}
\newcommand\ket{\rangle}
\newcommand{\del}{\partial}
\newcommand{\deld}[1]{\frac{\del}{\del {#1} }}
\newcommand{\frap}{\frac{1}{2\pi\ii}}
\renewcommand{\Re}{\operatorname{Re}}
\renewcommand{\Im}{\operatorname{Im}}
\newcommand{\Rp}{\mathbb{R}_{>0}}
\newcommand{\Rm}{\R_{<0}}

\def\lift{\mathrm{lift}}
\def\ul{\underline}
\newcommand\<{\langle}
\renewcommand\>{\rangle}
\def\h{\calH}
\def\D{\calD}
\def\aut{\mathpzc{Aut}}
\tikzcdset{arrow style=tikz, diagrams={>=stealth}}
\usetikzlibrary{arrows}

\def\acf{\mathbf{k}}
\def\PTS{\mathbb{P}T\surf}
\newcommand\coho[1]{\operatorname{H}^{#1}}
\def\MTS{\mathbb{R}T\surf^{\lambda}}

\def\ora{\overrightarrow}
\def\harc{\ora{\gamma_h}}
\def\varc{\ora{\gamma_v}}
\def\cube{\mathrm{U}}

\newcommand{\on}[1]{\operatorname{#1}}
\newcommand{\iv}[1]{(#1)^{-1}}


\title[Quadratic differentials as stability conditions]
      {Quadratic differentials as stability conditions:
        collapsing subsurfaces }

\author{Anna Barbieri}
\address{A.B.:  Dipartimento di Informatica - Settore Matematica,
Universit\`a di Verona,
Strada Le Grazie 15, 37134 Verona - Italy }
\email{anna.barbieri@univr.it}
\thanks{Research of A.B. was supported by the ERC Consolidator Grant ERC-2017-CoG-771507, StabCondEn, "Stability Conditions, Moduli Spaces and Enhancements"}

\author{Martin M\"oller}
\address{M.M.: Institut f\"ur Mathematik, Goethe-Universit\"at Frankfurt,
Robert-Mayer-Str. 6-8,
60325 Frankfurt am Main, Germany}
\email{moeller@math.uni-frankfurt.de}
\thanks{Research of M.M. and J.S.\ is supported
by the DFG-project MO 1884/2-1, by the LOEWE-Schwerpunkt
``Uniformisierte Strukturen in Arithmetik und Geometrie''
and the Collaborative Research Centre
TRR 326 ``Geometry and Arithmetic of Uniformized Structures''.}

\author{Yu Qiu}
\address{Y.Q.:
	Yau Mathematical Sciences Center and Department of Mathematical Sciences,
	Tsinghua University,
    100084 Beijing,
    China.
    \&
    Beijing Institute of Mathematical Sciences and Applications, Yanqi Lake, Beijing, China}
\email{yu.qiu@bath.edu}
\thanks{Research of Q.Y. is supported by
National Key R\&D Program of China (No. 2020YFA0713000),
Beijing Natural Science Foundation (Grant No.Z180003) and
National Natural Science Foundation of China (Grant No.12031007).
}

\author{Jeonghoon So}
\address{J.S.: Institut f\"ur Mathematik, Goethe-Universit\"at Frankfurt,
Robert-Mayer-Str. 6-8,
60325 Frankfurt am Main, Germany}
\email{so@math.uni-frankfurt.de}

\dedicatory{Dedicated to Bernhard Keller on the occasion of his sixtieth birthday}

\begin{abstract}

We introduce a new class of triangulated categories, which are Verdier quotients
of 3-Calabi-Yau categories from (decorated) marked surfaces, and show that its
spaces of stability conditions can be identified with moduli spaces of framed
quadratic differentials on Riemann surfaces with arbitrary order zeros and arbitrary
higher order poles.
\par
A main tool in our proof is a comparison of two exchange graphs, obtained by
tilting hearts in the quotient categories and by flipping mixed-angulations
associated with the quadratic differentials.
\end{abstract}
\maketitle
\tableofcontents

\section{Introduction}

The notion of stability conditions on a triangulated category $\D$ was introduced
by Bridgeland  in \cite{BrStab}. Since then, the stability
space~$\Stab\D$, which as a set consists of Bridgeland stability conditions on~$\D$,
has played a major role in algebraic geometry,  representation theory, mirror
symmetry and some branches of mathematical physics, providing interesting
synergies. By its very definition~$\Stab\D$ comes with a $\wt{\GL}^+_2(\bR)$-action,
just as moduli spaces of framed abelian and quadratic differentials do.
\par
While the global structure of~$\Stab\D$ as a complex manifold is still unknown
in many cases, there are examples that are quite well understood. This includes for
instance  the case of the stability space of a class of three-Calabi-Yau ($\CY_3$)
categories constructed from the Ginzburg algebra of quivers with potentials, that
are well known categories in representation and cluster
theory. Inspired by the work  of Gaiotto-Moore-Neitzke \cite{GMN}, Bridgeland and
Smith have shown in \cite{BS15} that some moduli spaces of meromorphic
quadratic differentials with simple zeros can be identified with those spaces
of stability conditions, appropriately quotiented by the action of autoequivalences.
\par\smallskip
Our goal here is to generalize the Bridgeland-Smith correspondence to
quadratic differentials with arbitrary higher order zeros. It implies studying
another class of categories, which are related to the previous ones as
quotients, but seem less well-behaved. Our motivation for this is two-fold.
\par\smallskip
\textbf{Categorification.} Spaces of quadratic differentials
with higher order zeros arise when zeros collide. As such they form a subspace
of the total space of quadratic differentials with no zero order condition,
in fact a subspace locally cut out by linear conditions in period coordinates.
In the spaces of abelian and quadratic differentials, the $\bR$-linear submanifolds
have received a lot of attention (see e.g.\ \cite{FiSurvey} for a recent survey
on the classification problem, however with focus on holomorphic differentials).
Since these submanifolds admit an action of the universal cover of $\GL_2^+(\bR)$, a natural question is
whether they all can be interpreted as spaces of stability conditions on an
appropriate triangulated category.
\par
More generally one can analyze the collision of zeros and poles, or even the
collapse of a higher genus subsurface. Our main result gives an answer to
the question how to interpret such collapses categorically. In a nutshell,
collapses correspond to taking Verdier quotients.
\par\smallskip
\textbf{Compactification.} Spaces of stability conditions are
typically non-compact, even after projectivization,  and several strategies of
compactification have recently
been explored. Some of them are Thurston-type compactifications with real
codimension one boundary (\cite{BDL_Thurston}, \cite{KKO}), some of them are partial
compactifications (\cite{barbara}, \cite{BPPW}). On the other hand, spaces of
projectivized quadratic differentials have a compactification as smooth complex
orbifolds (combine \cite{LMS} and \cite{kdiff}) and in forthcoming work we will
recast this compactification in terms of 'multi-scale stability conditions' for quiver
$\CY_3$ categories.
Spaces of quadratic differentials with higher order zeros appear naturally as
boundary strata in this compactification.

\subsection{The main result}
The combinatorics of a meromorphic quadratic differential~$q$ on a Riemann
surface $S$ is encoded in a weighted decorated marked
surface~$\surf_\mathbf{w}$, the real blow-up of $S$  at the poles of~$q$, see
Section~\ref{app:DMSqdiff} for the definition. The pole orders are encoded
in the markings of (a finite number of points in) the boundary components
of~$\surf_\mathbf{w}$, and usually hidden from notation. The weights~$\w$ encode the
tuple of orders of zeros of the differential. The horizontal trajectories
of a generic~$q$ induce a tiling of $\surf_\mathbf{w}$ into polygons, whose number of
edges depends on the orders of the zeros they contain.
\par
The case of $\surfo:=\surf_{\mathbf{w}\equiv \mathbf{1}}$ is the one originally
considered by \cite{BS15} and \cite{KQ2}. It corresponds to quadratic
differentials  with simple zeros. These differentials  induce a triangulation
of $\surfo$ to which, in turn, one associates a quiver with potential $(Q,W)$ and
its Ginzburg (differential graded) algebra $\Gamma(Q,W)$. The triangulated
category~$\D_3(\surfo)$ is defined as the perfectly valued derived category
$\pvd(\Gamma(Q,W))$, that is the subcategory of $\calD(\Gamma)$ of $\Gamma$-modules
with finite dimensional total homology. The correspondence of \cite{BS15,KQ2} can
be restated as an isomorphism of complex manifolds
\[
K: \FQuad^\circ(\surfo)   \to \Stab^\circ(\D_3(\surfo)),
\]
involving the moduli space of (Teichm\"{u}ller-)framed quadratic differentials
on $\surfo$ and a connected component  $\Stab^\circ(\cD)$ of the stability
manifold of $\D_3(\surfo)$.
\par
Section~\ref{sec:prelim} contains background material on Bridgeland stability
conditions and quotient categories, as well as how to associate
Ginzburg categories to quivers with potential. We summarize the notions of
marked surfaces, weighted decorations and quadratic differentials in
Section~\ref{sec:dmsQD}.
The geometry of moduli spaces with all kinds of framings is recalled
in Section~\ref{sec:geoqdiff}.
The previous results by \cite{BS15,KQ2} together with mapping class group actions
are restated in Theorem~\ref{thm:BS15_iso}.
\par
Consider now quadratic differentials with signature~$\mathbf{w}$ different from
the 'trivial' case $\w \equiv \mathbf{1}$ and their associated $\surf_\mathbf{w}$.
A weighted decorated marked surface with non-trivial weight can be obtained by
collapsing a subsurface $\subsur$ in $\surfo$, as we explain in
Section~\ref{sec_subsurface}. In such a case we denote it by $\colsur$.
\emph{In the whole paper the surface~$\colsur$ has at least one boundary component
(i.e.\ the quadratic differentials are meromorphic), there are no
punctures (i.e.\ none of the marked points is a regular point of the
quadratic differentials) and we disallow double poles and simple poles}
to avoid several technicalities like working with cohomology valued in local
systems (the space $\Quad^\heart$ of \cite{BS15}) and self-folded triangles in
triangulations.
\par
The main result of this paper is Theorem \ref{thm:KrhoBihol}, stated in short
form as follows:
\par
\begin{theorem} \label{intro:KrhoBihol}
There is an isomorphism of complex manifolds
\[
K: \FQuad^\bullet(\colsur)   \to \Stab^\bullet(\D(\colsur))
\]
between the principal part of the space of Teichm\"uller-framed quadratic differentials
inducing the weighted decorated marked surface~$\colsur$ and the principal part of the space of
stability conditions on the Verdier quotient
\[\D(\colsur) := \Dsan/\Dsub.\]
\end{theorem}
In this theorem the bullet points ('principal part') refer to a union of connected
components, defined in Sections~\ref{sec:refprinc} and~\ref{sec:EGquotheart}
respectively, and motivated below. Our results can most likely be extended to
include punctures and small order poles with appropriate care. The case of
holomorphic differentials is a whole different story, for which the recent
categorification by Haiden (\cite{H}) could be the point of departure.
\par
A given decorated marked surface~$\colsur$ may be realized as
the collapse of several different surfaces~$\surfo$ with simple weights: the
case $g(\surfo) = g(\colsur)$ is always possible, $g(\surfo) >
g(\colsur)$ is possible
if the entries of~$\w$ are large enough. Since the spaces of framed quadratic
differentials do not depend on the collapse, Theorem \ref{intro:KrhoBihol} gives the
realization of the same manifold~$M$ as $M \cong \Stab^\bullet(\D(\colsur))$
for different triangulated categories~$\D(\colsur)$. However the autoequivalences
of~$\D(\colsur)$ detect those different realizations, just as the mapping
class groups do on the topological side, see Section~\ref{sec:symmetrygroups}.
On the side of framed differentials, every component of $\FQuad(\colsur)$
is realized as a component of $\FQuad^\bullet(\colsur)$ for appropriate
choices of initial triangulations. For spaces of stability conditions however
we make no claims on the (non)-existence of spurious components of
$\Stab(\D(\colsur))$ not covered as target of our correspondence, just as
this question is left undecided in~\cite{BS15} for simple zeros.

\subsection{Techniques}\label{subsec_techniques} The proof of
Theorem \ref{intro:KrhoBihol}, given in Section \ref{sec_iso}, shares
with the original proof by Bridgeland and Smith the idea of extending a
chamber-wise identification. The main differences are an explicit isomorphism of
exchange graphs on both sides and a generalization of the method to extend beyond
the 'tame locus', as we now explain.
\par
Both $\FQuad(\colsur)$ and $\Stab(\D(\colsur))$
come with a natural chamber structure. In $\FQuad(\colsur)$ the open chambers
are given by quadratic differentials without horizontal saddle connections. The
trajectory structure of the differential gives rise to an arc system that we
call \emph{$\w$-mixed-angulation}, generalizing the triangulations in the simple
zero case. The preimage of the mixed-angulation under the collapse is called
a \emph{partial triangulation} of~$\surfo$. Adjacency of chambers is encoded
by a notion of forward flip of the partial triangulation and leads to the
definition of an exchange graph $\EG(\colsur)$. On the other side,
$\Stab(\D(\colsur))$ is also tiled in chambers identified by the heart of
a bounded t-structure the stability conditions are supported on. The first
step consists of studying and comparing these chamber structures.
\par
\paragraph*{\textbf{Comparison of exchange graphs}} We start from a distinguished
heart of $\calD(\colsur)$ and need to consider the exchange graph $\EG(\D(\colsur))$
whose vertices are hearts of bounded t-structures and whose arrows are simple tilts
(recalled in Section \ref{sec:prelim}). The idea is to relate
(parts of) $\EG(\colsur)$ and $\EG(\D(\colsur))$. When $\mathbf{w} \equiv \mathbf{1}$
this is the relation
between triangulations and finite hearts of bounded t-structures of a $\CY_3$ Ginzburg
category.
\par
Recall that $\D(\colsur)$ is by definition a Verdier quotient of a $\CY_3$
category $\calD_3(\surfo)$.
While a partial triangulation can always be refined
to a triangulation, a general expectation is that \emph{not} all hearts in $\D(\colsur)$ arise as quotients
of hearts in~$\Dsan$. When they do, we call them hearts \emph{of quotient type}. We restrict to the principal part $\EGb(\colsur)$ of $\EG(\colsur)$
whose vertices are those partial triangulations that can be refined to a
successive flip of a triangulation~$\TT$ fixed once and for all.
Correspondingly, $\EGb(\D(\colsur))$ includes precisely hearts that are quotients
of tilts of the heart of $\calD_3(\surfo)$ associated to~$\TT$ under the
original ($\mathbf{w}\equiv\mathbf{1}$) correspondence.  The definition and study of
these graphs covers Sections~\ref{sec_subsurface} and~\ref{sec:cat sub/col} and
leads to the isomorphism
\be\label{intro_EG_iso} \EGb(\colsur) \cong \EGb(\D(\colsur)),
\ee
stated as Theorem~\ref{thm:EGiso}. It allows us to define the map~$K$ of
Theorem \ref{intro:KrhoBihol} on the
complement~$B_2$ of the locus of differentials with more than one horizontal saddle connection.
\par
The viewpoint of refining partial triangulations to triangulations makes
it also clear why quotient categories naturally arise in this context: Any
two triangulation refinements of a partial triangulation differ by
successive flips in the additional edges
and we show in Proposition~\ref{equal_quot_heart} that the resulting quotient
heart is independent of these choices. A consequence of the main result and \eqref{intro_EG_iso} is that the union of $\bC$-orbits of hearts of
quotient  type of $\calD(\colsur)$ form
connected components of $\Stab(\D(\colsur))$.
\par
\paragraph*{\textbf{Walls have ends}} Finally we need extend~$K|_{B_2}$
to all of $\FQuad^\bullet(\colsur)$, which is stratified  by the number of closed
saddle connections and
recurrent trajectories. We generalize in Section~\ref{sec:wallsends}
the argument in \cite{BS15} that each component of all the higher order strata
$B_p$  of $\FQuad(\colsur)$ has 'ends' where it locally does not disconnect the
complement.
Our argument gives an alternative proof that does not depend on case distinctions
of local configurations of hat-homologous saddle connections. Those configurations
probably become hard to list as the orders of zeros in~$\w$ grow. As a downside,
our approach avoids classifying the moduli spaces of objects in $\D(\colsur)$ that
are stable and of phase zero in a given stability condition~$\sigma$,
compare \cite[Theorems~1.4 and~11.6]{BS15}. Due to the connection with computing
BPS-invariants in the $\CY_3$ context, it seems interesting to analyze this further.

\subsection{The category $\calD(\colsur)$}
The category $\calD(\colsur)$ is defined as the quotient of a $\CY_3$ triangulated
category $\calD(\surfo):=\pvd(\Gamma(Q,W))$ by a subcategory of the same
form $\Dsub:=\pvd(\Gamma(Q_I,W_I))$, where $(Q_I,W_I)$ is a subquiver
of the quiver with potential $(Q,W)$ defined by the combinatorial data of a quadratic
differential. As opposed to~$\calD(\surfo)$, the quotient category  is in general
not Calabi-Yau and not $\Hom$-finite, yet we need to consider its bounded t-structures.
Proposition~\ref{prop:ind.tilting} and Theorem~\ref{thm:EGiso} in
Section~\ref{sec:cat sub/col}, beyond proving the isomorphism of
the graphs~\eqref{intro_EG_iso}, tell us about the possibility to lift a simple
tilt in the quotient $\pvd(\Gamma(Q,W))/\pvd(\Gamma(Q_I,W_I))$ to a simple tilt
on $\pvd(\Gamma(Q,W))$ and viceversa. A comprehensive description of these categories and their t-structures
will appear in a subsequent paper.
\par

\par
\subsection{Exchange graphs and connected components, examples}
The classification of connected components of spaces of abelian or quadratic
differentials has attracted a lot of attention, and similarly the question
whether spaces of stability conditions are connected is an important
question in the topic. We give a short overview over the literature. For
differentials, there are two classification questions. For (plain, unframed)
differentials, the first result is by Kontsevich-Zorich (\cite{kozo1}) for
holomorphic abelian differentials, followed by Lanneau (\cite{lanneau}) for
holomorphic quadratic differentials. Boissy first classified components for
meromorphic abelian differentials (\cite{boissymero}). See work of
Chen-Gendron \cite{CGkdiff} for the latest results. Equally interesting and
challenging is the classification of (Teichmüller-)framed differentials, see
\cite{KQ2} for simple zero and higher order pole case and work of
Walker (\cite{walker}) and Calderon-Salter (\cite{CShigherspin}) for the latest
results in the holomorphic case.
In almost all known cases components are classified by spin invariants,
hyperellipticity and torsion conditions in genus one, some low genus strata
of quadratic differentials providing exceptions (see \cite{lanneau} and also
\cite{CM14}).
\par
For spaces of stability conditions the stability manifold is known to be connected
(and simply connected) for instance for the bounded derived category of
curves (Okada \cite{okada} for genus $g=0$ and Macr\`i \cite{MacriCurves} for
higher genus) or of some abelian surfaces and very general $K3$ surfaces (\cite{HMS}).
An example for a non-connected space of stability conditions is given by
Meinhardt and Partsch \cite{sven}. They study the quotient category $\calD^b_{(1)}(X)$
of the bounded derived category $\calD^b(X)$ on a smooth projective variety~$X$
with  $\dim(X)\geq 2$ by the full subcategory of complexes of sheaves supported in
codimension $c>1$. The classification of components is based on
computing $\widetilde{\GL}_2^+(\R)$ orbits of $\Stab(\calD^b_{(1)}(X))$.
\par
The walls-have-ends result Corollary~\ref{cor:pathB2} implies that connectivity of
spaces of quadratic differentials is equivalent to the connectivity of the
corresponding exchange graphs. Via our main theorem this gives a criterion to
show if the spaces $\Stab^\bullet(\D(\colsur))$ are disconnected. In fact,
Example~\ref{torusexample} together with the isomorphism~\ref{intro_EG_iso} and
Theorem~\ref{intro:KrhoBihol} shows disconnectivity for example already
for the space of quadratic differentials with a zero of order three and a triple pole.
\par
\begin{cor} For a surface~$\colsur$ of genus one with one boundary component
and one zero with weight~$\w = (3)$  the principal part $\Stab^\bullet(\D(\colsur))$
is disconnected.
\end{cor}

\subsection*{Acknowledgments} This projects has benefited from many inspiring
discussions, and we therefore thank Dylan Allegretti, Tom Bridgeland,
Jon Chaika, Xiaowu Chen, Yitwah Cheung, Merlin Christ, Haibo Jin,
Francesco Genovese, Fabian Haiden, Zhe Han, Bernhard Keller,
Paolo Stellari, Alex Wright, Dong Yang, and Yu Zhou.
Special thanks go to Ivan Smith and Dawei Chen for helping us to join initially
independent overlapping projects.
\par


\section{Preliminaries on categories and the stability manifold}\label{sec:prelim}

In this section we set some notation and collect background material about
stability conditions on triangulated categories, quotient categories, quivers
with potential and the $\CY_3$-categories associated to them. References are
\cite{tiltingbook, gelfandmanin,f-pervers, BrStab,bridgeland_survey, neeman,
DWZ,keller11}.
\par\medskip
\paragraph{\textbf{Notation}}
We fix $\acf$ an algebraically closed field for simplicity.
All categories considered in this paper are $\acf$-linear and all subcategories
are full. For an additive category $\calC$ with a subcategory (or set of objects)
$\calB$, we define
\[
    \calB^{\perp_\calC}\,:=\,\left\{C\in\calC:\Hom_\calC(B,C)=0\ \forall\,B\in \calB\right\},
\]
and similarly ${^{\perp_\calC}}\calB$. We will  omit the subscript~$\calC$ when
there is no confusion. If $\calC$ is triangulated, we denote by $\thick(\cl B)$ the smallest
thick additive full subcategory in $\calC$ containing $\cB$. It is triangulated.
Moreover, for full subcategories $\calH_1,\calH_2$ of an abelian or a triangulated
category~$\calC$, we let
\begin{align*}
    \calH_1*\calH_2&:=\{M\in\calC\mid\text{$\exists$ s.e.s or triangle $T\to M\to F$ s.t. $T\in\calH_1, F\in\calH_2$} \},\\
	\bra \cl B\ket&:=\{M\in\calC\mid\text{$\exists$ s.e.s or triangle $T\to M\to F$ s.t. $T,F\in\mathrm{Add}\, \cl B$} \}.
\end{align*}
Consequently  $\calH_1*\calH_2\subset \langle \calH_1,\calH_2\rangle \supset \calH_2*\calH_1$.
If $\calH_1,\calH_2$ satisfy $\Hom(\calH_1,\calH_2)=0$, then we write
$\calH_1\perp\calH_2$ for $\calH_1*\calH_2$.
\par
A finite length abelian category will be said to be \emph{finite} if it has finitely
many simple objects. Throughout the paper we will use the following complex half-planes
\ba \halfplane&:=\left\{\rho e^{\pi i \theta}\,|\, \rho\in \R_{>0},\ 0<\theta <1\right\}, \\
\chalfplane&:=\left\{\rho e^{\pi i \theta}\,|\, \rho\in \R_{>0},\ 0< \theta \leq1\right\}.
\ea
\subsection{Structures on triangulated and abelian categories}\label{sec:cats}
Here we collect some background material about bounded t-structures, stability conditions
on triangulated categories, and quotients of abelian and triangulated categories.
\par
\medskip
\paragraph{\textbf{Bounded t-structures and simple tilts}}
A \emph{t-structure} on a triangulated category $\calD$ is
the torsion part of a torsion pair (so that $\calD=\cl P \perp \cl P^\perp$) satisfying $\cl P[1]\subset \cl P$.
The $t$-structure is said to be \emph{bounded} if $\calD=\cup_{m\in\Z}\cl P[m]\cap\cl P^\perp[-m]$.
The \emph{heart} of a bounded t-structure $\cl P\subset\calD$ is the full subcategory $\cl H=\cl P\cap\cl P^\perp[1]$, which is abelian.
Denote by $K(\h)\simeq K(\calD)$ their Grothendieck groups.
The bounded t-structure and its heart determine each other uniquely and hence we will use them interchangeably.
\par
Given a torsion pair $(\calT,\calF)$ in the abelian heart of a bounded t-structure $\cl H= \calT\perp\calF$,
there is a new heart $\mu_{\torsionfree}^\sharp\h:= \torsion \perp_\calD \torsionfree[1]$,
known as \emph{forward tilt with respect to $(\torsion,\torsionfree)$} of $\h$,
see e.g.~\cite{tiltingbook}.
Obviously, tilting commutes with autoequivalences, i.e., for any $\Phi\in\Aut(\calD)$,
\begin{equation}\label{phi_mu}
\Phi\left(\mu_\torsion^{\sharp}(\h)\right)
	\= \mu^{\sharp}_{\Phi(\torsion)} \Phi(\h)\,.
\end{equation}
\par
A forward tilting $\h\to\h^\sharp$ is simple if the corresponding torsion free class $\torsionfree$ is generated by a simple $S$ of $\h$, i.e. $\torsionfree=\<S\>$.
For a finite heart~$\h$ with a rigid simple~$S$ the simple forward tilt with respect
to~$S$ exists and is denoted by $\mu_{S}^{\sharp}\h$.
Moreover, by a tilting formula in \cite[Proposition~5.4]{kq}, the new simples are
$\Sim\mu^\sharp_S\h=\{S[1]\}\;\cup\;\{ \psi^\sharp_S(X)\mid X\in \Sim\h,X\neq S\}$
where
\begin{eqnarray} \label{newsimples}
    \psi^\sharp_S(X) \=
    \Cone\left(X\xrightarrow[]{f} S[1]\otimes\Ext^1(X, S)^* \right) [-1]\,.
\end{eqnarray}
\par
Recall that the partial order on hearts $\h_1\le\h_2$ means $\calP_1\supset\calP_2\Leftrightarrow\calP_1^\perp\supset\calP_2^\perp$.
The following is a characterization for all hearts in the interval $[\h,\h[1]]$.
\par
\begin{lemma}\cite[Remark~3.3]{kq}\label{lem:nearby_heart}
Fix a heart~$\h$.  Then a heart $\h'$ is a forward tilt of~$\h$ if and only if $\h
\le \h' \le \h[1]$. In this case, the tilting is with respect to the
torsion pair $\torsion=\h' \cap \h$ and $\torsionfree=\h'[-1] \cap \h$.
\end{lemma}
\par

\medskip\par
\paragraph{\textbf{Stability structures}}
\par
A \emph{stability function} on an abelian category $\h$ is a group homomorphism $Z:K(\h)\to \bC$,
such that for any $0\neq A\in\h$, $Z(A)\in \chalfplane$.
An object $A\in\h$ is said to be \emph{$Z$-semistable} if for any non-zero proper sub-object $B\hookrightarrow A$ then $\frac{1}{\pi}\arg Z([B])\leq \frac{1}{\pi}\arg Z([A])$. It is called \emph{stable} if the inequality holds strictly.
The quantity $\frac{1}{\pi}\arg Z(A)$ is called the \emph{phase} of $A$. A stability function is called a \emph{central charge} if it moreover satisfies the so-called
\emph{support property} and \emph{Harder-Narasimhan property}, see \cite{BrStab,ks,BMS}
for more details.
\par
A \emph{stability condition} $\sigma$ on $\calD$ is a pair $\sigma=(Z,\h)$,
consisting on a heart $\h$ together with a central charge $Z\in\Hom(K(\h),\bC)$.
Let
\[
    \cl P(\phi) := \{E[\lfloor \phi \rfloor]  \; | \; E \;
        \text{is}\; Z\text{-semistable in } \; \h \; \text{of phase } \phi - \lfloor \phi \rfloor \}\ , \quad \forall \phi\in\mathbb{R}
    \]
be the \emph{slice} consisting of semistable objects of phase $\phi$.
The collection of all slices is known as a slicing, denoted by
    $\cl P_\bR:=\{\cl P(\phi)\}_{\phi\in\R}\subset \calD$.
Denote by $\cl P(I)=\< P(\phi)\mid \phi\in I \>$ for any interval $I\subset \R$.
The heart of a stability condition $\sigma$ is $\cl H=\cl P(0,1]$, and the data of a stability conditions $\sigma=(Z,\h)$ is equivalent to a pair $(Z,\cl P_{\bR})$
with certain compatibility conditions, see \cite{BrStab}. We recall the main result 
of \cite{BrStab}:
\par
\begin{theorem}\label{prop_stab_mfd}
The set of all stability conditions on $\D$ form a complex manifold $\stab(\calD)$ with local coordinates given by
the central charge $Z\in\Hom(K(\calD),\bC)$.
\end{theorem}
\par
The group of autoequivalences $\Aut(\calD)$ acts on the left on $\stab(\calD)$ by
\[
    \Phi.(Z,\h)= \left( Z\circ\Phi^{-1}, \Phi(\h)\right),
\]
which commutes with the action by scalars, for any $\lambda \in \bC$:
\be \label{eq:Caction}
\lambda(Z,\cl P)\=(e^{-\pi i \lambda}Z, \cl P') \quad \text{where} \quad \cl P'(\phi)
\=\cl P(\phi + \Re \lambda).
\ee
\par
\medskip
\paragraph{\textbf{Generic-finite connected components}}
We denote by $\cube(\h)$ the locus of stability conditions supported on a heart~$\h$
and by $\cube_0(\h)$ its interior. If $\cl H$ is finite, then $\cube(\cl H)=\chalfplane^{|\Sim(\cl H)|}$.
\par
\begin{definition} \label{def:wallsandchambers}
Let $\cl H$ be a finite heart, and $\cl W_S(\h) \subset\cube(\h)$ be the real
codimension~1 subset for which a simple~$S$ has phase 1 and all other simples
in $\h$ have phase in $(0,1)$. We call $\cl W_S(\h)$ a \emph{wall}. A \emph{chamber}
is a connected component of the complement of the closure of the union of
the walls in $\stab(\calD)$.
\end{definition}
\par
If $\cl H_1$ and $\cl H_2$ are finite hearts, the intersection $\cube(\calH_1)
\cap\overline{\cube(\calH_2)} \= \cl W_S(\calH_1)$ if and only if $\calH_2
\=\mu^\sharp_S\calH_1$ by~\cite{woolf1}. We let
\be
    \Stab_0(\calD) \= \bigcup_{\calH\,\text{finite}} \cube_0(\calH)
    \quad\text{and}\quad
    \Stab_2(\calD) \= \bigcup_{\calH\,\text{finite}} \cube(\calH).
\ee
The indexing convention is parallel with the one that we will use for spaces of
quadratic differentials in Section~\ref{sec:wallsends}.
If a connected component of $\Stab(\calD)$ has been specified, we decorate these spaces
by a~$\circ$ accordingly.
\par
\begin{definition}\label{def:tame}
A connected component $\Stap(\calD)$ is called
\begin{itemize} [nosep]
	\item a \emph{finite type component}, if $\Stap\D=\Stap_2\D$;
	\item a \emph{generic-finite type component}, if $\Stap\D=\bC\cdot\Stap_2\D$.
	\end{itemize}
\end{definition}
\par
\medskip
\paragraph{\textbf{Abelian and triangulated quotient categories}}
Recall that a subcategory $\cl S$ of an abelian category $\cA$ is called
a \emph{Serre subcategory} if it is abelian and for any short exact sequence
$0 \to A_1 \to E \to A_2 \to 0$
in $\cA$, we may conclude $E \in \cl S \text{ if and only if }A_1,A_2 \in \cl S$.
In such a case, the quotient category $\cA/\cl S$ is also abelian, cf.
\cite[{Lemma~A.2.3}]{neeman} and \cite{gabriel}.

\par

On the other hand, given a triangulated category $\calD$ and a triangulated subcategory
$\calV\hookrightarrow \calD$, we can construct the so-called
\emph{Verdier quotient} $\calD/\calV$.
If $\calV$ is thick (i.e. closed under direct summands), then
\[0 \to \calV \to \calD \to \calD/\calV \to 0\]
is a short exact sequence of triangulated categories with exact
functors, \cite[Proposition~2.3.1]{V}.

\begin{rmk*} Whenever $\pi:\calD\to \calD/\cl V$ is a quotient functor of triangulated categories, and $\cl B$ is a subcategory of $\calD$, by $\pi(\cl B)$ we will mean the \emph{essential image} of $\cl B$ through~$\pi$. This will apply in particular to the image of abelian hearts $\h\subset\calD$.\end{rmk*}

\subsection{Quivers with potential, mutation and Jacobian algebras}
\label{sec:backquivers}
In this article $(Q,W)$ is a non-degenerate finite (possibly disconnected) oriented quiver $Q=(Q_0,Q_1,s,t)$ that has no loops or 2-cycles, with potential $W$ considered up to right equivalence. The cyclic derivative of $W$ with respect to an arrow~$a\in Q_1$ is denoted~$\partial_a W$. We refer to \cite{DWZ} for all these basic notions.
\par
The operation of \emph{mutation at a vertex $i\in Q_0$}, defined for instance in \cite[Section 2]{KY}, produces a new quiver with potential denoted $\mu_i(Q,W)$ or $(\mu_i Q,\mu_i W)$, that will have no loops nor 2-cycles.
\par
The Jacobian algebra $\cl J(Q,W)$ of $(Q,W)$ is the quotient of $\widehat{\acf Q}$, the completion of the path algebra with respect to bilateral ideals generated by arrows, by the ideal $\del W:=\bra \del_a W \mid a\in Q_1\ket$.
The category of finite dimensional modules over $\cl J(Q,W)$ is denoted
\[\h(Q,W):=\modules \cl J(Q,W),\]
sometimes shortened to $\h_Q$ when the quiver with potential is clear.
It is a finite length, finite, abelian category, see for instance \cite[{Section 3}]{keller_dilog}.
The vertices of the quiver give the simple objects of the module category $\Sim(\h_Q)$, 
so that in particular the  Grothendieck group is $K(\h(Q,W))=\Z^{|Q_0|}.$
\par\smallskip
Let $I\subset Q_0$ be a proper subset of the set of vertices of $Q$, and $I^c$ its complement.
There is an operation of \emph{restriction} on $(Q,W)$, to get a new quiver with potential, denoted by $(Q_I,W_I)$
with vertices $(Q_I)_0=I$. It is obtained by deleting
vertices in $I^c$, arrows incoming or outgoing from
vertices in $I^c$, and all cycles in $W$ passing through
vertices in $I^c$, \cite{DWZ}. We call $(Q_I,W_I)$ a \emph{(full) subquiver}.
When $i\in I$, the operations of mutation $\mu_i$ and of restriction $\mid_I$ commutes,
cf. \cite{LabaFragQP}.
The following is obvious.
\begin{lemma}
Let $k\not\in I\subset Q_0$ be a vertex of $Q$ such that there are no arrows from~$i$
to~$k$ or from $k$ to $i$ for all $i\in I$. Then
$(\mu_k(Q,W))_I = (Q_I,W_I)$.
\end{lemma}
\par
The finite-length property of $\h(Q,W)$ immediately implies:
\begin{lem} \label{lem_Subquiv_Serre}
	There is a bijection between Serre subcategories of $\h(Q,W)$ and full
	sub-quivers $(Q_I,W_I)$.
\end{lem}
\par
We will be interested in the quotient abelian category
\begin{equation}\label{standard_quotheart}\h(Q,W)/\h(Q_I,W_I)\end{equation}
 which is a category of modules over a finite dimensional algebra as well, \cite[Propositions~2.2 and~5.3]{geigle}. In particular it is a finite length, finite abelian category. The Grothendieck group splits as
 \[K(\h(Q,W))\simeq K(\h(Q_I,W_I))\oplus K(\h(Q,W))/K(\h(Q_I,W_I)).\]

\subsection{$\CY_3$ categories associated to a quiver with potential}
\label{subsec_quiver3cat}
We denote by $\Gamma := \Gamma(Q,W)$ the \emph{complete Ginzburg differential
graded (dg) algebra} associated to a quiver with potential $(Q,W)$. The underlying
graded algebra~$\Gamma$ is the completion of the path algebra of a graded quiver
obtained from~$Q$ and the differential is given by the potential $W$. It is defined
in \cite{ginzburg,K1}. The zero-th homology of this algebra $H_0\big(\Gamma(Q,W)\big)
\simeq \cl J(Q,W)$ gives back the Jacobian algebra of the original quiver with potential.
\par
Recall that a category~$\calC$ is said to be Calabi-Yau of dimension~$N$, or
simply~$\CY_N$ if for any objects $E,F\in \calC$ there is a natural isomorphism
$\nu:\Hom_\calC(E,F)\stackrel{\sim}{\to}\Hom_\calC(F,E[N])^\vee$ of $\acf$-vector spaces.
\par
Let $\calA$ be a dg algebra with derived category $\D(\calA)$. Denote by $\per(\calA)$
and $\pvd(\calA)$ the \emph{perfect derived category} and the \emph{perfectly valued
derived category} of~$\calA$, respectively. The perfect category $\per\Gamma$ is
generated by the indecomposables projective dg modules $P_i=e_i\Gamma$, $i=1,\ldots,n$.
The perfectly valued derived category $\pvd(\Gamma)$ coincides with the
subcategory of $\D(\Gamma)$ consisting on dg modules of total finite-dimensional
homology.
\par
 We collect some well-known results from \cite{keller11} and \cite[Section~3-4]{KY}.
\par
\begin{prop} \label{prop:mutationiso}
For $\Gamma=\Gamma(Q,W)$ as above the following statements hold:
\begin{itemize}[nosep]
\item The category $\pvd(\Gamma)$ is $\Hom$-finite and $\CY_3$,
for any non-degenerated quiver with potential $(Q,W)$, and it is contained in $\per(\Gamma)$.
\item If $\Gamma'=\Gamma(Q',W')$ is obtained by mutation,
then $\pvd(\Gamma')\isom \pvd(\Gamma)$ and $\per\Gamma'\isom\per\Gamma$.
\item The category $\pvd(\Gamma)$ admits a standard heart of bounded t-structure
\[
    \calH(\Gamma)=\calH(Q,W):=\modules \cl{J}(Q,W).
\]
\end{itemize}
\end{prop}
The $\CY_2$ Verdier quotient
$\per(\Gamma)/\pvd(\Gamma)$,
sitting in a short exact sequence of triangulated categories,
\begin{gather}
    0\to \pvd(\Gamma)\to\per\Gamma\xrightarrow{\pi_\Gamma} \calC(\Gamma)\to0
\end{gather}
is called the \emph{cluster category} and denoted by $\calC(\Gamma)$, following
Amiot \cite{amiot}. 
\par
Let $(Q_I,W_I)=(Q,W)_I$ denote a full subquiver of $(Q,W)$, as in the previous subsection, and  $\Gamma_I=\Gamma(Q_I,W_I)$, $\cl J_I=\cl J(Q_I,W_I)$. The standard bounded t-structure $\h(\Gamma)$ on $\pvd(\Gamma)$ restricts
to the standard bounded t-structure
\[\modules \cl J_I=:\h(\Gamma_I)=\pvd(\Gamma_I)\cap\h(\Gamma)\subset \h(\Gamma)
\]
on the subcategory $\pvd(\Gamma_I)=\thick_{\pvd(\Gamma)}\h(\Gamma_I)
\subset\pvd(\Gamma)$.
\par
We are interested in the Verdier quotient $\pvd\Gamma/\pvd\Gamma_I$ and
in those hearts that are the images, under the quotient functor, of a heart in $\pvd \Gamma$.
We will study a component of the exchange graph of $\pvd\Gamma/\pvd\Gamma_I$ containing
the heart $\h(\Gamma)/\h(\Gamma_I)$ in Section \ref{sec:cat sub/col}.

\section{Decorated marked surfaces and quadratic differentials}
\label{sec:dmsQD}

\subsection{Quadratic differentials}

We set up notion for quadratic differentials, using the book of Strebel \cite{Strebel} as background.
\par
Let $X$ be a compact Riemann surface and $\omega_X$ be its holomorphic
cotangent bundle. A \emph{meromorphic quadratic differential~$q$} on $X$ is a
meromorphic section of the line bundle $\omega_{X}^{2}$. We denote by~$\bfz$
the collection of points where~$q$ has a pole or vanishes,
the \emph{singularities} or \emph{critical points} of~$q$. These can be
grouped into the finite critical points (zeros and simple poles) of~$q$,
and infinite critical poles (higher order poles).
We denote by $Z_j(q)$ the set of finite critical points of~$q$ or order~$j$
and $P_k(q)$ the set of poles of~$q$ with order $k\geq 2$
Finally, let $Z(q)=\bigcup_{j \geq -1} Z_j(q)$ and $P(q)=\bigcup_{j\geq 2} P_j(q)$
and group them together as $\Crit(q)=Z(q) \cup P(q)$. We let
$\w = (w_1,\ldots, w_r)$ be the orders of the finite critical points
and $\w^- = (w_{r+1},\ldots, w_{r+b})$ be the negative orders of higher
order poles (i.e.\ $w_i \leq -2$ for $i \geq r+1$). The
tuple $(\w,\w^-)$ is the \emph{signature} of the quadratic differential.
\par
\medskip
\paragraph{\textbf{The canonical covering construction.}}
Associated with a quadratic differential~$q$ on a compact curve~$X$
there is a {\em canonical double cover $\wh\pi: \wh{X} \to X$} such that
$\wh{\pi}^* q = \omega^2$ is the square of an abelian differential, unique
up to sign. See e.g.\ \cite[Section~2.1]{kdiff} for various methods of
construction. The tuple of preimages of the singularities of~$(X,q)$
is denoted by~$\wh\bfz$, and decomposed into the finite critical
points~$\wh{Z}$ and infinite critical points~$\wh{P}$. To compute the
signature of the double cover we define
\be \label{eq:typemuhat}
(\wh{\w},\wh{\w}^-)
\,:=\, \Bigl(\underbrace{\wh w_1, \ldots, \wh w_1}_{\gcd(2,w_{1})},\,
\underbrace{\wh w_2,  \ldots, \wh w_2}_{\gcd(2,w_{2})} ,\ldots,\,
\underbrace{\wh w_{r+b}, \ldots, \wh w_{r+b}}_{\gcd(2,w_{r+b})} \Bigr)\,,
\ee
where $\wh w_i := \tfrac{2+w_{i}}{\gcd(2,w_{i})}-1$ and where now
$\wh{\w}^-$ is the tuple of the negative entries among these integers.

\par
\medskip
\paragraph{\textbf{Trajectory structure}}
We now turn to the global trajectory structure of a quadratic differential~$q$,
following \cite{Strebel}. We suppose throughout that $q$ has at least one zero
and at least one \emph{infinite critical point}, i.e.\ a pole of
order~$\leq -2$ or equivalently that $\w^- \neq \emptyset$. We do not suppose
that~$q$ has simple zeros (i.e.\ we do \emph{not} work only with
Gaiotto-Moore-Neitzke (GMN) differentials).
\par
A \emph{saddle connection}
is  a trajectory (in some arbitrary direction) whose maximal domain is a finite
interval. Both its end points are zeros of~$q$. A \emph{saddle trajectory} is a saddle
connection in the horizontal direction. A trajectory is \emph{closed} if its a
saddle trajectory and both its end points coincide. The remaining trajectories are
either
\begin{enumerate}[nosep]
	\item \emph{separating}, i.e., approaching an infinite critical point at precisely
	one end,
	\item \emph{recurrent} in at least one of its directions, or
	\item \emph{generic}, approaching an infinite critical point in both directions.
\end{enumerate}
\par
We now fix the direction to be the horizontal direction unless specified
otherwise, so 'trajectories' refers to 'horizontal trajectories'.
Removing from~$X$ the separating trajectories and saddle trajectories decomposes
the surface into connected components, which are of the following types.
\begin{enumerate}[nosep]
	\item \emph{ring domains} or \emph{cylinders} that are foliated by closed
	trajectories,
	\item \emph{horizontal strips} isometric to $S = \{a < \Im(z) < b\}$
	with $q|_S =  dz^{\otimes 2}$,
	\item \emph{half-planes}, isometric to $\bH$ with $q|_{\bH} = dz^{\otimes 2}$, or
	\item \emph{spiral domains}, the interior of the closure of a recurrent trajectory.
\end{enumerate}
\par
A ring domain is called \emph{degenerate} if one of its boundary components
is a double pole. A saddle trajectory is called \emph{borderline} if it lies
on the boundary of a degenerate ring domain, half-plane, or horizontal strip.
\par
The quadratic differential~$q$ is called \emph{saddle-free} if is does not have
any saddle trajectories. By \cite[Lemma~3.1]{BS15} such a differential
does have neither closed   trajectories nor recurrent trajectories.
In particular the complement of its saddle trajectories and separatrices is
a union of half planes and horizontal strips. We call this the
\emph{horizontal strip decomposition} of $(X,q)$.
\par
Given a quadratic differential~$q$ on~$X$ we define the closed subsurface~$X^+$
to be the closure of the union of all horizontal strips, half-planes and
degenerate ring domains. The closed subsurface~$X^-$ is defined to be the
closure of the union of all spiral domains and non-degenerate ring domains.
The two subsurfaces $X^\pm$ meet along a collection of saddle connections,
all of which are borderlines. See Figure~\ref{fig:logow} for examples.
\par
If $\eta$ is a path tracing a saddle connection on~$X$, we let $\eta'$ and $\eta''$
be the two lifts of the path to~$\wh{X}$. We define the lifted class
$[\wh{\eta}] \in H_1(\wh{X} \setminus \wh{P}, \wh{Z}, \bZ)$ to be
$[\wh{\eta}] = [\eta']$ if $[\eta']+[\eta'']=0 \in H_1(\wh{X} \setminus
\wh{P}, \wh{Z}, \bZ)$ and we define $[\wh{\eta}] = [\eta']- [\eta'']$
otherwise. We declare two saddle connections $\eta_1$ and $\eta_2$ to be
\emph{hat-homologous} if for some choice of orientation the equality
$[\wh{\eta}_1] = [\wh{\eta}_2]$ holds in $H_1(\wh{X} \setminus \wh{P}, \wh{Z},
\bZ)$. We say that two saddle connections are \emph{hat-proportional} if
$[\wh{\eta}_1]$ and $[\wh{\eta}_2]$ are proportional. The characterization
in \cite[Proposition~1]{MaZo} via rigid configurations shows that saddle connections
are hat-proportional if and only if they are hat-homologous. This is the reason
for our definition of $[\wh{\eta}]$, which differs sometimes by a factor~$2$
from the one in \cite{BS15}, compare with \cite{IkedaAn}.

\subsection{Decorated marked surfaces}
\label{app:DMSqdiff}

The notion of marked surface encodes the raw combinatorics of a quadratic
differential with the limit points of trajectories at the poles and possible additional
auxiliary punctures, but without specifying order and location of the
zeros. Marked points are usually referred to as \emph{prongs} at the
poles in flat surface literature. Here we follow \cite[Section~3]{BS15} and
\cite[Section~4]{KQ2} to relate quadratic differentials and weighted marked surfaces.
\par
\begin{definition}
A \emph{marked surface} $\surf=(\surf,\MM,\PP)$ consists of a connected
bordered differentiable  surface with a fixed orientation, together with a finite
set $\MM= \bigcup_{i=1}^b M_i$ of marked point on the boundary
$\partial\surf=\bigcup_{i=1}^b
\partial_i$ and a finite set $\PP=\{p_j\}_{j=1}^p$ of punctures in its interior
$\surf^\circ=\surf-\partial\surf$, such that each connected component of
$\partial\surf$ contains at least one marked point.
\end{definition}
\par
Up to homeomorphism, $\surf$ is determined by the following data
\begin{itemize}[nosep]
\item the genus $g$;
\item the number $b$ of boundary components;
\item the number $p=\#\PP$ of punctures;
\item the negative integer partition~$\w^-$ of $-m=-\#\MM$ into $b$
parts describing the number of marked points on its boundary, and consisting in 
$w^-_i=-\#\MM_i$.
\end{itemize}
The rank of $\surf$ is defined to be
\begin{gather}\label{eq:n}
	N\=6g+3p+3b+m-6.
\end{gather}
\par
\emph{For simplicity, we only consider the $\PP=\emptyset$ case in this paper.}
\par
Decorations and weight add to a marked surface the data of
the location and orders of zeros of a differential.
\par
\begin{definition}
A \emph{decorated marked surface} (abbreviated as \emph{DMS}) is obtained
from a marked surface $\surf$ by decorating it with a set
$\Delta= \{z_i\}_{i=1}^r$ of points in the surface interior~$\surf^\circ$.
These points are called \emph{finite critical points}.
A \emph{weight} function on $\Tri$ is a $\Z_{\geq -1}$-valued function
	\[
	\w\colon \Tri\to\Z_{\geq -1}\,.
\]
We write $r = |\w|=|\Tri|$ for the number of finite critical points and
$\|\w\|=\sum_{Z\in\Tri} \w(Z)$ for their total weight. We say $\w$
is \emph{compatible} with $\surf$ if
	\begin{gather}\label{eq:cp1}
		\|\w(Z)\| - (m+2b) \=4g-4\,.
	\end{gather}
\par
If~$\w$ and $\surf$ is compatible, we will write $\sow=(\surf, \Tri, \w)$
and call this tuple a weighted DMS (abbreviated as wDMS).
\end{definition}
\par
A weight $\w$ is \emph{simple} if $\w\equiv1$. We write
$\surfo$ to indicate that we work with a wDMS with simple weight.
This is the case studied previously, e.g., in \cite{qiubraid16,Q3,QZ2,BQZ,KQ2},
and corresponds to the setting of principal strata of quadratic differentials
discussed in \cite{BS15}.
\par
\medskip
\paragraph{\textbf{Quadratic differentials on marked surfaces}}
\par
Fix a quadratic differential~$q$ and a let~$\theta \in S^1$
be a direction. A maximal straight arc (for the metric~$|q|$) in the
direction~$\theta$ is called \emph{trajectory} (in the direction~$\theta$).
Locally near a finite critical point of order~$w \geq -1$ there are $w+2$
distinguished directions that are limits of a trajectory in the
direction~$\theta$. Similarly, at a pole $p$ of order $|w|:=\operatorname{ord}_q(p)
\geq 3$ there are $|w|-2$ distinguished
directions  that are limits of a trajectory in the direction~$\theta$.
These directions are called \emph{prongs} at the zero or pole.
\par
The \emph{real (oriented) blow-up} of $(X,q)$ is the differentiable surface
$X^q$, which is obtained from~$X$ by replacing a pole $p\in P(q)$
of order at least~3  by a boundary circle~$\partial_p \cong S^1$.
Moreover, we mark the points on $\partial_p$ that correspond to the
distinguished tangent directions, so there are $\operatorname{ord}_q(p)-2$
marked points on $\partial_p$. This turns $X^q$ into a marked surface.
Adding the set of zeros $Z(q)$ together with their orders as weight
make $X^q$ into a wDMS, the \emph{weighted decorated real blow-up
of $(X,q)$}. Fixing moreover a diffeomorphism to a reference
surface gives a framing of~$(X,q)$.
\par
\begin{definition}\label{def:SFquad}
Fix a wDMS $\sow$. An \emph{$\sow$-framed quadratic differential}
$(X,q,\psi)$ is a Riemann surface $X$ with quadratic differential $q$,
equipped with a diffeomorphism $\psi\colon\sow \to\xx^q$,
preserving the marked points, decorations and their weights.
\par
Two $\sow$-framed quadratic differentials $(X_1,q_1,\psi_1)$ and $(X_2,q_2,
\psi_2)$ are isomorphic, if there exists a biholomorphism $f\colon X_1\to
X_2$ such that $f^*(q_2)= q_1$ and furthermore $\psi_2^{-1}\circ
f_*\circ\psi_1\in\Diff_0(\sow)$ is isotopic to the identity preserving marked
points, decorations and their weights (setwise). Here $f_*\colon(\xx_1)^{q_1}
\to(\xx_2)^{q_2}$ is the induced diffeomorphism
on real oriented blowups.
\end{definition}
\par
In flat surface literature this kind of framing is usually called
a (Teichm\"{u}ller) marking. To avoid confusion with the (prong) markings
used here, we stick to the terminology common to e.g.\ \cite{BS15}
and \cite{KQ2}, but we use \lq\lq Teichm\"{u}ller\rq\rq\ to refer to this kind of
marking without specifying the underlying wDMS.

\subsection{Arc systems}  \label{sec:MS}

We consider a decorated marked surfaces (wDMS)~$\sow$ with
decorations~$\Delta$, weight $\mathbf{w}:\Delta\to \Z_{\geq-1}$, and marked
points $\MM$. We let $\sow^\circ:=\sow\setminus\partial\sow$ and introduce
the following additional notation.
\begin{itemize}
\item An \emph{open arc} is an (isotopy class of a) curve $\gamma\colon I\to\sow$
such that its interior is in $\sow^\circ\setminus\Tri$ and its endpoints
are in the set of marked points~$\MM$.
\item A \emph{closed arc} is a curve $\eta\colon I\to\sow$ such that its interior
is in $\sow^\circ\setminus\Tri$ and its endpoints are in the set of
decoration points $\Tri$. (To memorize: The interval that maps to~$\sow^\circ$
is \emph{closed}.)
\end{itemize}
\par
For the simply decorated case, i.e.\ for $\w\equiv\mathbf{1}$, we 
denote
by $\CA(\surfo)$ the set of closed arcs on~$\surfo=\surf_{\w\equiv\mathbf{1}}$ that have no
self-intersections, not even at the endpoints in $\Tri$. Similarly,
let $\OA(\surfo)$ be the set of open arcs of $\surfo$. Throughout this paper $\gamma$'s denote open arcs  and $\eta$'s
denote closed arcs, unless stated otherwise.
\par
An \emph{(open) arc system} $\AS = \{\gamma_i\}$ is a collection of open arcs
on $\sow$ such that there is no (self-)intersection between any of them
in~$\sow^\circ$. Open arc systems first appeared for triangulations of simply weighted marked
surfaces (i.e., $\w\equiv\mathbf{1}$). A \emph{triangulation~$\TT$} of $\surfo$ is a maximal arc
system of open arcs, which in fact divide $\surfo$ into triangles.
Two triangulations are related by a \emph{flip} if they only differ by one arc.
Locally, the two arcs involved in a flip are the two diagonals of a square.
\par
We now move on to the weighted version of this notion. The motivation for
the notion is provided in Section~\ref{sec:MixedAngfromQDiff},
compare also with \cite{Kr}.
\par
\begin{definition} \label{def:mixedang}
A $\w$-\emph{mixed-angulation} of $\sow$ is a
collection of open arcs that divides $\sow$ into once-decorated polygons,
such that each decoration $z$ with weight $w = w(z)$ is in a $(w+2)$-gon. We
denote this $(w+2)$-gon by $\AS(z)$ and call it an \emph{$\AS$-polygon}.
\par
The \emph{forward flip} on $\w$-{mixed-angulation} $\AS$, with respect to
an arc $\gamma\in\AS$, is an operation that moves the endpoints of $\gamma$
anti-clockwise along two adjacent sides of the $\AS$-polygons containing $\gamma$,
cf. Figure~\ref{fig:logow}.
\end{definition}
\par

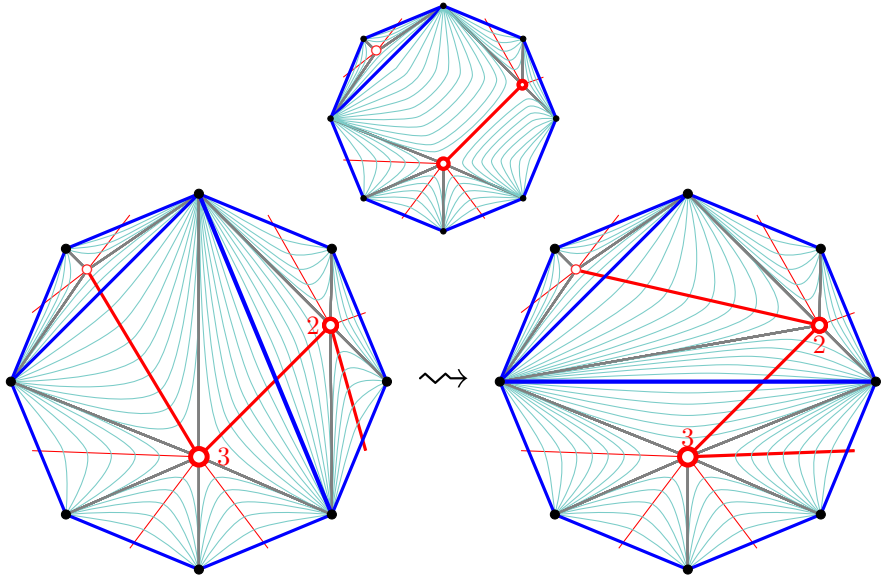
\begin{figure}[ht]\centering
	\makebox[\textwidth][c]{
		\begin{tikzpicture}[scale=.5]
			\foreach \j in {0,...,7}{\draw
				(90+45*\j:5) coordinate  (w\j)
				(45+45*\j:5) coordinate  (u\j);}

			\path ($(w0)!.5!(w2)$) coordinate (v1)
			($(w0)!.5!(w5)$) coordinate (v2);
			\draw[red]($(w0)!.7!(w6)$) coordinate(q2)
			(90-45*5+45:2) coordinate(q3){}
			(90+45:4.2)  coordinate(q1);
			\foreach \j in {0,...,7}{\draw[red](22.5+45*\j:4.8) coordinate (z\j);}
			\draw[red]
			(q2)edge(z0)edge(z1)edge(z7)
			(q3)edge(z6)edge(z5)edge(z4)
			(q1)edge(z2)edge(z3);\draw[red,very thick](z7)to(q2);
			
			\foli{.2}{w1}{q1}{w2}{.6};
			\foli{.2}{w1}{q1}{w0}{.6};
			\foli{.33}{w0}{q1}{w2}{.66};
			
			\foli{.1}{w0}{q3}{w2}{.9};
			\foli{.15}{w3}{q3}{w2}{.75};
			\foli{.15}{w3}{q3}{w4}{.75};
			\foli{.15}{w4}{q3}{w5}{.75};
			\foli{.12}{w5}{q3}{w0}{.72};
			
			\foli{.15}{w0}{q2}{w7}{.75};
			\foli{.2}{w7}{q2}{w6}{.6};
			\foli{.15}{w6}{q2}{w5}{.75};
			\foli{.15}{w5}{q2}{w0}{.75};
			
			\draw[red,very thick](q1)\ww to(q3)\ww to(q2)\ww;
			\draw[red,ultra thick,fill=white](q3)circle(.23);
			\draw[red,ultra thick,fill=white](q2)circle(.17);

			\draw(6.5,0)node{\Huge{$\rightsquigarrow$}}(14,0)node{};

			\draw[red,ultra thick,fill=white](q3)circle(.21) node[right]{$\;3$};
			\draw[red,ultra thick,fill=white](q2)circle(.17) node[left]{$2$};
			\draw[blue,very thick](w0)to(w2);
			\draw[blue,ultra thick](w5)to(w0);
			\foreach \j in {0,...,7}{\draw[blue,very thick](w\j)to(u\j);}
			\foreach \j in {1,...,8}{\draw[fill=black](90+45*\j:5)circle(.12);}

			\begin{scope}[shift={(13,0)}]
				\foreach \j in {0,...,7}{\draw
					(90+45*\j:5) coordinate  (w\j)
					(45+45*\j:5) coordinate  (u\j);}
				\path ($(w0)!.5!(w2)$) coordinate (v1)
				($(w0)!.5!(w5)$) coordinate (v2);
				\draw[red]($(w0)!.7!(w6)$) coordinate(q2)
				(90-45*5+45:2) coordinate(q3){}
				(90+45:4.2)  coordinate(q1);
				\foreach \j in {0,...,7}{\draw[red](22.5+45*\j:4.8) coordinate (z\j);}
				\draw[red]
				(q2)edge(z0)edge(z1)
				(q3)edge(z6)edge(z5)edge(z4)edge(z7)
				(q1)edge(z2)edge(z3);\draw[red,very thick](z7)to(q3);
				
				\foli{.2}{w1}{q1}{w2}{.6};
				\foli{.2}{w1}{q1}{w0}{.6};
				\foli{.33}{w0}{q1}{w2}{.66};
				
				\foli{.15}{w6}{q3}{w2}{.75};
				\foli{.1}{w3}{q3}{w2}{.8};
				\foli{.15}{w3}{q3}{w4}{.75};
				\foli{.15}{w4}{q3}{w5}{.75};
				\foli{.1}{w5}{q3}{w6}{.8};
				
				\foli{.15}{w0}{q2}{w7}{.75};
				\foli{.2}{w7}{q2}{w6}{.6};
				\foli{.1}{w0}{q2}{w2}{.9};
				\foli{.15}{w2}{q2}{w6}{.75};
				
				\draw[red,very thick](q1)\ww to(q2)\ww to(q3)\ww;
				\draw[red,ultra thick,fill=white](q3)circle(.23);
				\draw[red,ultra thick,fill=white](q2)circle(.17);

				\draw[red,ultra thick,fill=white](q3)circle(.21) node[above]{$3$};
				\draw[red,ultra thick,fill=white](q2)circle(.17) node[below]{$2$};
				\draw[blue,very thick](w0)to(w2);
				\draw[blue,ultra thick](w2)to(w6);
			\foreach \j in {0,...,7}{\draw[blue,very thick](w\j)to(u\j);}
			\foreach \j in {1,...,8}{\draw[fill=black](90+45*\j:5)circle(.12);}
			\end{scope}
			
			\begin{scope}[shift={(6.5,7)},scale=.6]
				\foreach \j in {0,...,7}{\draw
					(90+45*\j:5) coordinate  (w\j)
					(45+45*\j:5) coordinate  (u\j);}
				\path ($(w0)!.5!(w2)$) coordinate (v1)
				($(w0)!.5!(w5)$) coordinate (v2);
				\draw[red]($(w0)!.7!(w6)$) coordinate(q2)
				(90-45*5+45:2) coordinate(q3){}
				(90+45:4.2)  coordinate(q1);
				\foreach \j in {0,...,7}{\draw[red](22.5+45*\j:4.8) coordinate (z\j);}
				\draw[red]
				(q2)edge(z0)edge(z1)
				(q3)edge(z5)edge(z4)edge(z6)
				(q1)edge(z2)edge(z3);
				
				\foli{.2}{w1}{q1}{w2}{.6};
				\foli{.2}{w1}{q1}{w0}{.6};
				\foli{.33}{w0}{q1}{w2}{.66};
				
				\foli{.15}{w3}{q3}{w2}{.75};
				\foli{.15}{w3}{q3}{w4}{.75};
				\foli{.15}{w4}{q3}{w5}{.75};
				
				\foli{.15}{w0}{q2}{w7}{.75};
				\foli{.2}{w7}{q2}{w6}{.6};

			    \draw[blue,very thick](w0)to(w2);				

				\draw[Emerald!50]
				($(w0)!.5!(w2)$) coordinate (w)
				\foreach \m in {.09,.18,...,.81,.88}{
					plot [smooth,tension=.4] coordinates
					{(w2) ($(w)!\m!(q3)$) ($(w)!\m!(q2)$) (w0)}
				};
				\draw[Emerald!50]
				($(w6)!.5!(w5)$) coordinate (w)
				\foreach \m in {.4,.5,...,.9}{
					plot [smooth,tension=.4] coordinates
					{(w5) ($(w)!\m!(q3)$) ($(w)!\m!(q2)$) (w6)}
				}
				\foreach \m in {.1,.2,.3}{
					plot [smooth,tension=.4] coordinates
					{(w5) ($(w)!\m!(q3)$) (w6)}
				};
				\draw[red,very thick](q1)\ww (q3)\ww to(q2)\ww;
				\draw[red,ultra thick,fill=white](q3)circle(.23);
				\draw[red,ultra thick,fill=white](q2)circle(.17);

			\foreach \j in {0,...,7}{\draw[blue,very thick](w\j)to(u\j);}
			\foreach \j in {1,...,8}{\draw[fill=black](90+45*\j:5)circle(.12);}
			\end{scope}
	\end{tikzpicture}}
	\caption{The figure shows several local horizontal strip decompositions on $\sow$ with fixed weighted decorations, depending on a quadratic differential $q$. Here
			the black vertices are marked points on $\partial\sow$,
			the red vertices are weighted zeros of $q$,
			the green arcs are geodesics,
			the black arcs are separating trajectories.
			The blue lines define $\w$-mixed-angulations of $\sow$.
	The red solid arcs are simple saddle connections,
	(except for the thick one in the top small octagon, which is a saddle trajectory) and represent the duel graphs of the $\w$-mixed-angulations.
	The picture in the middle represents crossing a wall of second kind, resulting in a forward flip. 
	}
	\label{fig:logow}
\end{figure}

\par
Although the definition allows for 1-gon and 2-gons, we will
consider decorations with weight at least one only (i.e.\ $k$-gons
with $k \geq 3$) in accordance with the standing assumption from the
introduction and the one in Section~\ref{sec:collapse} below.
\par
When $\sow=\surfo$ has
simple weight $\w \equiv \mathbf{1}$, the $\w$-mixed-angulations
are (decorated) triangulations of $\surfo$. 
We recall here that a quiver $Q_\TT$ (without loops or 2-cycles) with a potential $W_\TT$ can be associated to a
triangulation $\TT$ of a simply decorated marked surface as follows:
\begin{itemize} [nosep]
	\item the vertices of $Q_\TT$ correspond to the open arcs in $\TT$;
	\item the arrows of $Q_\TT$ correspond to (anticlockwise) oriented intersection between open arcs in $T\TT$, so that there is a 3-cycle in $Q_\TT$ locally in each triangle as shown in Figure~\ref{fig:3-CYcle}.
	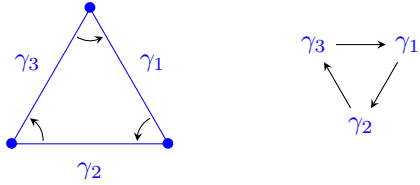
\begin{figure}\centering
		\begin{tikzpicture}[scale=1.2,>=stealth]
			\foreach \j in {1,2,3}{
				\draw[blue](90+120*\j:1)\nn to (-30+120*\j:1)\nn;
				\draw[blue](150-120*\j:.8)node{$\gamma_\j$};
				\draw[->](90+120*\j+12:.7) to[bend left=-30] (90+120*\j-12:.7);
			}
			\begin{scope}[shift={(3,.3)}]
				\foreach \j in {1,2,3}{
					\draw[blue](150-120*\j:.6)node(v\j){$\gamma_\j$};
				}
				\draw[](v1)\jiantou(v2)(v2)\jiantou(v3)(v3)\jiantou(v1);
			\end{scope}
		\end{tikzpicture}
		\caption{Local 3-cycle associated to a triangle of $\TT$}
		\label{fig:3-CYcle}
	\end{figure}
	\item the potential $W_\TT$ is the sum of all 3-cycles that locally coming from a triangle of $\TT$ as above.
\end{itemize}
The corresponding Ginzburg algebra $\Gamma(Q_\TT,W_\TT)$ will usually be
denoted  by~$\Gamma_\TT$.


\section{Moduli spaces of quadratic differentials}
\label{sec:geoqdiff}

The main goal is the proof that \lq\lq walls have ends\rq\rq\
in Proposition~\ref{prop:wallhaveends}, which is used in
Corollary~\ref{cor:pathB2} to homotope paths into the locus of so-called tame
differentials, inside the moduli space of quadratic differentials. The proof
of our main result, Theorem~\ref{thm:KrhoBihol}, relies on this
corollary. These two results are generalization of results in
\cite{BS15} where those statements are proven for differentials with
simple zeros only. Their proof relies crucially on this hypothesis in
\cite[Lemma~5.1]{BS15}. Our proof avoids most of the discussion of configuration
of hat-homologous saddle connections and uses more elaborate ways to deform
half-translation surfaces instead. We start by recalling some moduli spaces
of quadratic differentials.

\subsection{Space of quadratic differentials}
\label{sec:qdiffgeneral}

Following flat surface literature we let $\qmoduli[g,r+b](\w,\w^-)$ be the moduli
space of quadratic differentials~$(X,\bfz,q)$ on a pointed curve $(X,\bfz)$
where $\bfz = (z_1,\ldots,z_{r+b})$ such that~$q$ has signature $(\w,\w^-)$.
We emphasize that in this space the critical points are labeled. The unlabeled
version is denoted by $\qmoduli[g](\w,\w^-)$, i.e., without the subscript for
the number of labeled points. Since every quadratic differential is compatible
with a (unique, up to diffeomorphism) wDMS~$\sow$, which encodes both
the zeros (via weight) and the polar part of the signature (via the marking),
we also use the notation $\Quad(\sow) = \qmoduli[g](\w,\w^-)$. Next we discuss
several types of framed moduli spaces. Note that $\Quad(\sow)$ is in general
an orbifold and non-connected. The (finite) number of components is classified
in some cases in \cite{CGkdiff}.
\par
\medskip
\paragraph{\textbf{Framings by periods}} Spaces of quadratic differentials
are locally modeled on the anti-invariant eigenspace of the relative cohomology
of the canonical cover, the so-called \emph{hat-cohomology}. We fix a quadratic
differential~$q$ of signature~$(\w,\w^-)$ on a surface~$X$ and recall that $\wh{Z}$ and~$\wh{P}$ are preimages of zeroes and poles on the double cover. 
Then the
\emph{hat-homology group} with integral coefficients is defined as
\be
\Gamma := \wh{H}_1(q) \= H_1(\wh{X} \setminus \wh{P}, \wh{Z}, \bC)^-\,,
\ee
where the minus sign denotes the antiinvariant part of the homology with respect
to the involution~$\tau$ whose quotient map is the canonical double cover
$\pi: \wh{X} \to X$.
\par
\emph{Period coordinates}, i.e.\ integrating the one-form on
the double cover against a basis of hat-homology, give a local isomorphism
\be
\Per : U(q) \to H^1(\wh{X} \setminus \wh{P}, \wh{Z}, \bZ)^- \= \Hom(\Gamma,\bC)
\ee
on a neighborhood $U(q)$ of~$q$ in the moduli space of quadratic differentials.
Note that if all the entries of~$\w$ are odd, the hat-homology group
is unchanged if we do not consider
homology relative to the zeros. ('The principal strata of quadratic differentials
have no REL'.)
\par
To globalize the period map we fix a trivialization of the hat-homology group.
That is, we fix a reference differential $(X_0,q_0)$ and define
\bes
\qmoduli[g][\Gamma](\w,\w^-) \= \{(X,q,\rho) \in
\qmoduli[g](\w,\w^-), \quad
\rho:  \wh{H}_1(q_0) \overset{\cong}{\longrightarrow} \wh{H}_1(q) \}\,,
\ees
the space of \emph{period-framed quadratic differentials} of
signature~$(\w,\w^-)$.
\par
\medskip
\paragraph{\textbf{(Teichm\"uller) framed quadratic differentials}}
For fixed discrete data $(g,b,\w^-,\w)$ we denote by $\FQuad(\sow)$ the moduli space
of $\sow$-framed quadratic differentials. This moduli space is a manifold,
but non-connected. We denote by $\FQuad^\circ(\sow)$\footnote{The $\circ$ should
remind of the symbol for the connected component of the identity in a topological
group.} a connected component, in applications typically singled out to contain
a given $\sow$-framed differential.
\par
These spaces are strata of a vector bundle.
The top dimensional stratum of this vector
bundle is $\FQuad(\surf_{\Tri})$ with simple weighted decorations.
\par
\medskip\paragraph{\textbf{Mapping class group action}}
In our context two mapping class groups are important. In general,
the \emph{(unpunctured) mapping class group of a marked surface~$\surf$} is the group
$\MCG(\surf)$ of isotopy classes of diffeomorphisms~of $\surf$ relative to
the boundary and marked points. Similarly we define the \emph{full
mapping class group} $\MCG(\sow) = \MCG(\w,\w^-)$ as diffeomorphisms with
the additional condition to respect finite critical points and their weight (set-wise).
\par
The mapping class group acts on the set of all Teichm\"uller framings
by precomposition. Obviously $\Quad(\sow)=\FQuad(\sow)/\MCG(\sow)$
as orbifolds.
\par

\subsection{Walls have ends and homotopies to tame paths}
\label{sec:wallsends}

In our case, just as for the GMN-differentials
treated  in \cite[Section~5]{BS15}, the space $\Quad(\sow)$ has
a stratification by the number of horizontal saddle connections. The difference
is that the number of horizontal trajectories emerging from a zero is not three,
but~$w_i+2$ if the zero is of order~$w_i$. This means that the number $s_q$ of
saddle trajectories, the $r_q$ recurrent trajectories
and the number $t_q$ of separating trajectories satisfy
\be
k \,:= \, r_q + 2s_q + t_q \= \sum_{i=1}^{r_1} (w_i+2)\,.
\ee
\par
\paragraph{\textbf{Stratification}}	We define
\be
B_p \,\coloneqq\, B_p(\sow) \= \{q \in  \Quad(\sow): r_q + 2s_q \leq p\}
\ee
and observe that $B_0 = B_1$ is the set of saddle-free differentials by the preceding
observation. There is an increasing chain of subspaces
\be \label{eq:B_i_stratification}
B_0 = B_1 \subset B_2 \subset \cdots \subset B_k \= \Quad(\sow)\,.
\ee
This follows from the lower semicontinuity of the function~$t_q$ on $X$. The space~$B_2$
is called the space of \emph{tame differentials}. We define the stratification
\be
F_p\,\coloneqq\, F_p(\sow) \= B_p \setminus B_{p-1}.
\ee
We observe that $F_0$ is dense, $F_1$ is empty, and $F_2$ consists of differentials
with exactly one saddle trajectory, since the boundary of a spiral domain has
a saddle trajectory (\cite[Lemma~3.1]{BS15}). In fact, we have the more precise
statement from \cite[Lemma~4.11]{BS15} and \cite[Theorem~1.4]{A}.
\par
\begin{lemma} \label{le:B0dense}
	$B_0(\sow)$ is dense in $\FQuad(\sow)$.
	In fact, $\FQuad(\sow)=\bC\cdot B_0(\sow)$.
\end{lemma}
\par
We will see that $B_p$ is not
always locally finite, and even if it is, the relation between the integer~$p$ and the codimension
of~$B_p$ is complicated and depends on~$s_q$ and the geometry of the spiral domains.
\par
We can now state and prove our goal, the generalization of \cite[Proposition~5.8]{BS15}
to zeros of arbitrary order. It follows as a corollary of the following
Proposition \ref{prop:wallhaveends} which is proven in the sequel.
\par
\begin{prop}\label{prop:wallhaveends}
	Suppose that $p>2$ and suppose that the negative part of the signature is not
	$\w^- = (-2)$. Then each component of the stratum~$F_p$ contains a point~$q$
	and a neighborhood $U \subset \Quad_{g} (\w,\w^-)$ of~$q$ such that
	$U \cap B_{p}$ is contained in the locus $\Per(\alpha) \in \bR$ for some
	$\alpha \in \Gamma$, and that this containment is strict in the more precise
	sense that $U \cap B_{p-1}$ is connected.
\end{prop}
\par
The locus $\Per(\alpha)\in \bR$ appearing in the first property is the wall
(i.e., a real codimension one locus) the subsection title
alludes to, and the second statement guarantees the end of this wall.
We will apply this proposition in the following form:
\par
\begin{cor}[{\cite[Proposition~5.8]{BS15}}] \label{cor:pathB2}
	Suppose that the negative part of the signature is not $\w^- = (-2)$.
	Then any path in $\Quad_{g} (\w,\w^-)$ can be homotoped relative to its
	end points to a path in $B_2$.
\end{cor}
\par
\begin{proof}[Sketch of proof] Suppose the path lies in~$B_p$. We
	inductively reduce~$p$ by first perturbing it so that it intersects~$F_p$
	in only finitely many points. For each of them, drag the path along the
	nearby~$B_{p-1}$ to an end of the wall given by
	Proposition~\ref{prop:wallhaveends}, go around and return to the other
	side of the intersection point with the wall.
\end{proof}
\par
\begin{lemma} \label{le:ringshrinking} Let~$R$ be a ring domain in a
surface $(X,q)$ that belongs to a stratum $F_p \subset \Quad_{g} (\w,\w^-)$
and let~$X^c = X \setminus \ol{R}$ be the complement of the closed ring domain.
Then there exists a path~$\alpha: [0,1] \to F_p$ such that $\alpha(0) = (X,q)$,
such that~$X^c$ is unchanged along~$\alpha$ and such that $\alpha(1) = \ol{X^c}$
is the closure of the ring domain complement.
\end{lemma}
\par
\begin{proof} Let~$I$ be the intersection $\ol{X}^c \cap \ol{R}$ of the ring
domain with the rest of the surface.
Let $\beta$ be a saddle connection crossing the ring domain
once. Consider horizontal twists of the cylinder~$R$, i.e.\ the action of the
upper triangular group on~$R$ while not changing~$X^c$. This changes the
period of~$\beta$
by some real number while keeping the lengths of all saddle trajectories fixed.
We choose this twist so that there is no vertical saddle connection emanating
from a zero on $\partial R$ that stays within~$\ol{R}$. (The set of twists where
such a vertical saddle connection does exist is countable.)
This is the first part of the path~$\alpha$.
\par
Now we shrink the height of the cylinder, i.e., the imaginary part of the
period of~$\beta$
to zero. We claim that we stay in $\Quad_{g} (\w,\w^-)$ during this
process. This is proven in detail in \cite[Section~4.3]{AWhighrank} and
sketched in \cite[Section~3.1]{MWboundary}. The idea is to draw the vertical
separatricies in the cylinder until they leave the cylinder. This has to
happen, since otherwise we'd have a vertical spiral domain, the boundary of
which has vertical saddle connections, but we excluded these. These vertical
lines divide the cylinder into rectangles. In the limiting surface at
$\Im(\beta)=0$ the top and bottom of each of these rectangles (considered inside
the surface $\ol{X}^c$ slit open along~$I$) are glued together.
\par
To see that this path stays in $F_p$ note that all union of the rays emanating
into~$I$ and on the boundary of~$R$ are saddle trajectories for each surface
along the path~$p$ just described, including its end points. Since the set of this
rays is constant along~$p$ and since $X^c$ is unchanged along the path, the claim
follows.
\end{proof}
\par
We can also get rid of spiral domains by small perturbation in a fixed stratum.
Recall from \cite[Section~11.2]{Strebel} that the boundary of a spiral domain consists
of saddle trajectories.
\par
\begin{lemma} \label{le:spiralperturbation}
	Let~$S$ be a spiral domain in a surface $(X,q)$ that belongs to a stratum
	$F_p \subset \Quad_{g} (\w,\w^-)$ and let~$X^c = X \setminus \ol{S}$ be the
	complement of the closure of the spiral domain.
	Then there exists a path~$\alpha: [0,1] \to F_p$ such that $\alpha(0) = (X,q)$,
	such that $X^c$ is unchanged along~$\alpha$ and such that $\alpha(1) \setminus
	X^c$ contains a ring domain.
\end{lemma}
\par
\begin{proof} Since $S$ is a spiral domain there is at least one saddle connection
	$\beta$ starting in the interior of the spiral domain~$\ol{S}$ with $\Per(\beta)
	\not\in \bR$, say oriented to have positive imaginary part. To show this
	we can e.g.\ use the decomposition of the spiral domain into rectangles
	from \cite[Section~11.3]{Strebel}: if there was no zero in the interior, this
	decomposition would exhibit the spiral domain actually as a ring domain.
	An arbitrarily small purely imaginary deformation  of~$\beta$ will create a saddle
	trajectory that intersects $\ol{X}^c$ at most at its end points. Since we
	may make the deformation smaller than the shortest saddle connections, no two points
	have collided and we stay in the space $\Quad_{g} (\w,\w^-)$. If after
	this deformation the complement of $X^c$ does not yet contain a ring domain
	it must contain spiral domains and we can repeat the procedure, creating
	a new saddle trajectory at each step. The process has to terminate once $p=2s_q$
	and then $r_q=0$,  i.e.\ the complement of $X^c$ must contain a ring domain.
	\par
	We argue that we stay in $F_p$ along this process. This follows since
	$X^c$ is unchanged in the whole process, and since all
	horizontal trajectories emanating from a zero into the complement of $X^c$ contribute
	to $r_q+2s_q$ at any stage of the process.
\end{proof}
\par
\medskip
\paragraph{\textbf{Proof of Proposition~\ref{prop:wallhaveends}}} The beginning of the
following proof follows \cite[Proposition~5.8]{BS15}, replacing an argument
using generic (in the sense of loc.\ cit.) differentials by an alternative
argument. The second part is based on our version of the surface perturbations.
\par
First recall the following Lemma due to \cite{BS15}, that provides the end of the
wall, if the~$\eta_i$ are independent in hat-homology so that their periods can
be modified independently.
\par
\begin{lemma}[{\cite[Proposition~5.3]{BS15}}] \label{le:nbhdborderline}
Suppose that~$q_0$ has a half-plane or a horizontal strip
bounded by exactly $s$ saddle trajectories~$\gamma_i$, numbered consecutively.
Let $\alpha = \sum_{i=1}^s \gamma_i$. Then there is an open neighborhood~$U$ of~$q_0$
such that
\bes
\text{if }q \in U \cap F_p \quad \text{then} \quad \Per(\alpha) \in \bR\,.
\ees
Moreover, $q \in U \cap F_p$ implies that
\be \label{eq:imconstraint}
\Im\Bigl(\sum_{i=1}^k \Per(\gamma_i)\Bigr) \leq 0
\ee
for all $0<k<s$, if the surfaces is oriented such that a half-plane or a
horizontal strip is above the real axis.
\end{lemma}
\par
\begin{proof}[Proof of Proposition~\ref{prop:wallhaveends}]
Consider any point $q \in F_p$ with $p>2$. In this situation there is a borderline
saddle connection. Hence for a sufficiently small neighborhood~$U$ we have
\be \label{eq:UcapFp}
U \cap F_p \, \subseteq \, \{q : \Per(\alpha) \in \bR\}
\ee
for some $\alpha\in\Gamma$ by Lemma~\ref{le:nbhdborderline} and the analogous
\cite[Lemma~5.4]{BS15} for the boundary of a degenerate ring-domain. A
neighborhood~$U$ satisfying the first property is thus available for every~$q$.
\par
Suppose that~$q$ has only one saddle trajectory. Then, since $p>2$, there must
exist a spiral domain and $X^-$ must be non-empty. The intersection $X^+ \cap X^-$
thus consists of one saddle trajectory only. This saddle trajectory has to be
the boundary of a degenerate ring domain, and since any component of~$X^+$ contains
an infinite critical point and a saddle trajectory on its boundary, there is
only one double pole, contradicting the hypothesis.
\par
Consequently, we may assume that there are at least two saddle trajectories.
More precisely, we may assume that $\alpha$ is the class of a union of saddle
trajectories
on the boundary of one component of~$X^+$ and that either there is another
component of~$X^+$ with boundary class $\alpha'$, or that $s \geq 2$ in
Lemma~\ref{le:nbhdborderline}.
\par
If the inclusion in~\eqref{eq:UcapFp} is strict, we are done. This happens
if $s \geq 2$ by~\eqref{eq:imconstraint}, if moreover not all~$\gamma_i$ in this lemma
are hat-proportional and so two of them can be moved independently.
This also happens if $\alpha'$ and $\alpha$ are not hat-proportional (by
tilting~$\alpha'$), or equivalently if they are hat-homologous.
\par
We thus need to analyze the situation that~$q$ has two or more borderline
saddle trajectories $\gamma_1, \gamma_2, \ldots$ and all the borderline saddle
trajectories are  hat-proportional. If a single $\gamma_i$ or a union of these
separates off a subsurface~$X_0$ contained in $X^-$, i.e.\ without poles,
then we are done by the following \emph{subsurface argument}: As long as the
subsurface contains spiral domains we apply Lemma~\ref{le:ringshrinking}, creating
a new cylinder each time. Since the number of horizontal cylinders is bounded by
the topology, this procedure terminates. Now we apply successively
Lemma~\ref{le:spiralperturbation} to each of these cylinders. Note that a
saddle connection crossing a cylinder cannot be hat-homologous to~$\gamma_i$.
Consequently we arrive after the ring domain shrinking process at a point where
we conclude by Lemma~\ref{le:nbhdborderline}.
\par
In general there are three cases depending on the position of the first two, say,
of these trajectories~$\gamma_i$.
\par
\textbf{Case 1:} Suppose both of them are closed. If there is a path starting and ending
at a pole crossing one~$\gamma_i$ but not the other, then the two are not
hat-proportional,
since the Lefschetz pairing (see~\cite[Theorem 6.2.17]{spanier})
\bes
H_1(\wh{X} \setminus \wh{P},\wh{Z}, \bZ)  \times H_1(\wh{X} \setminus \wh{Z},
\wh{P},\bZ)  \to \bZ
\ees
is non-degenerate. The only case not yet covered by the subsurface argument
is that $\gamma_1$ and $\gamma_2$ jointly cut~$X$ into two components, one
of which has no higher order poles, i.e.\ belongs to~$X^-$. We conclude again by the
subsurface argument applied to the component without higher order poles.
\par
\textbf{Case 2:} Suppose none of them is closed. If $\gamma_1 \cup \gamma_2$
does not separate the surface, take a path joining a pole to itself, crossing
$\gamma_1$ once, but not~$\gamma_2$. Take one of the lifts of this path to the
canonical cover and use that Lefschetz pairing to obtain a contradiction to
$[\wh{\gamma}_1] = [\wh{\gamma}_2]$ in hat-homology. If there are poles on both
sides of this loop, the same Lefschetz pairing argument applies. It remains
to deal with the case that $\gamma$ splits off a subsurface in~$X^-$, which
is being dealt with by the subsurface argument.
\par
\textbf{Case 3:} Suppose that precisely one of them, say~$\gamma_1$, is closed.
As in Case~1, if $\gamma_1$ separates off a surface without poles we conclude
by the subsurface argument. Otherwise there is a path starting and ending at
a pole, crossing~$\gamma_1$ once and not crossing~$\gamma_2$. The lift of this
path to~$\wh{X}$ and the Lefschetz pairing invalidates that $\gamma_1$ and
$\gamma_2$ are hat-proportional in $H_1(\wh{X} \setminus \wh{P},\bZ)$.
\end{proof}
\par
\medskip
\paragraph{\textbf{The locus $B_2$ is not locally connected}} 
We show that in general the homotopy to a tame path can
not be performed locally. Consider an elliptic curve whose horizontal leaves
are dense. Make a slit and glue the two sides of the slit (one after rotation
by~$\pi$) to adjacent saddle trajectories on the top of a half-plane.
This results in a surface in $\qmoduli[1](2,1,-3)$, consisting of a spiral
domain and the half plane. The two slit segments $\gamma_1,\gamma_2$ are
hat-homologous. This type of surfaces belongs to~$F_6$, and $B_6$ has locally
$\bR$-codimension one, cut out by $\Per(\gamma_1) \in \bR$.

\subsection{Mixed-angulations from quadratic differentials} \label{sec:MixedAngfromQDiff}

This section gives the geometric justification for introducing $\w$-mixed
angulations by studying quadratic differentials with higher order poles, and shows that adjacency of chambers of
saddle-free quadratic differentials is encoded by flips of
mixed-angulations.
\par
\begin{definition} \label{def:ASfromq}
Let $(X,q, \psi\colon\sow\to X^q)$ be an $\sow$-framed quadratic differential
which is saddle-free. Then there is a $\w$-mixed-angulation $\AS_q$ on~$\sow$
\emph{induced from~$q$} (or more precisely from $(q,\psi)$) where the open
arcs are inherited from (isotopy classes of) generic trajectories.
\end{definition}
\par
The dual graph $\AS_q^*$ also has a geometric interpretation. Its arcs represent
the saddle connections crossing once each horizontal strip. It can be enhanced
with a ribbon-graph structure and as such carries the information about~$\w$.
We refer to~$\AS_q^*$ as the \emph{$\w$-ribbon graph induced by~$q$}. The
trajectory structure on $\sow$ induced by a quadratic differential, hence
the local picture of $\AS_q$ and its dual are illustrated together with the
effect of a forward flip in Figure~\ref{fig:logow}.
\par
Definition~\ref{def:ASfromq} implies that each component of
the locus $B_0 \subset \FQuad(\w,\w^-)$ of saddle-free differentials gives
the same mixed-angulation. We next highlight the role of the locus
of tame differentials:
\par
\begin{prop} \label{prop:EGB2}
Two components of $B_0$ can be connected by an arc~in~$B_2$ with only one point
non-saddle-free if and only if the corresponding $\w$-mixed-angulations are
related by a forward flip.
\end{prop}
\par
\begin{proof} Suppose that the two components of~$B_0$ are connected by such
an arc, which we may homotope to be a small rotation of a saddle connection
near the real axis while fixing the geometry of the rest of the surface. The
question is thus local, in the neighborhood of this saddle connection.
Using a metrically correct drawing, as in the middle of Figure~\ref{cap:flip}
one checks that rotating in clockwise (anticlockwise) direction has the effect
of passing from the leftmost to the rightmost picture in terms of horizontal
strip decompositions. Picking a generic trajectory from the strips, we observe
that this changes the mixed-angulation by a forward flip (backward flip).
\begin{figure}[ht]
	\begin{tikzpicture}[scale=1.2,cap=round,>=latex]
		\foreach \x in {0,36,...,360} {
			\draw[blue] (\x:1cm) -- (\x +36:1cm);
		}
		
		\foreach \x in {0,36,...,360} {
			\filldraw[blue] (\x:1cm) circle(1pt);
		}
		
		\draw (330:0.4cm)  -- (36:1cm);
		\draw (330:0.4cm) -- (72:1cm);
		\draw (190:0.0cm) -- (72:1cm);
		\draw (190:0.0cm) -- (108:1cm);
		\draw (190:0.0cm) -- (144:1cm);
		\draw (190:0.0cm) -- (180:1cm);
		\draw (190:0.0cm) -- (216:1cm);
		\draw (190:0.0cm) -- (252:1cm);
		\draw (190:0.0cm) -- (288:1cm);
		\draw (330:0.4cm) -- (288:1cm);
		\draw (330:0.4cm) -- (324:1cm);
		\draw (330:0.4cm) -- (0:1cm);
		
		\node at (85:0.5cm) {\tiny a};
		\node at (42:0.6cm) {\tiny b};
		\node at (278:0.6cm) {\tiny c};
		\node at (310:0.7cm) {\tiny d};
		\node at (244:0.6cm) {\tiny e};
		\node at (14:0.7cm) {\tiny f};
		
		
		
		\draw[red] (330:0.4cm)  -- (190:0.0cm);
		
		\draw[red,thick,fill=white](190:0.0cm)circle(00.05);
		\draw[red,thick,fill=white](330:0.4cm)circle(00.05);

	\end{tikzpicture}
	\hspace{15pt}
	\begin{tikzpicture}[scale=1.2,cap=round,>=latex]
		
		\begin{scope}
			\clip (-1,-1) rectangle (0,1);
			\draw[thick] (0,0) circle(1);
		\end{scope}
		
		\draw[red] (0,0)  -- (0.5,-0.3);
		\draw[red, thick] (0,0)  -- (-0.5,-0.3);
		
		\draw (0,-1) -- (0,0);
		\draw[thick] (0,0) -- (0,1);
		\draw[thick] (0.5,-1) -- (0.5,-0.3);
		\draw[thick] (0,-1) -- (0.5,-1);
		
		\draw[dashed] (-0.5,-0.86) -- (-0.5,0.86);
		
		\node at (-0.1,0.5) {\tiny e};
		\node at (-0.1,-0.6) {\tiny c};
		\node at (0.6,-0.625) {\tiny d};
		
		\node at (-0.6,-0.625) {\tiny d};
		\node at (-0.6,0.3) {\tiny g};
		
		\draw[red,thick,fill=white](0,0)circle(00.05);
		\draw[red,thick,fill=white](0.5,-0.3)circle(00.05);
		\draw[red,thick,fill=white](-0.5,-0.3)circle(00.05);
		
		\draw [-stealth, dotted] (0.5,-0.3) to [out=225,in=315] (-0.5,-0.3);
		
	\end{tikzpicture}
	\begin{tikzpicture}[scale=1.2,cap=round,>=latex]
		
		\begin{scope}
			\clip (1,1) rectangle (0,-1);
			\draw[thick] (0,0) circle(1);
		\end{scope}
		
		\draw[red] (0,0)  -- (-0.5,+0.3);
		\draw[red, thick] (-0.5,+0.3)  -- (-1,0);
		
		\draw[thick] (0,0) -- (0,-1);
		\draw (0,1) -- (0,0);
		\draw[thick] (-0.5,1) -- (-0.5,0.3);
		\draw[thick] (0,1) -- (-0.5,1);
		
		\draw[dashed] (-1,0) -- (-1,-1);
		\draw[dashed] (0,-1) -- (-1,-1);
		\draw[dashed] (-0.5,+0.3) -- (-0.5,-1);
		
		\node at (0.1,-0.5) {\tiny f};
		\node at (0.1,0.5) {\tiny b};
		\node at (-0.6,0.625) {\tiny a};
		
		\node at (-1.1,-0.5) {\tiny f};
		\node at (-0.6,-0.35) {\tiny h};
		
		\draw[red,thick,fill=white](0,0)circle(00.05);
		\draw[red,thick,fill=white](-0.5,+0.3)circle(00.05);
		\draw[red,thick,fill=white](-1,0)circle(00.05);
		
		\draw [-stealth, dotted] (0,0) to [out=225,in=315] (-1,0);

	\end{tikzpicture}
	\hspace{15pt}
	\begin{tikzpicture}[scale=1.2,cap=round,>=latex]
		\foreach \x in {0,36,...,360} {
			\draw[blue] (\x:1cm) -- (\x +36:1cm);
		}
		
		\foreach \x in {0,36,...,360} {
			\filldraw[blue] (\x:1cm) circle(1pt);
		}
		
		\draw[red, thick] (275:0.4cm)  -- (190:0.0cm);

		\draw (275:0.4cm)  -- (36:1cm);
		\draw (190:0.0cm) -- (36:1cm);
		\draw (190:0.0cm) -- (72:1cm);
		\draw (190:0.0cm) -- (108:1cm);
		\draw (190:0.0cm) -- (144:1cm);
		\draw (190:0.0cm) -- (180:1cm);
		\draw (190:0.0cm) -- (216:1cm);
		\draw (190:0.0cm) -- (252:1cm);
		\draw (275:0.4cm) -- (252:1cm);
		\draw (275:0.4cm) -- (288:1cm);
		\draw (275:0.4cm) -- (324:1cm);
		\draw (275:0.4cm) -- (0:1cm);

		\draw[red,thick,fill=white](190:0.0cm)circle(00.05);
		\draw[red,thick,fill=white](275:0.4cm)circle(00.05);
		
		\node at (85:0.5cm) {\tiny a};
		\node at (244:0.6cm) {\tiny e};
		\node at (48:0.6cm) {\tiny h};
		\node at (5:0.55cm) {\tiny f};
		\node at (267:0.75cm) {\tiny g};
		\node at (295:0.75cm) {\tiny d};
	\end{tikzpicture}
	\caption{Horizontal foliation before and after rotating} \label{cap:flip}
\end{figure}
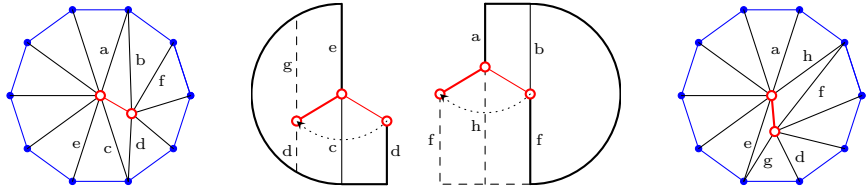
\par
Conversely, if two mixed-angulations differ by a forward flip we take differentials
locally as indicated in the metric picture and rotate the saddle connection
to produce a path as required.
\end{proof}
\par
We will recast this statement in terms of exchange graphs and generalize it
to collapsed surfaces in Section \ref{sec_subsurface}.

\section{Subsurface collapsing}\label{sec_subsurface}

In this section we formalize in the notion of \emph{collapse of a
subsurface}. In the special case of collisions, just a collection of simply
decorated points (but no topology) are pinched. This will be the
simplest ways to realize the generalized Bridgeland-Smith correspondence,
but also the general case will play a role in sequels.
\par
We summarize several notions of exchange graphs, related to tilting,
mutations and flips, and recall the relations between them, thereby introducing
spherical twist groups and braid twist groups. In particular we recall
an isomorphism between exchange graphs for $\pvd(\Gamma)$ and for decorated
marked surfaces with simple weights. This isomorphism will subsequently
be generalized to non-simple weights. As preparation on the topological
side we analyze refinements of mixed-angulations. Finally, we show auxiliary
connectivity results for the graph of refinements to be used in the
next section.

\subsection{Collapse of subsurfaces}\label{sec:collapse}

Let $\Sigma$ be a subsurface of a weighted DMS $\surf_{\w^0}$, possibly disconnected
with connected components $\Sigma_i$. We denote by $c_{ij}$ the (simple
closed) curves such that the union $\cup_j c_{ij}$ forms the intersection of
the boundary of~$\Sigma_i$  with the boundary of $\surf_{\w^0}\setminus \Sigma$. These will be the boundary components of~$\Sigma$ we will be most interested
in. An assignment of integers $\kappa_{ij}$ to each curve $c_{ij}$
is called an \emph{enhancement} (terminology in accordance with
\cite{LMS}) if
\be \label{eq:enhcond}
-\sum_j (\kappa_{ij}+2) + \sum_{k \in \Sigma_i} w^0_k \= 4g(\Sigma_i) - 4
\ee
for each~$i$, where we write $k \in \Sigma_i$, if the $k$-th decoration point
belongs to~$\Sigma_i$.
\par
\begin{definition}\label{def:collapes}
  A \emph{collapse datum} for $\surf_{\w^0}$ is a subsurface $\Sigma$ and an
enhancement $\{\kappa_{ij}\}$ with $\kappa_{ij} \geq 1$ for all $(i,j)$.
The \emph{collapse} of $\Sigma$ in $\surf_{\w^0}$ is the weighted DMS $\colsur$ obtained
by filling each boundary $c_{ij}$ in $\surf_{\w^0} \setminus \Sigma$  by
a disc with one decorated point that carries the weight $w_{ij} = \kappa_{ij} - 2$.
\end{definition}
\par
The condition~\eqref{eq:enhcond} ensures that the weights of $\colsur$
indeed satisfy the condition of a wDMS. The case of enhancements $\kappa =0$
ruled out here is special and requires a different treatment. \textbf{For
simplicity we consider here only collapse data with all $\kappa_{ij} \geq 3$.}
(The remaining cases involve mixed angulations with self-folded edges or
$2$-gons.) A special case of a collapse is a \emph{collision} where
the subsurface is topologically a disc. In case of a collision of
zeros, the positivity condition for the enhancements in the strong sense,
i.e., $\kappa_{ij} \geq 3$, is automatically satisfied.
\par
Consider the special case that $\surf_{\w^0} = \surfo$ has simple weights
and its subsurface~$\subsur$ also has simple weights. We can consider
$\subsur$ as a DMS with $\kappa_{ij}$ marked points on each boundary
component. We denote by $\colsur$ the resulting wDMS and thus we can put the
three surfaces into a \emph{symbolic short exact sequence}
\be\label{eq:ses-surf}
\begin{tikzcd}
    \subsur \ar[r,hookrightarrow] & \surfo \ar[r,rightsquigarrow] & \colsur\,.
\end{tikzcd}
\ee
\begin{figure}[ht]\centering
\makebox[\textwidth][c]{
\begin{tikzpicture}[scale=.3]
\begin{scope}[shift={(-4,0)}]
\draw[thick, fill = green!10](0,0) circle (5);
\node[font=\tiny] at(5,-3) {$C_{11}$};
\foreach \j in {0,...,4}{\draw (90+72*\j:5) coordinate  (w\j);}
\draw[thick,fill = gray!10](0,0) circle (1.5);
\node[font=\tiny] at(0,0) {$C_{12}$};
\foreach \j in {0,...,4}{\draw($(w\j)$)\nn;}
\foreach \j in {0,...,3}{\draw (45+90*\j:1.5) coordinate  (x\j);}
\foreach \j in {0,...,3}{\draw($(x\j)$)\nn;}
\foreach \j in {0,...,8}{\draw (90+40*\j:3) coordinate  (y\j);}
\foreach \j in {0,...,8}{\draw($(y\j)$)\ww;}
\end{scope}

\begin{scope}[shift={(15,0)},scale=1.5]
\draw[fill=cyan!10](-1,-3) to[out=190,in=-90](-7,-1)to[out=90,in=180](-3,4)to[out=0,in=170](0,3.2) to[out=350,in=160](4.2,2)to(4.2,-2)to[out=-160,in=10](-1,-3);
\draw[,fill=green!10](-0.2,-0.2)to[bend right=60](-1,-3) to[out=190,in=-90](-7,-1)to[out=90,in=180](-3,4)to[out=0,in=170](0,3.2)to[bend right=60](-0.2,0.2);
\draw (-5,1.4)\ww (-4.2,1.4)\ww (-3.4,1.4)\ww(-4.6,0.6)\ww
    (-5.4,.6)\ww(-3.8,.6)\ww (-5,-0.2)\ww
    (-4.2,-0.2)\ww (-3.4,-0.2)\ww (2,0)\ww;
\draw [dashed] (-1,-3)to[bend right=60](-1,-0.2);
\draw [dashed] (-0.5,0.2)to[bend right=60](0,3.2);
\draw[thick,fill=white] (1,0) to[bend left=45] (-2,0) (.8,0) to[bend left=-30] (-1.7,0);
\node[font=\tiny] at(-.1,1.8) {$C_{11}$};
\node[font=\tiny] at(-.8,-1.7) {$C_{12}$};
\end{scope}
\begin{scope}[shift ={(21,0)},xscale=0.23,scale=1.5]
\draw[thick,fill = gray!10](0,0) circle (2);
\foreach \j in {0,...,7}{\draw (90+45*\j:2) coordinate  (z\j) ;}
\foreach \j in {0,...,7}{\draw($(z\j)$)\nn;}
\end{scope}

\begin{scope}[shift={(30,0)},scale=1.2]
\draw[thick,fill=cyan!10](0,0) circle (4);
\foreach \j in {0,...,7}{\draw (90+45*\j:4) coordinate  (z\j) ;}
\foreach \j in {0,...,7}{\draw($(z\j)$)\nn;}
\coordinate (q2) at (1,1);
\coordinate (q3) at (-1,1);
\coordinate (q) at (0,-1.5);
\draw[red,ultra thick,fill=white,font=\scriptsize](q3)circle(.25) node[above]{$3$};
\draw[red,very thick,fill=white,font=\scriptsize](q2)circle(.2) node[above]{$2$};
\draw[red,thick,fill=white,font=\scriptsize](q)circle(.17) node[below]{$1$};
\end{scope}
\draw(23.5,0)node{\huge{$\rightsquigarrow$}};
\draw[right hook-latex,>=stealth](2,0)to(4,0);

\draw (0,-6) node {$\subsur$}  (16,-6)node {$\surfo$} (27,-6) node{$\colsur$};
\end{tikzpicture}}
\label{fig:collapse}\caption{A collapse with $\kappa_{11}=5,\kappa_{12}=4$.}
\end{figure}
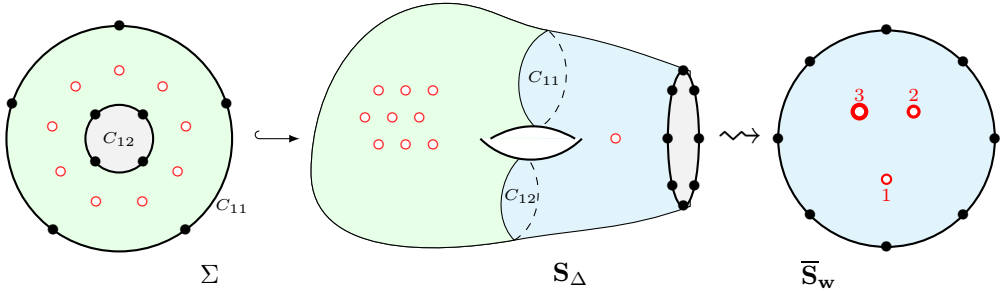
\par
We formalize the structure induced by a quadratic differential on $\sow$,
generalizing the notion induced in Definition~\ref{def:mixedang}
\par
\begin{definition} \label{def:partialtriang}
A \emph{partial triangulation~$\AS$ of a collapsed surface $\colsur$} is a
collection
of open arcs that triangulates the subsurface of $\surfo$ whose complement is
homeomorphic to $\Sigma$, and such that each boundary component $c_{ij}$
of~$\Sigma$ is homotopic in $\colsur \setminus \AS$ to a $(\kappa_{ij}
= w_{ij}+2)$-gon, possibly with ends points identified.
\par
The \emph{forward flip} of a  partial triangulation $\AS$, with respect to
an arc $\gamma\in\AS$, is an operation that moves the endpoints of $\gamma$
anti-clockwise (i.e. by a left fractional twist) along two the adjacent sides
of the smallest $\AS$-gon containing $\gamma$. The inverse of a forward flip is
a \emph{backward flip}, which moves the endpoints clockwise.
\end{definition}
\par
\medskip
\paragraph{\textbf{Refinements}}
Let $\TT = \{\gamma_j\}_{j\in J}$ be a triangulation of $\surfo$ and $\AS$
be a partial triangulation of the collapsed surface $\colsur$. We say that $\TT$
is a \emph{refinement} of~$\AS$ if the preimage of~$\AS$ under
$\surfo \rightsquigarrow \colsur$ is isotopic to a subset of~$\TT$. (Note
that these preimages are well-defined even though the collapse is not an
injective map if a component of~$\Sigma$ has several boundary components.)
We let $I = I(\TT,\AS) \subset J$ be the index set of the \emph{complementary
arcs}, the arcs in $\TT \setminus \AS$.
\par
The same remark justifies:
\par
\begin{definition} \label{def:PTfromq}
Let $(X,q, \psi\colon\sow\to X^q)$ be an $\sow$-framed quadratic differential
which is saddle-free. Then there is a partial triangulation $\AS_q$ on $\colsur$
\emph{induced from~$(q,\psi)$}, the preimage of the mixed-angulation given by
Definition~\ref{def:ASfromq} under the collapse $\surfo \rightsquigarrow \colsur$.
\end{definition}
\par
\begin{cor} \label{cor:EGB2}
Two components of $B_0 \subset \FQuad(\colsur)$ can be connected by an
arc~in~$B_2$ with only one point non-saddle-free if and only if the
corresponding partial triangulations are related by a forward flip.
\end{cor}
\par
\begin{proof} Take the preimage of the construction in Proposition~\ref{prop:EGB2}
under the collapse map.
\end{proof}

\subsection{Exchange graphs and spherical twists}\label{subsection_EG}

The mutation (of quivers), tilting (of categories) and flipping (of edges) operations give rise to a number of
exchange graphs that we summarize here.
\begin{itemize}[nosep]
\item The \emph{unoriented exchange graph} $\uEG(\surf)$ has vertices
  corresponding to triangulations of $\surf$ and edges corresponding to \emph{flips}.
\item Given a mutation equivalence class~$\mathfrak{Q}$ of a quiver,
  the \emph{unoriented cluster exchange graph}
$\uCEG(\mathfrak{Q})$ is the oriented graph whose vertices are cluster tilting
objects in $\calC(\mathfrak{Q})$ and whose edges are mutations between them
(see \cite{keller11} for more details).
\end{itemize}
\par
For the second definition note that mutation equivalences above identify all the
associated cluster categories (without nontrivial autoequivalences as monodromy).
Hence the symbol $\calC(\mathfrak{Q})$ is well-defined. In general
underlined symbols correspond to unoriented (exchange) graphs.
We need the oriented version of these graphs:
\begin{itemize}[nosep]
\item The \emph{exchange graph $\EG(\surf)$ of (an undecorated) surface $\surf$}
is obtained from $\uEG(\surf)$ by replacing each unoriented edge with a 2-cycle.
\item Similarly, the oriented version $\CEG(\mathfrak{Q})$ is obtained from
$\uCEG(\mathfrak{Q})$ by replacing each unoriented edge with a 2-cycle.
\item The \emph{exchange graph of the wDMS~$\sow$} is the directed graph $\EG(\sow)$
whose vertices are partial triangulations and whose oriented edges are forward flips
between them.
\item The \emph{(total) exchange graph} $\EG(\D)$ of a triangulated category $\D$
is the oriented graph whose vertices are all hearts in $\D$
and whose directed edges correspond to simple forward tiltings between them
(Section~\ref{sec:cats}). We abbreviate $\EG(\Gamma) \coloneqq \EG(\pvd(\Gamma))$.
\end{itemize}
\par
We usually focus attention on \emph{a connected component} $\EGp(\Gamma)$
of the exchange graph $\EG(\pvd(\Gamma))$, called the \emph{principal component},
consisting of those hearts that are
reachable by repeated simple tilting from the canonical heart $\h(\Gamma)$ in the quiver
case for $\Gamma=\Gamma(Q,W)$. Similarly, we write $\EGp(\sow)$ for a connected
component of the surface exchange graph. We also write $\EGp(\colsur)$ to
indicate that the wDMS is obtained by a subsurface collapse.
\par
Recall that a graph is called \emph{$(m_1,m_2)$-regular}, if each vertex
has $m_1$ outgoing edges and $m_2$ incoming edges. By definition the graphs
$\EG(\sow)$ and $\EG(\D)$ are $(m,m)$-regular with $m$ being the number of
arcs of the mixed-angulation or of the partial triangulation or the rank
of $K(\calD)$ respectively.
\par
We start the comparison of these graphs in the coarse (undecorated)
cases. If two triangulations are related by a flip, then both the corresponding quivers
with potential are related by a mutation, in the sense of \cite{FST,DWZ}.
\par
\begin{theorem}[{\cite{FST}}]\label{thm:FST} There is an isomorphism
$\uEG(\surf)\cong\uCEG(\surf)$
of the unoriented (triangulation) exchange graphs and cluster exchange graphs.
This isomorphism upgrades to an isomorphism $\EG(\surf)\cong\CEG(\surf)$.
\end{theorem}
\par
\medskip
\paragraph{\textbf{Spherical twist groups}} For further graph comparison we
let $\ST(\Gamma) \leq \Aut(\pvd(\Gamma))$ be the \emph{spherical twist group}
of $\pvd(\Gamma)$, that is the subgroup generated by the set of twists
$\{\Phi_{S}\mid S\in \Sim\calH(\Gamma)\}$,
where the \emph{twist functor} $\Phi_S$ is defined by
\begin{equation}\label{eq:phi+}
    \Phi_S(X)=\Cone\left(S\otimes\Hom^\bullet(S,X)\to X\right)
\end{equation}
Note that $\ST(\Gamma)$ is in fact generated by spherical twists along all reachable
spherical objects, that is all simples in some $\calH\in\EGp\pvd(\Gamma)$, see \cite[\S~2.2]{qiubraid16}.
\par
For a heart $\calH\in\EGp(\pvd(\Gamma))$ we denote by
$\EGp[\calH,\calH[1]]$ the full subgraph whose vertices are intermediate
hearts $\calH\leq \h'\leq \calH[1]$.
The following result is \cite[Theorem~2.10]{KQ2}, based on the unpublished
result of Keller-Nicol\'as announced in \cite[Theoreom~5.6]{keller11}.
\par
\begin{theorem}\label{thm:KN}
Let $\Gamma$ be the Ginzburg dg algebra of some non-degenerate quiver with
potential $(Q,W)$. There is a covering of oriented graphs
\begin{gather}\label{eq:quo}
    \EGp(\pvd(\Gamma))/\ST(\Gamma) \cong \CEG(\Gamma).
\end{gather}
The fundamental domain of $\EGp(\pvd(\Gamma))/\ST(\Gamma)$ is
$\EGp[\calH,\calH[1]]$ for any heart $\calH\in\EGp(\pvd(\Gamma))$,
in the sense that there is an isomorphism between unoriented graph
\[
    \underline{\EGp}[\calH,\calH[1]]\cong\uCEG(\Gamma),
\]
where $\underline{\EGp}$ denotes the underlying unoriented graph of $\EGp$.
\end{theorem}
\par
\par
\medskip

\subsection{Braid groups}
\label{sec:braidgroups}

Two types of braid groups provide the relation between the various exchange
graphs appearing here.
\par
\medskip
\paragraph{\textbf{Surface braid groups}}
One of the standard definitions of the surface braid group $\SBr(\surfo)$
of a DMS (with non-empty boundary) is as the fundamental group of
the configuration space $\operatorname{conf}_{\deco}(\surf)$ of $|\deco|$
(unordered) points in $\surf$. It is a well-known theorem (see e.g.\
\cite[Section~2.4, equation~(5)]{GP})
that the surface braid group is a subgroup of mapping class groups
\be\label{eq:EG/SBr}
    \SBr(\surfo):=\pi_1\operatorname{conf}_{\deco}(\surf)=\ker\big( \MCG(\surfo) \xrightarrow{F_*} \MCG(\surf) \big),
\ee
where $F_*$ is induced by the forgetful map $F\colon \surfo\to\surf$, forgetting the decorations.
There is a natural isomorphism between graphs
\be \label{eq:EGSBr}
    \EG(\surfo)/\SBr(\surfo)=\EG(\surf),
\ee
induced by the induced map $F:\EG(\surfo)\to\EG(\surf)$.
\par
While $\SBr(\surfo)$ is the traditional generalization of the classical braid group,
we need a (normal) subgroup of it, since we would like to restrict
$\EG(\surfo)$ in \eqref{eq:EGSBr} to a connected component.
\par
\medskip
\paragraph{\textbf{Braid twist groups}}
For any closed arc $\eta\in\CA(\surfo)$, there is a (positive) \emph{braid twist}
$\Bt{\eta}\in\MCG(\surfo)$ along $\eta$, as shown in Figure~\ref{fig:Braid twist}.
The \emph{braid twist group} $\BT(\surfo)$ of the decorated marked surface $\surfo$
is the subgroup of $\MCG(\surfo)$ generated by the braid twists $\Bt{\eta}$
for all $\eta\in\CA(\surfo)$.
\par
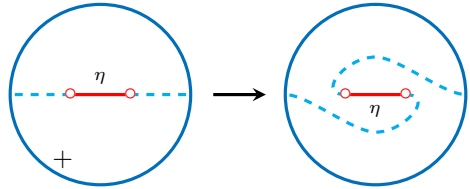
\begin{figure}[ht]\centering
\begin{tikzpicture}[scale=.2]
  \draw[very thick,NavyBlue](0,0)circle(6)node[above,black]{$_\eta$};
  \draw(-120:5)node{+};
  \draw(-2,0)edge[red, very thick](2,0)  edge[cyan,very thick, dashed](-6,0);
  \draw(2,0)edge[cyan,very thick,dashed](6,0);
  \draw(-2,0)node[white] {$\bullet$} node[red] {$\circ$};
  \draw(2,0)node[white] {$\bullet$} node[red] {$\circ$};
  \draw(0:7.5)edge[very thick,->,>=stealth](0:11);
\end{tikzpicture}\;
\begin{tikzpicture}[scale=.2]
  \draw[very thick, NavyBlue](0,0)circle(6)node[below,black]{$_\eta$};
  \draw[red, very thick](-2,0)to(2,0);
  \draw[cyan,very thick, dashed](2,0).. controls +(0:2) and +(0:2) ..(0,-2.5)
    .. controls +(180:1.5) and +(0:1.5) ..(-6,0);
  \draw[cyan,very thick,dashed](-2,0).. controls +(180:2) and +(180:2) ..(0,2.5)
    .. controls +(0:1.5) and +(180:1.5) ..(6,0);
  \draw(-2,0)node[white] {$\bullet$} node[red] {$\circ$};
  \draw(2,0)node[white] {$\bullet$} node[red] {$\circ$};
\end{tikzpicture}
\caption{The braid twist $\Bt{\eta}$}
\label{fig:Braid twist}
\end{figure}
Let $\TT$ be a triangulation of the decorated surface $\surfo$ consisting of
open arcs. The \emph{dual graph} $\TT^*$ of $\TT$ is then a collection of
closed arcs $\eta$. By \cite[Lemma~4.2]{qiubraid16},
$\{\Bt{\eta}\mid\eta\in\TT^*\}$ is a set of generators of $\BT(\surfo)$.
\par
%

For later use, we give a characterization of $\BT(\surfo)$.
By~\eqref{eq:EG/SBr}, any $\xi\in\SBr(\surfo)$ corresponds to $|\Delta|$
paths~$p_i$ on $\surfo$. Their union~$\coprod p_i$ forms a collection of
cycles in $\surf$. The product of these cycles gives a well-defined element
in $\on{H}_1(\surf)$ (more details are given in the forthcoming paper
\cite{Q24}) and we obtain a map, called \emph{topological Abel-Jacobi
map} $\on{AJ}=\on{AJ}_{\surfo}\colon\SBr(\surfo)\to\on{H}_1(\surf)$.
\par

\begin{lemma}\label{eq:BT=kerAJ}
The braid twist group $\BT(\surfo)$ is precisely the kernel of the
map~$\on{AJ}$ provided $|\Delta|>1$.
\end{lemma}
\begin{proof}
By \cite[Proposition~2.7, in particular Figure~4]{QZ2}, the group
$\SBr(\surfo)$ admits a set of generators~$\sigma_i$ for $1\leq i\leq
|\Delta|-1$, and~$\delta_r$ for~$1\leq r\leq 2g+b-1$. Here the~$\sigma_i$
are are braid twists along a collection of arcs connecting the marked
points (within a topological discs). The~$\delta_r$ are point-pushing
diffeomorphisms around simple closed curves based at the first marked point,
namely the~$2g$ curves of a canonical dissection and~$b-1$ curves around the
boundary components. We denote by~$H$ the subgroup of $\SBr(\surfo)$
generated by $\delta_r$, which is isomorphic to $\pi_1(\surf,Z_1)$. In
particular~$H$ is a free group and $H/[H,H]\cong\on{H}_1(\surf)$.
By definition, $\BT(\surfo)$ is a normal subgroup of~$\SBr(\surfo)$ and
contained in~$\ker\on{AJ}$. Thus we express the map $\on{AJ}$ as
$$\SBr(\surfo) \=\BT(\surfo)\cdot H\to H/[H,H],$$
sending the generators of~$\BT(\surfo)$  to the neutral element and the
elements in~$H$ to their classes modulo~$[H,H]$. This implies that
$\BT(\surfo)\cap H\le[H,H]$. A direct calculation shows that
$[\delta_s,\delta_r]$ is in $\BT(\surfo)$ for any $1\le s<r\le 2g+b-1$.
In fact, if we change generators for convenience as in
\cite[Proposition~3.1, in particular Figure~7]{QZ2}
and define (setting $\varepsilon_0 = 1$)
$$\varepsilon_r \= \begin{cases} \delta_r \varepsilon_{r-1} & \text{if $r \not
\in 2\bN_{\leq g}$} \\
\delta_r \varepsilon_{r-2} & \text{if $r \in 2\bN_{\leq g},$}
\end{cases}
\quad \text{i.e.} \quad
\delta_r \= \begin{cases} \varepsilon_r \varepsilon_{r-1}^{-1}
& \text{if $r \not
\in 2\bN_{\leq g}$} \\
\varepsilon_r \varepsilon_{r-2}^{-1} & \text{if $r \in 2\bN_{\leq g},$}
\end{cases}
$$
then such a commutator equals (using $\tau_j = \epsilon_j \sigma_1
\epsilon_j^{-1}$, as drawn in \cite[Figure~8]{QZ2})
\[ [\epsilon_s,\epsilon_r] \= \begin{cases}
  \iv{\tau_{s}b}a\iv{\tau_{r}a\tau_{s}}
  ab\tau_{r}b & \text{if $s+1 \in 2\bN_{\leq g}$,}\\
  \iv{bb\tau_{s}b}a\tau_{r}{a}^{-1}\tau_{s}ab\tau_{r}b & \text{otherwise,}
\end{cases} \]
where $a=\sigma_2\sigma_1\sigma_2^{-1}$ and $b=\sigma_2$. Here the cases
depend on the relative position of $\tau_s$ and $\tau_r$ at~$Z_2$.
(More precisely, if $s<r$ the first case occurs precisely if $\tau_r$ is
before~$\tau_s$ in the counterclockwise order of a neighborhood of $Z_2$ slit
along the arc~$\sigma_1$.)
This implies that $[H,H]\le\BT(\surfo)$ and hence that $\BT(\surfo)\cap
H=[H,H]$ or equivalently $\BT(\surfo)=\ker\on{AJ}$, as claimed.
\end{proof}

\par
We can now summarize the whole discussion in the following two theorems.
The first restricts~\eqref{eq:EGSBr} to a connected component.
\par
\begin{theorem}\label{thm:EGsurfo}
There is an isomorphism
$\EGp(\surfo)/\BT(\surfo)=\EG(\surf)$
between the exchange graph of the undecorated surface and the
braid twist orbits of the exchange graph of the decorated surface.
\end{theorem}
\par
\begin{proof} This is the content of \cite[Remark~3.10]{qiubraid16}.
In fact, Lemma~3.9 in loc.\ cit.\ shows that there is a well-defined
surjective map $\EGp(\surfo)/\BT(\surfo) \to \EG(\surf)$. To show injectivity
it suffices to know that the directed graph of intermediate hearts is
a fundamental domain for the $\BT(\surfo)$-action. The  Lemma~3.8 in loc.\
cit.\ shows that composition of two forward flips is a braid twist and
completes the proof. For the claim on fundamental domains we apply
\cite[Proposition~8.3]{kq}. (We can't apply Theorem~\ref{thm:KN} since
the current theorem is used in its proof.)
\end{proof}
\par
The twist groups in the preceding theorems can be identified and the
corresponding isomorphism can be lifted.
\par
\begin{theorem}\cite{qiubraid16,Q3} \label{thm:STBTiso}
There is an isomorphism $\ST(\Gamma_\TT)\cong\BT(\TT)$ between the twist groups,
sending the standard generators to the standard generators.
Thus the isomorphism (between oriented graphs) in Theorem~\ref{thm:FST} lifts to an isomorphism
\be \label{EGpvdEG}
\EGp\pvd(\Gamma_\TT)\,\cong\,\EGp(\surfo)\,.
\ee
\end{theorem}
As a consequence, we have $\EGp\pvd(\Gamma_\TT)/\ST(\surfo) \cong \EG(\surf)$.

\subsection{Principal parts of exchange graphs}
\label{sec:refprinc}

In order to use the preceding results on $\EG(\surfo)$ to explore the graph
$\EG(\colsur)$ we need to relate partial triangulations and triangulations.
\par
\medskip
\paragraph{\textbf{Principal parts}} Let us fix an initial triangulation~$\TT_0$
of~$\surfo$ and let $\EGp(\surfo)$ be the principal connected component
of $\EG(\surfo)$ containing~$\TT_0$. We define the \emph{principal part}
$\EGb(\colsur)$ of~$\EG(\colsur)$ to be the full subgraph of $\EG(\colsur)$
consisting of the partial triangulations which admit a refinement that belongs
to the component $\EGp(\surfo)$. Note that
\begin{itemize}
  \item we do not claim that $\EGb(\colsur)$ is connected. Moreover,
\item a priori it is not even clear if $\EGb(\colsur)$ consists of connected
  components. That is, it is not a priori clear that vertices in $\EGb(\colsur)$
  that are connected through~$\EG(\colsur)$ are in fact connected
  through~$\EGb(\colsur)$.
\end{itemize}
\par
\medskip
\paragraph{\textbf{Connectedness of refinements}}
Next, we show that when restricted to principal part, certain connectedness property holds.
\par
\begin{prop}\label{refdifferbymutation}
Let $\AS$ be a partial triangulation in $\EGb(\colsur)$. The full subgraph
of the exchange graph~$\EGp(\surfo)$ consisting of refinements of $\AS$ is connected.
\end{prop}
\par
\begin{proof}
Without loss of generality we only need to consider the case when $\subsur$
has one connected component. Take any two refinements $\TT_1$ and $\TT_2$ of~$\AS$
in $\EGp(\surfo)$. Let $T_1, T_2$ be their images in $\EG(\surf)$ under the
forgetful map $F\colon\surfo\to\surf$. By \cite{Hatcher}, there is a flip
sequence connecting the triangulation~$T_2$ and~$T_1$ in the complement of $F(\AS)$.
Such a sequence lifts to a flip sequence of refinements of $\AS$ from $\TT_2$
to some triangulation $\TT_1'$ with the property that $F(\TT_1')=T_1=F(\TT_1)$.
Then $\TT_1$ and $\TT_1'$ differ by an element $b$ of $\BT(\surfo)$ by
Theorem~\ref{thm:EGsurfo}
since these triangulations are both in the principal component $\EGp(\surfo)$.
Moreover, $b$ preserves $\surfo\setminus\subsur$ pointwise as $\TT_1$ and $\TT_1'$ are both refinements of $\AS$.
By Lemma~\ref{lem:SBr}, we know that $b$ is actually in $\BT(\subsur)$.
By Theorem~\ref{thm:EGsurfo} again, the two triangulations of $\subsur$ induced
by $\TT_1$ and $\TT_1'$ are connected by a flip sequence that lifts to a flip
sequence from $\TT_1$ to $\TT_1'$ in the refinements of $\AS$. Composing the
two flip sequences implies the claim.
\end{proof}
\par
\begin{lemma}\label{lem:SBr}
If an element $b$ in $\BT(\surfo)$ preserves $\surfo\setminus\subsur$ pointwise,
then $b$ is actually in $\BT(\subsur)$.
\end{lemma}
\begin{proof}

Since the element $b$ preserves $\surfo\setminus\subsur$ pointwise, it belongs
to $\SBr(\Sigma)$. We conclude that
$$b\in\BT(\surfo)\cap\SBr(\Sigma) \= \ker\on{AJ}_{\surfo}\mid_{\SBr(\Sigma)}
\=\ker\on{AJ}_{\Sigma} \=\BT(\Sigma)$$
by Lemma~\ref{eq:BT=kerAJ}.

\end{proof}

\begin{rmk}
In the case of a collision, i.e., when $\subsur$ is a disk (or the disjoint
union of many disks), the exchange graph $\EG(\subsur)$ is already connected.
Then the lemma above holds automatically.
\end{rmk}
\par
We end this section with a proposition showing that there exists a refinement of
a flip of a partial triangulation in an appropriate sense.
\par
\begin{prop}\label{prop:refine_of_flip}
  Let $\AS$ be a partial triangulation in $\EGb(\colsur)$. Any forward flip
$\AS\xrightarrow{\gamma}\AS^\sharp_\gamma$ in $\EG(\colsur)$ can be refined to
a forward flip $\TT\xrightarrow{\gamma}\TT^\sharp_\gamma$ in $\EGp(\surfo)$.
That is, $\AS$ can be refined to a triangulation~$\TT$ such that
the $\gamma$-forward flip of~$\TT$ composed with forgetting the complementary
arcs is the same as the $\gamma$-forward flip in~$\AS$.
\par
The same statement holds, with 'forward flip' replaced throughout by
'backward flip'. In particular the principal part $\EGb(\colsur)$ is a union
of connected components of $\EG(\colsur)$.
\end{prop}
\par
Yet another restatement of the first statement of the proposition is
that any forward flip of an arc~$\gamma$ in a partial triangulation~$\AS$ in
$\EGb(\colsur)$ leads again to a partial triangulation in $\EGb(\colsur)$.
\par
\begin{proof}
Let $\gamma^\sharp$ be the new arc in $\AS^\sharp_\gamma$. The vertices at the
end points of~$\gamma$ and $\gamma^\sharp$ form a quadrilateral~$\mathrm{Q}$
in $\colsur$. Two of its edges are the counterclockwise adjacent edges
of~$\gamma$ in the $\AS$-polygon~$P_0$ in~$\colsur$ containing $\gamma$.
These two adjacent edges and~$\gamma$ forms two angles~$a$ and~$b$,
drawn in red in Figure~\ref{fig:refine of flip}.
The other two edges are not necessarily in $\AS^\sharp_\gamma$,
see the green dashed arcs in Figure~\ref{fig:refine of flip}.
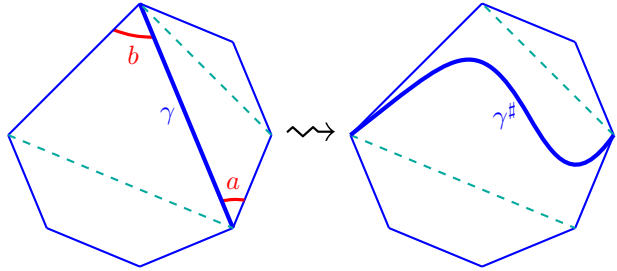
\begin{figure}[ht]\centering
	\makebox[\textwidth][c]{
		\begin{tikzpicture}[scale=.35]
			\foreach \j in {3,...,8}{\draw (90+45*\j:5) edge[blue,thick] (45+45*\j:5);}
			\foreach \j in {0,...,7}{\draw (90+45*\j:5) coordinate  (w\j) (45+45*\j:5) coordinate  (u\j);}
			\path ($(w0)!.5!(w2)$) coordinate (v1)
			($(w0)!.5!(w5)$) coordinate (v2);
        \draw[very thick,red]($(w5)!.12!(w0)$)to[bend left=10]node[above]{$a$}($(w5)!.3!(w6)$);
        \draw[very thick,red]($(w5)!.85!(w0)$)to[bend left=10]node[below]{$b$}($(w0)!.2!(w2)$);
			\draw[blue,thick](w0)to(w2);
			\draw[blue,ultra thick](w5)to node[left]{$\gamma$}(w0);
			\draw[Emerald,thick,dashed](w2)to(w5)(w6)to(w0);
			\foreach \j in {0,...,7}{\draw[red](22.5+45*\j:4.8) coordinate (z\j);}
			\draw(6.5,0)node{\Huge{$\rightsquigarrow$}};
		\begin{scope}[shift={(13,0)}]
			\foreach \j in {3,...,8}{\draw (90+45*\j:5) edge[blue,thick] (45+45*\j:5);}
			\foreach \j in {0,...,7}{\draw (90+45*\j:5) coordinate  (w\j) (45+45*\j:5) coordinate  (u\j);}
			\path ($(w0)!.5!(w2)$) coordinate (v1)
			($(w0)!.5!(w5)$) coordinate (v2);
			\draw[blue,thick](w0)to(w2);
			\foreach \j in {0,...,7}{\draw[red](22.5+45*\j:4.8) coordinate (z\j);}
			\draw[Emerald,thick,dashed](w2)to(w5)(w6)to(w0);
				\draw[blue,ultra thick]
				(w2).. controls +(38:5) and +(120:4) ..($(w0)!.5!(w5)$) node[left]{$\gamma^\sharp$}
				.. controls +(-60:2) and +(-125:2.5) ..(w6);
		\end{scope}
	\end{tikzpicture}}
	\caption{Refinement of a flip $\AS\xrightarrow{\gamma}\AS^\sharp_\gamma$ (collision case)}
	\label{fig:refine of flip}
\end{figure}
We only need to refine $\AS$ to a triangulation~$\TT$ of $\surfo$ so that the
angles~$a$ and~$b$ are not cut by the new added arcs.
If it is a collision, Figure~\ref{fig:refine of flip} shows that
by including the green dashed arcs mentioned above in the refinement, the job
is done.
\par
In general, the decoration in the $\AS$-polygon $P_0$ containing the angle~$a$
is obtained from a boundary component $\partial_0$ of $\Sigma$,
cf. Figure~\ref{fig:refine of flip+}.
Then by identifying the marked points $M_i$ in this component $\partial_0$
with the vertices of $P_0$, the angle~$a$ corresponds to some angle $a$ between segments of~$\partial_0$. As $\kappa_{ij}\ge 3$ one can refine $\AS$ (by
choosing a triangulation of $\Sigma$) so that the angle~$a$ (and similarly
for $b$) is not cut by new added arcs as required. (In fact $\kappa_{ij} \geq 2$
is enough, but for $\kappa_{ij} = 1$ the green arc in
Figure~\ref{fig:refine of flip+} might not exist.)
\begin{figure}[ht]\centering
	\makebox[\textwidth][c]{
		\begin{tikzpicture}[scale=.35]
      \draw[draw=none,fill=orange!20,opacity=.6] (90+45*5:5) -- (90+45*0:5) -- (45+45*8:5) -- (90+45*6:5);
			\foreach \j in {3,...,8}{\draw (90+45*\j:5) edge[blue,thick] (45+45*\j:5);}
			\foreach \j in {0,...,7}{\draw (90+45*\j:5) coordinate  (w\j) (45+45*\j:5) coordinate  (u\j);}
      \path ($(w0)!.5!(w2)$) coordinate (v1)
			($(w0)!.5!(w5)$) coordinate (v2);
        \draw[very thick,red]($(w5)!.12!(w0)$)to[bend left=10]node[above]{$a$}($(w5)!.3!(w6)$);
        \draw[very thick,red]($(w5)!.85!(w0)$)to[bend left=10]node[below]{$b$}($(w0)!.2!(w2)$);
			\draw[blue,thick](w0)to(w2);
			\draw[blue,ultra thick](w5)to node[left]{$\gamma$}(w0);
			\foreach \j in {0,...,7}{\draw[red](22.5+45*\j:4.8) coordinate (z\j);}
    \foreach \j in {4}\draw[orange,dashed,very thick](-90+45*\j:5)edge(13-4,4);
    \foreach \j in {1}\draw[orange,dashed,very thick](-90+45*\j:5)edge(13-4,-4);
    \foreach \j in {2}\draw[orange!50,dashed,very thick](-90+45*\j:5)edge(13+4,-4);
    \foreach \j in {3}\draw[orange!50,dashed,very thick](-90+45*\j:5)edge(13+4,+4);
    \foreach \j in {1,...,4}{\draw(-90+45*\j:5.6)node[font=\footnotesize]{$M_\j$};}
		\begin{scope}[shift={(13,0)}]
      \draw[very thick, fill=orange!20,opacity=.6](4,4)rectangle(-4,-4) (5,0)node{$\partial_0$};
            \draw[Emerald,dashed,very thick](-4,4)to[bend left=-30] (4,-4);
            \draw[red,very thick](-4,-3)to[bend left=30](-3,-4) (-3,-3)node{$a$};
    \foreach \j in {1,...,4}{\draw(135+90*\j:6.3) node[font=\footnotesize]{$M_\j$};}
    \draw[very thick] (0,0)node{$\Sigma$};
    \draw[very thick](4,4)rectangle(-4,-4);
		\end{scope}
	\end{tikzpicture}}
	\caption{Refinement of a flip $\AS\xrightarrow{\gamma}\AS^\sharp_\gamma$ (in general)}
	\label{fig:refine of flip+}
\end{figure}
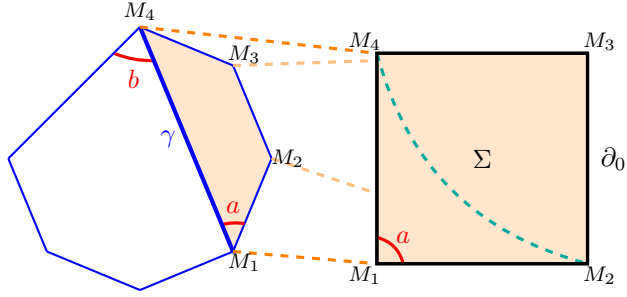
\par
The backward flip statement is proven the same way using the quadrilateral
formed by the end points of $\gamma$ and $\gamma^\flat$.
As a consequence, the graph $\EGb(\colsur)$ is $(m,m)$-regular, where
$m$ is the number of arcs in any mixed-angulation of~$\colsur$. Since we already
remarked the same statement for~$\EG(\colsur)$, the second claim of the proposition
follows.
\end{proof}
\par

\subsection{Example of non-connectedness}
We finish this section by giving an example of exchange graphs,
showing that $\EGb(\colsur)$ is not connected in general. We define
$\EG_\w(\ol{\surf})$ to be the exchange graph of partial triangulations of the
undecorated collapsed surface $\ol{\surf}$. This graph is easy to draw if the
mapping class group of~$\ol{\surf}$ is finite and captures some connectivity
information of the principal part for the following reason:
\par
\begin{lemma}\label{lem:conn}
The forgetful map $F\colon\EG^\bullet(\colsur)/\SBr(\colsur)\to\EG_\w(\ol{\surf})$ is
surjective and hence an isomorphism. As a result, if $\EG_\w(\ol{\surf})$ is not
connected, neither is $\EGb(\colsur)$.
\end{lemma}
\begin{proof}
Given any partial triangulation $\ul{\AS}$ in $\EG_\w(\ol{\surf})$, one can
refine it to a triangulations $\T$ of $\surf$. By \cite{Hatcher}, the exchange
graph $\EG(\surf)$
of the undecorated (non-collapsed) surface with simple weights~$\surf$ is connected
and thus $\T\in\EG(\surf)$ lifts to a triangulation~$\TT$ in the principal
component~$\EGp(\surfo)$ with $F(\TT)=\T$. Restricting~$\TT$ back to $\colsur$, we
obtain a partial triangulation $\AS\in\EGb(\colsur)$ with $F(\AS)=\ul{\AS}$.
\end{proof}
\par
\par
\begin{example} \label{torusexample}
Let $\colsur$ be a torus with one boundary component $\partial$ and one
decoration with weight $\w=3$.
Then $\EG_\w(\ol{\surf})$ and hence $\EGb(\colsur)$ are not connected.
\end{example}
\begin{proof}
Let $\ol{\surf}$ be the undecorated torus with boundary circle~$\partial$.
We identify a fundamental domain of the universal cover~$\ol{\surf}$ with the unit
square in $\mathbb{R}^2$ with~$\partial$ being a (real) bubble at the corner of first
quadrant. The first homology of this surface is simply $H_1(\ol{\surf})=\mathbb{Z}^2$.
We denote by $D_{p,q}$ the Dehn twist along an oriented
simple closed curve $C_{p,q}$ with homology class $H_1(C_{p,q})=(p,q)$
for $(p,q)\in\mathbb{Z}^2$ satisfying $\gcd(p,q)=1$.
The mapping class group of $\ol{\surf}$ is  the group
$$\MCG(\ol{\surf})\=\<X,Y\>/(XYX-YXY)\,\cong\,\Br_3\,,$$ generated by
$X=D_{1,0}$ and $Y=D_{0,1}$. Note that the Dehn twist $D_\partial:=(XY)^6$ is
in the center of $\MCG(\ol{\surf})$.
\par
A partial triangulation $\ul{\AS}$ of $\EG_\w(\ol{\surf})$ in this case is
just a $\w$-mixed-angulation, a pentagon with edges $\gamma_h$, $\gamma_v$,
$\partial$, such that glueing edges different from~$\partial$ yields a torus.
We oriented them so that $\harc,\varc,-\harc,-\varc$ are in anticlockwise
order,  cf.~Figure~\ref{fig:example}.
\begin{figure}[ht]
\begin{tikzpicture}[scale=.5,arrow/.style={->,>=stealth}]
\draw[green,opacity=.2,ultra thick, fill=green!15](0,0)rectangle(4,4);
\draw[arrow,cyan,very thick](0,0)to node[below]{$\gamma_h$}(4,0);
\draw[arrow,blue,very thick](0,0)to node[left]{$\gamma_v$}(0,4);
\draw[arrow,cyan,very thick](0,4)to node[above]{$\gamma_h$}(4,4);
\draw[arrow,blue,very thick](4,0)to node[right]{$\gamma_v$}(4,4);
\draw[fill=white, thick](0,0).. controls +(15:2) and +(75:2) ..(0,0)\nn node[below left]{$(0,0)$}
    (0,4)\nn node[left]{$(r,s)$} (4,4)\nn (4,0)\nn node[below]{$(p,q)$};
\draw[arrow](5,2) to (7,2);\draw(6,2)node[above]{$\mu^\sharp_{\gamma_h}$} ;
\draw[arrow](-3,2) to (-1,2);\draw(-2,2)node[above]{$\mu^\sharp_{\gamma_h^\flat}$} ;
\begin{scope}[xscale=1,shift={(8,0)}]
\draw[green,opacity=.2,ultra thick, fill=green!15](0,0)rectangle(4,4);
\draw[arrow,blue,very thick](0,0)to node[left]{$\gamma_v$}(0,4);
\draw[arrow,blue,very thick](4,0)to node[right]{$\gamma_v$}(4,4);
\draw[arrow,cyan,very thick](2,0)to(4,4);\draw[cyan, very thick](0,0)to(2,4) (2,2)node{$\gamma_h^\sharp$};
\draw[fill=white, thick](0,0).. controls +(5:2) and +(65:2) ..(0,0)\nn
    (0,4)\nn(4,4)\nn(4,0)\nn;
\end{scope}
\begin{scope}[xscale=1,shift={(-8,0)}]
\draw[green,opacity=.2,ultra thick, fill=green!15](0,0)rectangle(4,4);
\draw[arrow,blue,very thick](0,0)to node[left]{$\gamma_v$}(0,4);
\draw[arrow,blue,very thick](4,0)to node[right]{$\gamma_v$}(4,4);
\draw[cyan,very thick](0,0).. controls +(85:1.5) and +(120:1.5) ..(1.5,0);
\draw[arrow,cyan,very thick](1.5,4)to node[left]{$\gamma_h^\flat$}(4,0);
\draw[fill=white, thick](0,0).. controls +(15:1.5) and +(75:1.5) ..(0,0)\nn
    (0,4)\nn(4,4)\nn(4,0)\nn;
\end{scope}
\end{tikzpicture}
\caption{The forward flips of the pentagon on torus}
\label{fig:example}\end{figure}
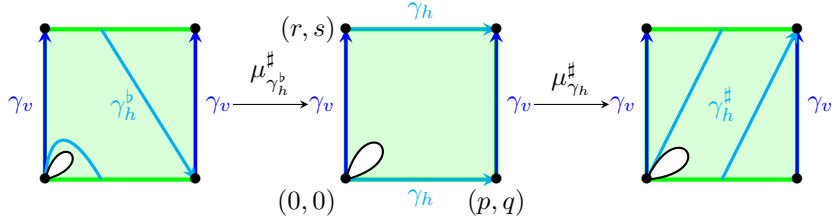
\par
To show non-connectivity of $\EG_\w(\ol{\surf})$ we coarsify the datum given by a
mixed-angulation and find an invariant. First, up to composition with an element in the normal subgroup
generated by~$D_\partial$, a mixed-angulation $\ul{\AS}$ is determined by a 2-by-2
matrix with rows $\ora{h}=H_1(\harc)$ and $\ora{v}=H_1(\varc)$, together with the
location (at which of the four corners) of the boundary~$\partial$.
We represent the position by boxing the corresponding element in the matrix,
called the~\emph{bubble}.
With our orientation conventions for the arcs,
each mixed-angulation is coarsely represented by one of the following four
matrices:
\be\label{eq:bubble}
    \begin{pmatrix}p&q\\\boxed{r}&s\end{pmatrix}\cong
    \begin{pmatrix}\boxed{r}&s\\-p&-q\end{pmatrix}\cong
    \begin{pmatrix}-p&\boxed{-q}\\-r&-s\end{pmatrix}\cong
    \begin{pmatrix}-r&-s\\p&\boxed{q}\end{pmatrix}.
\ee
We shall use, as the normal form, that the bubble is in the left bottom corner.
In such a form the matrix is $\begin{pmatrix}\ora{h}\\\ora{v}\end{pmatrix}=\begin{pmatrix}p&q\\\boxed{r}&s\end{pmatrix}$. Second, we coarsify by considering all matrix entries modulo $3\Z$, and we show that $\ora{h}+\ora{v}\in(\mathbb{Z}_3)^2$ is constant on
each connected component.
\par
The flips in Figure~\ref{fig:example} can be represented as
\[
    \begin{pmatrix}\boxed{p-r}&q-s\\r&s\end{pmatrix}
    \xrightarrow[M_X\cdot]{\mu^\sharp_{\gamma_h^\flat}}
    \begin{pmatrix}p&q\\\boxed{r}&s\end{pmatrix}
    \xrightarrow[M_X^2\cdot]{\mu^\sharp_{\gamma_h}}
    \begin{pmatrix}\boxed{p+2r}&q+2s\\r&s\end{pmatrix},
\]
where, on the level of matrixes, the two flips (i.e., the forward flip at~$\gamma_h^\flat$, resp.\ at~$\gamma_h$) to, resp.\ starting at,~$\ul{\AS}$ are
represented by multiplying on the left by $M_X=\left(\begin{smallmatrix}1&1\\0&1\end{smallmatrix}
\right)$ and $M_X^2$ respectively. Changing to the normal form, we have
\[
    \begin{pmatrix}-r&-s\\\boxed{p-r}&q-s\end{pmatrix}
    \xrightarrow[M_X'\cdot]{\mu^\sharp_{\gamma_h^\flat}}
    \begin{pmatrix}p&q\\\boxed{r}&s\end{pmatrix}
    \xrightarrow[M_X^{''}\cdot]{\mu^\sharp_{\gamma_h}}
    \begin{pmatrix}-r&-s\\\boxed{p+2r}&q+2s\end{pmatrix},
\]
for $M_X'=\left(\begin{smallmatrix}-1&1\\-1&0\end{smallmatrix}\right)$
and $M_X^{''}=\left(\begin{smallmatrix}0&-1\\1&2\end{smallmatrix}\right)$.
It is now straightforward to check that $\ora{h}+\ora{v}\in(\mathbb{Z}_3)^2$
remains unchanged.
Similarly, the other two flips to/at~$\ul{\AS}$ can be represented as
\[
    \begin{pmatrix}r-p&s-q\\\boxed{-p}&-q\end{pmatrix}
    \xrightarrow[M_Y'\cdot]{\mu^\sharp_{\gamma_v^\flat}}
    \begin{pmatrix}p&q\\\boxed{r}&s\end{pmatrix}
    \xrightarrow[M_Y^{''}\cdot]{\mu^\sharp_{\gamma_v}}
    \begin{pmatrix}r+2p&s+2q\\\boxed{-p}&-q\end{pmatrix}
\]
for $M_Y'=\left(\begin{smallmatrix}0&-1\\1&-1\end{smallmatrix}\right)$
and $M_Y^{''}=\left(\begin{smallmatrix}2&1\\-1&0\end{smallmatrix}\right)$.
The row $\ora{h}+\ora{v}\in(\mathbb{Z}_3)^2$ is preserved again.
\par
Using $\ora{h}+\ora{v}\in(\mathbb{Z}_3)^2$ as an invariant for the selected
connected component of the exchange graph, we see that the mixed-angulations
$\begin{pmatrix}1&0\\\boxed{0}&1\end{pmatrix}$
and $\begin{pmatrix}1&1\\\boxed{0}&1\end{pmatrix}$
are not in the same connected component of $\EG_\w(\ol{\surf})$.
\end{proof}


\section{Categorification of collapses}\label{sec:cat sub/col}

In this section, we categorify the constructions in Section~\ref{sec_subsurface},
by associating with a collapse a quotient category. For computations
it is convenient to express this quotient category in terms of triangulations.
The main point of this section is to analyse a subset of hearts of bounded t-structures of
the quotient category that we call of quotient type with respect to the subcategory
that has been collapsed. The leads to a notion of exchange graphs of
these quotient type hearts. The goal of this section, Theorem~\ref{thm:EGiso}, is
to show that the principal part of this exchange graph agrees with the
principal part of the exchange of partial triangulations we introduced previously.

\subsection{The quotient categories associated to collapsed surfaces}
\label{sec:QuotCat}

We have been associating in Section~\ref{sec:MS} a  $\CY_3$-category
$\pvd(\Gamma_\TT)$ to a triangulation~$\TT$ of a wDMS~$\surfo$ with simple weights.
Theorem~A1 in \cite{BQZ} shows that this category~$\pvd(\Gamma_\TT)$ is in fact canonically
associated with~$\surfo$, i.e.\ the derived equivalences given by Proposition~\ref{prop:mutationiso}
can be identified consistently. We call it $\Dsan$. Thus the inclusion
$\Sigma\subset\surfo$ in~\eqref{eq:ses-surf}, together with the discussion in
\ref{subsec_quiver3cat} and \ref{sec:MS}, induces
a short exact sequence of triangulated categories:
\be\label{eq:ses.cats}
\begin{tikzcd}
    0 \ar[r] & \D_3(\subsur) \ar[r] & \Dsan \ar[r] & \D(\colsur) \ar[r] & 0.
\end{tikzcd}
\ee
Equivalently, we define the category $\D(\colsur)$ as the Verdier quotient
$\Dsan/\D_3(\subsur)$.
We will now give a more concrete construction of $\D(\colsur)$ by choosing (partial)
triangulations  and show that it is indeed independent of the choices.
\par
\medskip
\paragraph{\textbf{Triangulation of subsurfaces}}
If~$\TT$ is any triangulation of~$\surfo$, we can homotope the arcs to pass
each through one of the marked points on $\partial \Sigma$. In this way,
the subsurface inherits a triangulation~$\TT|_\Sigma$. This triangulation
is obviously a refinement of the mixed-angulation~$\AS$ obtained by forgetting the
edges in~$\TT|_\Sigma$ and collapsing to~$\colsur$.
\par
This defines an inclusion of triangulated categories
$\pvd(\TT|_\Sigma) \to \pvd(\TT)$. Any other refinement of~$\AS$ differs
from~$\TT$ by a sequence of flips, see Proposition~\ref{refdifferbymutation},
i.e.\ of mutations in the vertices of the corresponding subquiver.
Then the quotient category $\D(\colsur)$ can be realized as
$\pvd(\TT)/\pvd(\TT|_\Sigma)$. The independence of~$\D(\colsur)$ of the chosen
refinement is given by the following proposition.
\par
\begin{prop}\label{quotQ}
Let $I\subseteq Q_0=\{1,\dots, n\}$ be a non-empty subset, and $\mu$ be a
sequence of mutations at vertices $k_j\in I$ (possibly repeated), then we have the following
equivalence
\[\pvd(Q,W)/\pvd\big((Q,W)_I\big) \,\simeq \, \pvd\big(\mu (Q,W)\big)/\pvd\big((\mu (Q, W))_I\big)\]
of quotient triangulated categories.
\end{prop}
\begin{proof}
Let $I^{(0)}=I\subset Q_0=\{1,\dots, n\}$ be a non-empty subset. For simplicity we omit the potentials in the proof. Denote the
quiver~$Q$ as~$Q^{(0)}$. For $j\geq 0$, let $Q^{(j+1)}=\mu_{k_j}Q^{(j)}$, for
some $k_j\in Q_0^{(j)}$. We know that $(Q^{(j)},W^{(j)})$ are all right-equivalent
and their associated triangulated categories~$\pvd({Q^{(j)})}$ are equivalent,
see \cite{KY}. The equivalence also holds for $\pvd({Q^{(j)}}_{|I})$
and $\pvd({\mu_{k_j}(Q^{(j)}}_{|I}))$.
By the mutation-restriction compatibility, the latter is the same
as $\pvd({(\mu_{k_j}Q^{(j)})_{|I}})$ and the equivalence in the bottom level of the following diagram is compatible with the one above
\[\xymatrix{
	\pvd(Q^{(j)}) \ar@{-}[rr]^\simeq_{\text{\cite{KY}}} \ar@{}[d]^\bigcup && \pvd\left(Q^{(j+1)}\right)\ar@{}[d]^\bigcup\\
	\pvd\left({Q^{(j)}}_{|I}\right) \ar@{-}[r]^\simeq_{\text{\cite{KY}}} & \pvd\left(\mu_{k_j}({Q^{(j)}}_{|I})\right)\ar@{-}[r]^\simeq_{\text{\cite{LabaFragQP}}} & \pvd\left({Q^{(j+1)}}_{|I}\right)
}\]
The diagram is therefore commutative. For the statement to follow, is then enough to know that the Ginzburg category associated to a sub-quiver is a thick triangulated sub-category of the Ginzburg category attached to the whole quiver, hence the quotient is well-defined.
\end{proof}
\par

\subsection{Partial triangulations induce hearts of quotient type}
\label{sec:MAquotheart}
Now we focus on hearts of bounded t-structures in the quotient category $\D(\colsur)$. We will restrict
our attention to certain hearts that we call of quotient type. We show that
partial triangulation induce hearts of quotient type through the choice of a
refinement. We start with a general fact.
\par
\begin{prop}[{\cite[Proposition~2.20]{antieau}}]\label{prop_antieau}
Let $i:\cl C \to \cl D$ be a t-exact fully faithful functor of triangulated
 categories equipped with bounded t-structures,
with a well-defined quotient functor $j:\cl D \to \cl D/ \cl C$.
Let $\h_\D$ and $\h_\calC$ be the two hearts in $\D$ and $\calD$ respectively.
Then the following are equivalent
	\begin{itemize}
		\item[a)]the essential image $i(\h_\calC)\subset \h_\D$ is a Serre subcategory, and
		\item[b)] the quotient $\calD/\cl C$ has a bounded t-structure such that $j$ is t-exact, whose
		heart is equivalent to $\h_\D/\h_\calC$.
	\end{itemize}
\end{prop}
The (bounded) t-structure corresponding to
$\h_\D/\h_\calC$ in $\calD/\cl C$ of point b) is described in
{\cite[Proposition~2.20]{antieau}}.
\par
\begin{definition}\label{def_induced_heart}
If a heart on a quotient triangulated category arises as described by {\cite[Proposition 2.20]{antieau}},
we say that it is of \emph{quotient type}.
We say moreover that it is \emph{induced} by the hearts in $\calD$ and $\cl C$, or \emph{induced by the heart on $\calD$},
	if that heart on $\cl C$ is obtained by restriction.
\end{definition}
Note that, a priori, a triangulated category $\cal D/\cl C$ may have
many more hearts. The next definition encodes the key restriction on the pairs
of hearts and subcategories we consider. The reader may compare with
\cite[Section~3]{BPPW} for other criteria for hearts (or slicings) to descend
to quotient categories or to be lifted from there.
\par
\begin{definition}Let $\cl V$ be a full thick triangulated subcategory of $\calD$, and $\calH$ a heart of $\calD$.
We say that $\calH$ is \emph{$\cl V$-compatible} if $\calH\cap \cl V$ is a heart of $\cl V$ and it is a Serre full subcategory of $\calH$.
\end{definition}
\par
We now return to our case of interest, i.e., $\calD = \Dsan$ and
$\calV =\Dsub$, for a choice of a collapse $\nu$. We denote by $\pi_\nu$ the quotient functor
\be\label{pi_functor}\pi_\nu:\D\twoheadrightarrow\D/\calV,\ee
and always consider its essential images.
\par
\begin{prop} \label{equal_quot_heart}
Let $\TT$ be any refinement of a partial triangulation~$\AS$ of $\colsur$. Then the
canonical heart $\calH = \calH(\Gamma_\TT)$ is $\calV$-compatible.
Moreover, the quotient heart $\ol{\calH} \simeq \calH/(\calH \cap \calV)$ is independent of the choice of the refinement.
\end{prop}
\par
\begin{proof}
The first statement follows from the fact that
$\h$ is finite and $\h\cap\calV$ is also finite, generated by the simples corresponding to arcs in $\TT\setminus\AS$.
\par
The second statement follows from combining Proposition~\ref{refdifferbymutation}
and Lemma~\ref{cor: equal_quot_heart} (below). More precisely, there is a
sequence of flips/mutation connecting different refinements~$\TT_1$ and~$\TT_2$
of the partial triangulation $\AS$ by Proposition~\ref{refdifferbymutation}
which gives the same quotient hearts of $\calH(Q_{\TT_1})$ and $\calH(Q_{\TT_2})$
by Lemma~\ref{cor: equal_quot_heart}.
\end{proof}
\par
Let $I\subseteq Q_0=\{1,\ldots,n\}$ be a non-empty subset, and
$\mathbf{i}=(i_l^{\epsilon_l},\ldots,i_1^{\epsilon_1})$ be an ordered sequence with $i_j\in I$ and $\epsilon_j\in\{\sharp,\flat\}$.
By the simple tilting formula \eqref{newsimples},
the sequence $\mathbf{i}$ induces a sequence of simple tiltings
\[
    \mu_{\mathbf{i}}\h=\mu_{i_l}^{\epsilon_l}\cdots \mu_{i_1}^{\epsilon_1} \h
\]
for any heart $\h$ whose simples are parameterized by $Q_0$.

\begin{lemma}\label{cor: equal_quot_heart}
The quotient heart is an invariant under simple tilting in $I$, in the sense that for any  $\mathbf{i}=(i_l^{\epsilon_l},\ldots,i_1^{\epsilon_1})$ as  above
\be \label{quotheartsundermutation}
\calH(Q,W)/\calH\left((Q,W)_{|I}\right)
	\= \mu_\mathbf{i}  \calH\left(Q,W\right)  /  \mu_\mathbf{i}\calH\left((Q,W)_{|I}\right).
\ee
\end{lemma}
\begin{proof} We only consider the case of a single mutation corresponding to a simple
tilt, i.e. $\mathbf{i}=i^\epsilon$ for $i\in I$. Possibly repeating the argument
then proves the statement. Let $\Sim\h(Q,W)=\{S_k\mid k\in Q_0\}$.
Take $\D=\pvd(Q,W)$ and $\cl V = \pvd(Q,W)_I$
together with the
hearts $\h=\h(Q,W)$ and $\h'=\mu^\epsilon_{S_i}\h$.
Applying Lemma~\ref{quot_tilted_hearts} below we obtain the lemma.
\end{proof}
\par
\begin{lem}\label{quot_tilted_hearts}
  Suppose that $\calH,\calH'$ are $\cl V$-compatible hearts of bounded t-structures
  in~$\calD$. Then the bounded t-structures induced on the quotient $\calD/\calV$
  by a twice t-exact fully faithful functor $\iota:\cl V\to \calD$,
	\begin{align*}
		&\left(\calV, \calH\cap\calV\right)\to \left(\calD, \calH\right) \text{ and}\\
		&\left(\calV,\calH'\cap\calV\right)\to \left(\calD, \calH'\right),
	\end{align*}
	coincide if $\calH'=\mu_{\torsion}\calH$ at some torsion class $\torsion\subset \calH$ such that $\torsion\subset \calV$.
\end{lem}
An analogous statement holds, and can be proven similarly, for $\calH'=\Phi_S\calH$ for $S\in\cl V$, where $\Phi_S$ is the spherical twist at the simple $S\in\calH$.
\par
\begin{proof} Let $\torsionfree:= \torsion^\perp$ in $\calH$ and $\calH'=\torsionfree\perp\torsion[-1]$ be the $\torsion$-tilted heart in $\calD$.  We consider the diagram
	\[\xymatrix{\h_\calV=\calV\cap\calH \ar[r] & \calH\ar[r] & \calH/\h_\calV \\
		\cl V \ar[r]^\iota\ar@{}[u] |{\bigcap}\ar@{}[d] |{\bigcup} & \calD \ar[r]^\pi\ar@{}[u] |{\bigcap}\ar@{}[d] |{\bigcup} & \calD/\cl V \\
		\h'_\calV=\calV\cap\calH' \ar[r] & \calH' \ar[r] & \calH'/\h'_\calV
	}
	\]
	Serre-ness of $\h_\calV\subset\calH$, ${\h_\calV}'\subset \calH'$
	is equivalent to $\pi$ being $\calH$- and $\calH'$-exact,
	and $\calD/\cl V$ is endowed with bounded t-structures with hearts
	$\pi(\h)$, $\pi(\h')$ 	(Proposition \ref{prop_antieau}).
	By hypothesis, $\torsion\subset \h_\calV,\torsion[-1]\subset {\h_\calV}'$ are in the kernel of the quotient functor. Moreover, by definition of a torsion pair, for any $E\in\calH$, there are $T\in\torsion$, $F\in \torsionfree$, and a short exact sequence
	\[0\to T \to E \to F \to 0,
	\]
	that yields to a short exact sequence in the quotient $\pi(\h$)
	\[ 0\to \pi(T)\simeq 0\to \pi(E)\to \pi(F)\to 0.\]
	This means that for any $\pi(E)\in\pi(\calH)$ there is
	$\pi(F)\in\pi(\torsionfree)$ such that $\pi(E)\simeq \pi(F)$ in
	$\pi(\calH)$. Similarly, for any $E'\in \calH'$, there is
	$G\in\torsionfree$ such that $\pi(E')\simeq \pi(G)$ in
	$\pi(\h')$.
	Hence the fully faithful functor $\pi(\torsionfree)\to \pi(\calH)$
	is also essentially surjective. Therefore, $\pi(\torsionfree)\simeq \pi(\calH)$, and similarly $\pi(\torsionfree)\simeq \pi(\calH')$. We conclude that the bounded t-structures on the quotient $\calD/\cl V$ induced by $\calH$ and $\calH'$ coincide.
\end{proof}

\subsection{The exchange graphs of hearts of quotient type and its principal
part} \label{sec:EGquotheart}

Consider the exchange graph $\EG(\D(\colsur))$ of the quotient category.
We want to relate this exchange
graph to the exchange graph of~$\D(\surf_\Delta)$. We define the \emph{principal
part} $\EGb(\D(\colsur))$ to be full subgraph of~$\EG(\D(\colsur))$
whose vertices can be realized as quotients of $\calV$-compatible hearts~$\h \in
\EGp(\D(\surf_\Delta))$ in a fixed connected component. As in the topological
situation in Section~\ref{sec:refprinc}, it is a priori not clear that
the principal part consists of connected components of $\EG(\D(\colsur))$.
The next proposition prepares to show that this is indeed the case, namely
that simple tilts of quotient hearts stem from simple tilts of~$\h$,
if the heart~$\h$ is conveniently
chosen in terms of some $\Ext$-condition:
\par
\begin{prop}\label{prop:ind.tilting}
Suppose that $\h$ is finite rigid heart in $\D$ with an abelian subcategory~$\calK$
such that $\Sim\calK\subset\Sim\h$. Denote by $\calV=\thick(\calK)$.
Let $S\in\Sim\h\setminus\Sim\calK$ satisfying $\Ext^1(\calK, S)=0$.
Then the simple tilting $\h\xrightarrow{S}\h^\sharp_S$ in $\D$ induces a simple
tilting of quotient hearts
\[
    \ol{\h}\xrightarrow{\ol{S}}\ol{\h}^\sharp_{\ol{S}} = \ol{\h^\sharp_S}
\]
in $\D/\calV$, where $\ol{?}$ is the essential image of $?$ under $\D\twoheadrightarrow\D/\calV$.
\end{prop}
\begin{proof}
As $\h$ and $\calK$ are both abelian and finite,
$\calK$ is Serre in $\h$ and
the image of any simple $T$ in $\Sim\h\setminus\Sim\calK$ is a simple $\ol{T}$ in $\h/\calK$.
By the simple tilting formula \eqref{newsimples},
simples in $\Sim\calK$ remains in $\Sim\h^\sharp_S$ by the $\Ext^1$-vanishing property of $S$.
Thus, $\calK$ is also an abelian Serre subcategory of $\h^\sharp_S$.
By Proposition~\ref{prop_antieau}, the hearts $\h$ and $\h^\sharp_S$ induce two
hearts $\ol{\h}$ and $\ol{ \h^\sharp_S }$ in $\D/\calV$, such that the quotient
functor $\pi$ is t-exact. The t-structures on the quotient are the images of
the t-structures on~$\calD$, therefore $\h\le\h^\sharp_S\le\h[1]$ implies
\[
\overline{\h}\le\overline{\h^\sharp_S}\le\overline{\h}[1].\]
Moreover,
$\<S\>=\h^\sharp_S[-1]\cap\h$ implies
\[
        \<\overline{S}\>\=\overline{\h^\sharp_S}[-1]\cap\overline{\h}.
\]
By Lemma~\ref{lem:nearby_heart}, we see that
$\overline{\h^\sharp_S}$ is indeed a forward tilt of $\ol{\h}$ with respect to the simple $\ol{S}$.
\end{proof}
\par
We can now state the generalization of Theorem~\ref{thm:STBTiso} to non-simple
weights.
\par
\begin{theorem} \label{thm:EGiso}
Fix a triangulation~$\TT_0$ and the component
of $\EGp(\D(\surf_\Delta))$ corresponding to $\pvd(\Gamma_{\TT_0})$.
There is an isomorphism
$$\EGb(\colsur)\cong \EGb(\Dcol)$$
of the principal parts,
determined by~$\TT_0$ and $\EGp(\D(\surf_\Delta))$ respectively, of
the exchange graphs for partial triangulations and for hearts of quotient type.
\par
In particular $\EGb(\D(\colsur))$ is a union of connected components of
$\EG(\D(\colsur))$.
\end{theorem}
\par
\begin{proof} Let $\AS$ be a partial triangulation in $\EGb(\colsur)$ with
a refinement~$\TT$ in the component of~$\TT_0$. Define $\varphi: \EGb(\colsur)
\to \EGb(\Dcol)$ on vertices by mapping~$\AS$ to the quotient~$\ol{\h}({\TT})$ of
the canonical heart~$\h(\TT)$ of~$\pvd(\Gamma_\TT)$ in $\D(\colsur) = \calD/\calV$.
This is well-defined by Proposition~\ref{equal_quot_heart}. The surjectivity
of~$\varphi$ follows from the surjectivity part of the isomorphism~\eqref{EGpvdEG}
from Theorem~\ref{thm:STBTiso}. For the injectivity of~$\varphi$ we combine
the injectivity part of this isomorphism with Proposition~\ref{refdifferbymutation}
and Proposition~\ref{equal_quot_heart}.
\par
We now consider the edges. For any forward flip $\AS\xrightarrow{\gamma}
\AS^\sharp_\gamma=\AS'$ in $\EGb(\colsur)$, by Proposition~\ref{prop:refine_of_flip},
we can refine it to a forward flip $\TT\xrightarrow{\gamma}\TT^\sharp_\gamma=\TT'$
in $\EGp(\surfo)$ with the property that there is no arrow from $\gamma$ to any
open arc in $\TT\setminus\AS$ in $Q_\TT$.
Let $\h(\TT)\xrightarrow{S}\h(\TT')$ be the simple tilting corresponding
to $\TT\xrightarrow{\gamma}\TT'$, i.e., so that the simple~$S$ corresponds
to the arc~$\gamma$. Let $\calK$ be the subcategory of $\h(\TT)$ generated by
the simples in $\Sim\h(\TT)$ corresponding to arcs in $\TT\setminus\AS$.
By \cite[Lemma~2.15]{KY} the no-arrow-condition above implies $\Ext^1(\calK,S)=0$.
Then by Proposition~\ref{prop:ind.tilting} the simple tilting at~$S$
induces a ('quotient') simple tilting $\h(\AS)\to\h(\AS')$, and this
is indeed an edge in $\EGb(\Dcol)$. Conversely, every edge in $\EGb(\Dcol)$
arises by definition (and~\eqref{EGpvdEG}) from a flip $\h(\TT)\xrightarrow{S(\gamma)}
\h(\TT')$ between triangulations in the component of~$\TT_0$. This gives
rise by definition to an edge in $\EGb(\colsur)$. We have thus shown
that $\varphi$ is indeed a graph isomorphism.
\par
For the last statement recall from the end of the proof of
Proposition~\ref{prop:refine_of_flip} that $\EGb(\colsur)\cong \EGb(\Dcol)$ is
an $(m,m)$-regular graph, where $m$ is the number of edges in any $\w$-mixed-angulation
of $\colsur$. On the other hand, $\EG(\colsur)$ has at most $m=\mathrm{rank}(K(\Dcol))$
many edges. Since  $\EGb(\colsur)$ is defined as a (full) subgraph of $\EG(\colsur)$,
there cannot be any edges of $\EGb(\colsur)$ connecting a vertex of $\EGb(\colsur)$
to a vertex outside this subgraph. It must thus consist of components of $\EG(\colsur)$,
as we claimed.
\end{proof}

\subsection{The symmetry groups} \label{sec:symmetrygroups}
We study the symmetry groups of the surfaces and the categories,
which will be used later. For $\calD=\Dsan$,  we have the following subgroups
\be\label{eq:Nil+Aut}
    \Nil^\circ(\calD) \subset \Aut^\circ_K(\cD) \subset \Aut^\circ(\cD) \subset
\Aut(\cD)
\ee
defined as follows. $\Aut^\circ(\cD)$ is the subgroup of $\Aut(\cD)$ consisting on
autoequivalences of~$\cD$ that preserve the principal component $\Stap(\D)$
corresponding to $\EGp(\D)$. Let $\Aut_K^\circ(\cD)$ be the subgroup of
autoequivalences that moreover act as identity on the Grothendieck
group $K(\cD)$. We call autoequivalences that act trivially on $\Stap(\calD)$
\emph{negligible autoequivalences}. We will also be interested in the quotients
\be\label{eq:pzcAut}
 \mathpzc{Aut}^\circ(\calD) \= \Aut^\circ(\cD)/\Nil^\circ(\calD) \; \text{ and } \;
\mathpzc{Aut}^\circ_K(\calD) \= \Aut^\circ_K(\cD)/\Nil^\circ(\calD)\,. \
\ee
Note that as $\mathpzc{Aut}^\circ(\calD)$ acts faithfully on $\Stap(\D)$,
it also acts faithfully on $\EGp(\D)$.
\par
As preparation, we show that autoequivalences correspond to mapping classes
in the classical case. The following result is implicit in~\cite{KQ2}:
\par
\begin{prop} \label{prop:AutMCGclassical}
There is an embedding $$i_{\TT_0}\colon\aut^\circ(\Dsan) \to \MCG(\surfo)$$
depending on the choice of the initial triangulation $\TT_0$.
Restricted to $\ST(\Gamma_{\TT_0})$, the embedding~$i_{\TT_0}$ becomes the isomorphism
between twist groups in Theorem~\ref{thm:STBTiso}.
\par
The map~$i_{\TT_0}$ surjects onto the subgroup $\MCG^\circ(\surfo)$ of $\MCG(\surfo)$
that stabilizes the component $\EG^\circ(\surfo)$.
\end{prop}
\begin{proof}
Given $f\in\aut^\circ(\Dsan)$, it maps the heart $\h_0$ associated to $\TT_0$ to
some heart $\h\in\EG^\circ(\Dsan)$. Let $\TT$ be the triangulation corresponding to $\h$
in \eqref{EGpvdEG}. Since~$f$ is an autoequivalence, there is an element $\gamma
\in \MCG(\surf)$ that maps the triangulation~$T$ of the corresponding undecorated
surface to~$T_0$. In fact, in this way \cite[Theorem~9.9]{BS15} (see also
\cite[Theorem~4.12]{KQ2}) show that there is short exact sequence
\be\label{eq:ses.aut}
    1\to\mathpzc{ST}(\D) \to \mathpzc{Aut}^\circ(\D) \to \MCG(\surf) \to 1,
\ee
where $\mathpzc{ST}(\D)= \mathpzc{ST}(\Dsan)= \mathpzc{ST}(\Gamma_{\TT_0})$
is the image of the spherical twist group in the quotient by negligible
autoequivalences. It thus suffice to alter~$\gamma$ by an element in the surface
braid group to exhibit an element an element $i_{\TT_0}(f)$ in $\MCG(\surfo)$
that maps $\TT_0$ to $\TT$. This element is unique up to isotopy by the
Alexander Lemma (stating that any homeomophism of a once-decorated disk is
isotopy to
identity if it preserves the boundary pointwise). This uniqueness also shows
that the assignment $i_{\TT_0}(\cdot)$ is actually a group homomorphism.
It is injective as we have taken the quotient by the negligible
autoequivalences. Comparing with the proof of Theorem~\ref{thm:STBTiso}
we see that $i_{\TT_0}$ gives the isomorphism between twist groups there.
\par
For the surjectivity and thanks to~\eqref{eq:ses.aut} we only need to
ensure that the elements in the surface braid group that stabilizes
$\EG^\circ(\surfo)$ are in the image of $i_{\TT_0}(f)$. This stabilizer subgroup
is the braid twist group $\BT(\TT_0)$ by Theorem~\ref{thm:EGsurfo} and
then the isomorphism  $\BT(\TT_0) \cong \ST(\Gamma_{\TT_0})$  yields
the claim, since the latter group is obviously a subgroup of ${\aut}^\circ(\Dsan)$.
\end{proof}
\par
Now let us consider the case of $\colsur$ obtained from $\surfo$ by collapsing~$\subsur$,
and the quotient categories $\Dcol=\Dsan/\calD_3(\subsur)$. We need the follow
subgroups of mapping class groups. For any subgroup $G$ of $\MCG(\surfo)$
 let
\bes
G^\subsur\=\{g\in G\mid g(\subsur)=\subsur\}
\ees be the subgroups leaving invariant the subsurface $\subsur$.
Finally we let $\ul{\MCG}(\Sigma)$ be the mapping class group of the
unmarked surface associated with~$\Sigma$ and let $\ul{\MCG}^\circ({\subsur})
=\ul{\MCG}({\subsur})\cap \MCG^\circ(\surfo)^{{\subsur}}$.
We define the \emph{liftable subgroup} of the mapping class group of the
collapsed surface to be the quotient group and the subgroup
\bes
    \MCG^\bullet_\lift(\colsur)\colon=
        \frac{\MCG^\circ(\surfo)^{{\subsur}}}{\ul{\MCG}^\circ({\subsur})}
\,\subseteq \,
\MCG_\lift(\colsur) \,:=\, \frac{\MCG(\surfo)^{{\subsur}}}{\ul{\MCG}({\subsur})}
\quad \,\subseteq\, \MCG(\colsur)
\ees
Collapsing~$\Sigma$, a subsurface with two non-isomorphic connected components,
say one of them a disc and one with positive genus, such that the
collapse results in the same weights $w_i>1$, shows that in general the inclusion
is strict: mapping class group elements that swap the marked points corresponding
to the higher weights are not liftable.
\par
\medskip
We define the groups of autoequivalences like $\Aut^\bullet(\cD),
\mathpzc{Aut}^\bullet(\cD)$ as in \eqref{eq:Nil+Aut} and \eqref{eq:pzcAut}
by the requirement to stabilize the principal part (instead of a fixed component).
For any subgroup $G \subset \aut(\Dsan)$ we write $G^\Sigma$ for the subgroup that
stabilizes the subcategory~$\calD_3(\subsur)$. Finally we let $\ul{\aut}^\circ(\Dsub)$
be the subgroup of $\aut^\circ(\Dsub)$ consisting of elements that are restricted from
elements in $\aut^\circ(\Dsan)$. We define
\be\label{eq:aut.lift}
    \aut^\bullet_\lift(\Dcol) \= \frac{\aut^\circ(\Dsan)^\subsur}{\ul{\aut}^\circ(\Dsub)}\,.
\ee
We can now state the goal of this subsection:
\par
\begin{prop} \label{prop:AutMCGgen}
There is an embedding $$i_{\TT_0}\colon\aut^\bullet_\lift(\Dcol) \to \MCG(\colsur)$$
depending on the choice of the initial triangulation $\TT_0$ of $\surfo$.
The map~$i_{\TT_0}$ surjects onto the subgroup $\MCG^\bullet_\lift(\colsur)$.
\end{prop}
\par
To \emph{prove Proposition~\ref{prop:AutMCGgen}} we only need the following two lemmas.
\par
\begin{lemma}
There is an isomorphism $\ul{\aut}^\circ(\Dsub)\cong\ul{\MCG}^\circ(\Sigma)$ obtained
by restriction of the isomorphism $\aut^\circ(\Dsub)\to\MCG^\circ(\Sigma)$.
\end{lemma}
\par
\begin{proof} Choose any triangulation~$\TT_0$ of $\surfo$ that can be homotoped
to a triangulation~$\TT_\Sigma$ of $\Sigma$.
By definition, an element in $\ul{\aut}^\circ(\Dsub)$ is restricted
from an element in $\gamma \in \aut^\circ(\Dsan)^\Sigma$.  Regarding this element~$\gamma$
as mapping class on~$S_\Delta$ by Proposition~\ref{prop:AutMCGclassical},
the restriction condition on the categorical side translates by
Lemma~\ref{le:HomeoTriang} to the condition that the mapping class needs to preserve
the collapsing data. Thus $i_{\TT_0}(\gamma) \in \MCG^\circ(\surfo)^\Sigma$ and by
definition the initial automorphism  $i_{\TT_\Sigma}$ restricts
to an injection $\ul{\aut}^\circ(\Dsub)\to\ul{\MCG}^\circ(\Sigma)$.
It is surjective since any element in $\ul{\MCG}^\circ(\Sigma)$ can be regarded
as an element in $\MCG^\circ(\surfo)$ or equivalently via $i_{\TT_0}$ as an element
in $\aut^\circ(\Dsan)$. Restricted to $\Dsub$, we see that it is indeed an element in $\ul{\aut}^\circ(\Dsub)$ and the lemma follows.
\end{proof}
\par
\begin{lemma}
There is an isomorphism $\aut^\circ(\Dsan)^\Sigma\to\MCG^\circ(\surfo)^\Sigma$ obtained
by restriction of the isomorphism $\aut^\circ(\Dsan)\to\MCG^\circ(\surfo)$.
\end{lemma}
\par
\begin{proof}
We regard $\aut^\circ(\Dsan)^\Sigma$ as a subgroup of $\MCG^\circ(\surfo)$
and the condition of stabilizing $\Dsub$ translates topologically to stabilizing
all simple closed arcs in~$\Sigma$, using the correspondence  between closed arcs
and reachable spherical objects in \cite[Thm.~6.6]{qiubraid16}.
By \cite[Lemma~4.6]{qiubraid16},
$\Sigma$ is in fact a neighbourhood of the union of any triangulation dual to the
closed arcs in~$\Sigma$. Thus, the condition of stabilizing the arcs is topologically
equivalent to the condition of stabilizing $\Sigma$ and the lemma follows.
\end{proof}
\par
In the proofs above we have been using the following statement.
\par
\begin{lemma} \label{le:HomeoTriang}
Let $\Sigma_1$ and $\Sigma_2$ be two DMS with simple weights and without punctures,
with associated $\CY_3$ categories~$D_3(\Sigma_i)$. Then $D_3(\Sigma_1)$ is triangle
equivalent to $D_3(\Sigma_2)$ if and only if $\Sigma_1$ is homeomorphic to
$\Sigma_2$.
\end{lemma}
\par
\begin{proof}
The existence of a homeomorphism obviously implies the existence of a triangle
equivalence. For the converse we reconstruct the surface~$\Sigma$ from a single
heart~$\calH$ in a way that is obviously inverse to the construction of~$\calH$ from~$\Sigma$. First, the quiver~$Q$ is the graph given by the simples~$S_i$ in~$\calH$
with edges given by non-trivial $\Ext^1$'s. Second we reconstruct the potential~$W$
from the $\Ext$-algebra of the $\Gamma$-module $S = \oplus_{i \in \Sim(\calH)} S_i$.
This $\Ext$-algebra carries an $A_\infty$-structure, unique up to $A_\infty$-isomorphism.
An explicit construction of this structure is given for any quiver with potential
in \cite[Appendix~A.15]{keller11}. Since in our case the potential consists of
$3$-cycles only, the $\Ext$-algebra is formal, i.e.\ the higher multiplication
maps~$m_n$ for $n \geq 2$ vanish. This means that the model given in loc.\ cit.\ is
the minimal model of the $A_\infty$-isomorphism class and that the multiplication
map~$m_2$ given in log.\ cit.\ is canonically associated with~$S$. This map~$m_2$
determines the potential~$W$ uniquely.
\par
Finally, we reconstruct the surface from~$(Q,W)$, reversing the construction
in Section~\ref{sec:MS}: For $3$-cycle in~$W$ glue a triangle to the corresponding
edges of~$Q$. For each arrow of~$Q$ not in a $3$-cycle, glue a triangle with one
edge as boundary edge to the two edges representing head and tail of the arrow.
For each vertex of~$Q$ to which only a single of the preceding rules apply (in
the since that such a $3$-cycle passes through the vertex, or such an arrow
starts or ends in the vertex), glue a triangle with two boundary edges (and one
boundary marked point). Finally, if for a vertex none of the preceding rules
apply (which happens only for the $A_1$-quiver), then we glue two such triangles.
Since all data used for this reconstruction procedure are preserved by the
equivalence, the lemma follows.
\end{proof}
\par
\par


\section{An extension of the  Bridgeland-Smith correspondence}
\label{sec_iso}

In their paper \cite{BS15} Bridgeland and Smith gave a correspondence roughly
between the space of framed quadratic differentials with only simple zeroes
and stability conditions on the category $\pvd(Q,W)$ where~$(Q,W)$ is the
quiver with potential associated with a saddle-free differential.
In this section we recall their result and extend it to our main result,
a correspondence between the space of framed quadratic differentials with
higher order zeros and certain stability conditions supported on the
quotient categories introduced in Section~\ref{sec:cat sub/col}. Various
mapping class subgroups and groups of autoequivalence have been defined
in Section~\ref{sec:symmetrygroups}.

\subsection{The original Bridgeland-Smith correspondence}	
\label{sec:origsetting}

We state the Bridge\-land-Smith correspondence in the version of \cite{KQ2}
lifted to Teichm\"{u}ller-framed quadratic differentials and in the case that each
boundary component of $\bS$ has at least one marked point, i.e.\ the
quadratic differentials have poles of higher order $\geq 3$ only.
In this way we avoid the extra technicalities of local orbifold structure
(the space $\Quad_\heart(\bS, \MM)$ introduced in \cite{BS15}). For the notation
concerning spaces of quadratic differentials we refer the reader to
Section~\ref{sec:qdiffgeneral}.
\par
Fix a genus~$g$ polar part~$\w^-$ of the signature, the number~$n = 2g-2 + |\w^-|$
of simple zeros with $\surf_{\Delta}$ a reference surface of this type, and fix an
initial Teichm\"{u}ller-framed quadratic differential
$(X_0,q_0,\psi_0) \in \FQuad(\surf_{\Delta})$ and suppose that $q_0$ is saddle-free.
Denote by $\FQuad^\circ(\surf_{\Delta})$ the connected component containing~$q_0$.
Using (the classical version in \cite{BS15} for simple weights of)
Definition~\ref{def:ASfromq} the differentials gives us a triangulation~$\TT_0$,
which gives a quiver with potential $(Q_0,W_0)$ by the construction in
Section~\ref{sec:MS} and thus the category $\calD = \pvd(\Gamma_{\TT_0})$ defined
in Section~\ref{subsec_quiver3cat} with its standard heart $\calH_0$. Fix a
canonical double cover $(\widehat{X_0},\omega_0)$. For each horizontal strip
let $\eta_i$ be the saddle connection crossing that strip, by definition
a closed arc dual to~$\TT_0$.
Let $\widehat{\eta_i}$ be the corresponding hat-homology class, oriented such
that its $\omega_0$-period $\Per(\widehat{\eta_i}) \in \ol{\bH}$. Denoting by
$S_i \in \Sim(\calH_0)$ the simple object corresponding to  $\widehat{\eta_i}$
and define the map~$Z_0$ by $Z_0(S_i) = \Per(\widehat{\eta_i})$. In total we
defined a stability condition $\sigma_0 = (\calH_0,Z_0)$.
\par
Fix an isomorphism  $\theta_0 : \Gamma  \to  \widehat H_1(q_0)$ and fix an
isomorphism $\nu_0: \Gamma \to K(\calD)$. Recording just the central charge
gives a map
\ba \label{def:pi2}
\pi_2: \Stab^\circ(\cD) &\rightarrow \Hom_\Z(\Gamma,\bC),\\
(Z,\calA) &\mapsto (Z \circ \nu_0).
\ea
whose factorization through~$\Stab^\circ(\cD)/\mathpzc{Aut}_K(\cD)$ we denote by
the same symbol. On the other hand, on the space of period-framed quadratic
differentials the projection
gives a map
\ba \label{def:pi1}
\pi_1: \Quad^\Gamma_g(1^r, \w^-) &\rightarrow \Hom_\Z(\Gamma,\bC),\\
	(q, \rho) &\mapsto (\Per(q) \circ \rho \circ \theta_0).
\ea
Note that our notion of Teichm\"uller framing does not frame the double
cover, while the hat-homology depends on the double cover. Thus a priori
it is not clear whether the cover $\FQuad(\surf_{\Delta})$ dominates
the cover $\Quad^\Gamma_g(1^r,\w^-)$. This is proven along with the following
theorem.
\par
\begin{theorem}\label{thm:BS15_iso}
There is an isomorphism of complex manifolds
\be
K: \FQuad^\circ(\surf_{\Delta})   \to \Stab^\circ(\cD)\,.
\ee
The natural covering map $\FQuad(\surf_{\Delta}) \to \Quad_g(1^r,\w^-)$ factors
through a covering $\pi_0: \FQuad(\surf_{\Delta}) \to \Quad^\Gamma_g(1^r,\w^-)$.
The map $K$ commutes with the maps $\pi_1 \circ \pi_0$ and $\pi_2$
to $\Hom(\Gamma,\bC)$ given by periods and by the central charge respectively.
This map~$K$ is equivariant with respect to the action of the mapping class group
$\MCG(\surf_{\Delta})$ on the domain and of the group $\mathpzc{Aut}^\circ(\cD)$ on
the range. The map~$K$ descends to  isomorphisms of complex orbifolds
\ba
K^\Gamma:& \Quad^{\Gamma,\circ}_g(1^r,\w^-) \to
\Stab^\circ(\cD)/\mathpzc{Aut}^\circ_K(\cD) \\
\ol{K}:& \Quad_g(1^r,\w^-) \to \Stab^\circ(\cD)/\mathpzc{Aut}^\circ(\cD) \,,
\ea
where $\Quad^{\Gamma,\circ}_g(1^r,\w^-)$ is the connected component given by
the image of~$\pi_0$.
\end{theorem}
\par
\begin{proof} The existence of~$K^\Gamma$ is the content of \cite[Theorem~11.2]{BS15}.
The map is constructed in Propositions~11.3 and 11.11 and the fact that the
isomorphism descends is argued along with diagram~(11.6) in loc.\ cit. The lift~$K$
is constructed in \cite[Theorem~4.13]{KQ2}. The quotient~$\ol{K}$ is obtained
from~$K$ thanks to Proposition~\ref{prop:AutMCGclassical}. We expand on two arguments
that are only briefly discussed in these sources. First, the orbifold structure
requires that $\mathpzc{Aut}^\circ(\cD)$ acts properly discontinuously. See
the proof of Theorem~\ref{thm:KrhoBihol} for this argument.
\par
\input{doublecoverflip}
Second, we elaborate on the existence of~$\pi_0$, implicitly needed in
\cite[Theorem~4.13]{KQ2}. As we will recall in more detail in the proof of
Theorem~\ref{thm:KrhoBihol},  the map~$K$ is constructed first as a map~$K_0$ on
the locus~$B_0$ of saddle-free differentials proceeding as we did with~$q_0$ above.
This involves the choice of a lift $\widehat{\eta_i}$ of the crossing saddle
connections ~$\eta_i$ on each of the chambers, i.e., connected components of~$B_0$.
This lift also provides the map~$\pi_0$ on each chamber. The map~$K_0$ is then
extended to a map~$K_2$ on the locus~$B_2$ of tame differentials, identifying
chambers adjacent by forward flips and and forward tilts respectively,
using the exchange graph isomorphism~\eqref{EGpvdEG} in Theorem~\ref{thm:STBTiso}.
The continuity of~$K_2$ and~$\pi_0$ on $B_2 \setminus B_0$ then follows
once we checked the following condition, using our standard choice of oriented
lifts of saddle connections so that their periods are~$\ol{\bH}$-valued:
The lift of the flipped standard saddle connection are related to the lifts
of the original standard saddle connection in the same way as the image
simples objects are related, namely by~\eqref{newsimples}\footnote{This is probably
well-known, see around \cite[Proposition~10.4]{BS15}, but we are
uncertain about the role of the orientation of lifts there.}. This is a local
topological statement, to be checked only for the case where $\ext^1$ is
non-trivial. Consider Figure~\ref{doublecoverflip} where shaded slits indicate
branch cuts and black arrows the positive real axis. In particular on the
central horizontal strips the imaginary axis points upwards and all hat-homology
classes are oriented to have positive imaginary part. We can now
verify $\widehat{\eta_{12}} = \widehat{\eta_1} + \widehat{\eta_2}$ in hat-homology,
using the obvious homotopy exhibiting this relation for each
of the two sheets of the canonical double cover.
\end{proof}
\par

\subsection{The correspondence in the generalized setting}
\label{sec:gensetting}

We modify the setup of Section~\ref{sec:origsetting} as follows. Let now~$\w$
be any tuple of non-zero integers and let $\colsur$ be
a wDMS, obtained as collapse of the surface~$\surf_\Delta$. (That is, a collision
$g(\colsur) = g(\surf_\Delta)$ is the easiest possibility to realize this situation
but we also allow the case where the collapse is not a collision.)
Applying Definition~\ref{def:ASfromq} to an initial saddle-free $(X_0,q_0,\psi_0)
\in \FQuad(\colsur)$ now gives a mixed-angulation~$\AS_0$ on~$\colsur$,
and thus a partial triangulation on~$\surf_\Delta$,
which we refine to an initial triangulation~$\TT_0$ on~$\surf_\Delta$.
We let $\ol{\calH}_0$ be the quotient heart in the quotient category $\D(\colsur)$
given by the construction in Theorem~\ref{thm:EGiso}. As above let~$\eta_i$
be the saddle connections crossing strips, and lift them to hat-homology
classes $\wh{\eta_i}$ using the convention in Section~\ref{app:DMSqdiff}
and orient the lifts (which now might be non-closed, i.e., relative periods)
so that $\Per(\wh{\eta_i}) \in \ol{\bH}$. Finally, as above we define
the map~$Z_0$ by $Z_0(S_i) = \Per(\widehat{\eta_i})$ and define the
stability condition $\sigma_0 = (\ol{\calH}_0,Z_0)$ on~$\D(\colsur)$.
\par
The choice of~$q_0$ and~$\TT_0$ above fixes a principal part~$\EGb(\colsur)$
of the exchange graph of partial triangulations and by the isomorphism in
Theorem~\ref{thm:EGiso} also a principal part in~$\EGb(\Dcol)$. We
let $\FQuad^\bullet(\colsur)$ be the connected components whose associated
partial triangulations belong to~$\EGb(\colsur)$. These components include
the one with~$(X_0,q_0,\psi)$, and possibly others.
On the stability side we consider the set
\be \label{eq:defStabbullet}
\Stas(\Dcol) \=\bC\cdot \bigcup_{\calH\in\EGb(\Dcol)} \cube(\calH).
\ee
which is not a priori a union of connected components of $\Stab(\Dcol)$.
\par
For the period and central charge map we fix isomorphisms  $\theta_0 : \Gamma
\to  \widehat H_1(q_0)$ and $\nu_0: \Gamma \to K(\calD)$, keeping in mind
that the rank of~$\Gamma$ depends on $(\w,\w^-)$. We define the
projections~$\pi_1$ and~$\pi_2$ just as in~\eqref{def:pi2}
and~\eqref{def:pi1}, with domains $\Stas(\Dcol)$ and $\Quad_g^\Gamma(\w,\w^-)$
respectively. Again it is not a priori clear whether the cover
$\FQuad^\bullet(\colsur)$ dominates the cover $\Quad^\Gamma_g(\w,\w^-)$.
\par
Moreover, recall the definition of the liftable mapping class groups
and autoequivalences from Section~\ref{sec:symmetrygroups}. We
write $$\Quad_g(\w^\Sigma,\w^-) = \FQuad^\bullet(\colsur) / \MCG^\bullet_\lift(\colsur).$$
This space is a finite cover of the moduli space of quadratic differentials
where the zeros can be permuted only if the realization of $\colsur$ as collapse
allows this, i.e.\ if the permutation can be lifted to a mapping class element
in~$\surfo$.
\par
\begin{theorem} \label{thm:KrhoBihol}
There is an isomorphism of complex manifolds
\be
K: \FQuad^\bullet(\colsur)   \to \Stab^\bullet(\Dcol)\,.
\ee
The natural covering map $\FQuad^\bullet(\colsur) \to \Quad_g(\w,\w^-)$ factors
through a covering $\pi_0: \FQuad^\bullet(\colsur) \to \Quad^\Gamma_g(\w,\w^-)$.
The map $K$ commutes with the maps $\pi_1 \circ \pi_0$ and $\pi_2$
to $\Hom(\Gamma,\bC)$ given by periods and by the central charge respectively.
This map~$K$ is equivariant with respect to the action of the mapping class group
$\MCG^\bullet_\lift(\colsur)$ on the domain and of the
group $\mathpzc{Aut}_\lift^\bullet(\Dcol)$ on the range. The map~$K$ descends to
isomorphisms of complex orbifolds
\ba
K^\Gamma:& \Quad^{\Gamma,\bullet}_g(\w,\w^-) \to \Stab^\bullet(\Dcol)/
\mathpzc{Aut}^\bullet_K(\Dcol) \\
\ol{K}:& \Quad_g(\w^\Sigma,\w^-) \to
\Stab^\bullet(\Dcol)/\mathpzc{Aut}_\lift^\bullet(\Dcol)\,,
\ea
where $\Quad^{\Gamma,\bullet}_g(\w,\w^-)$ are the connected components given by
the image of~$\pi_0$.
\end{theorem}
\par
This result has topological consequences for the principal parts:
\par
\begin{cor} \label{cor:stabbtame}
The principal part $\Stab^\bullet(\Dcol)$ is a union of connected
components of $\Stab(\Dcol)$. In particular these components in the image
of $FQuad(S_w)$ are generic-finite in the sense of Definition~\ref{def:tame}.
\end{cor}
\par
\begin{proof} [Proof of Theorem~\ref{thm:KrhoBihol}]
We proceed similarly as in the proof of Proposition 11.3 in \cite{BS15}
using the stratification~\eqref{eq:B_i_stratification} by the number of
horizontal saddle connections. Recall for this purpose the definition
of $B_p$ and $F_p = B_p \setminus B_{p-1}$ from Section~\ref{sec:wallsends}.
\par
\medskip
\paragraph{\textbf{The maps on the tame locus}}
We first define the map on the saddle-free locus
\[K_{0}: B_0 \rightarrow \Stab^\bullet(\Dcol)\]
associating with a framed differentials a stability condition just as we did
with~$q_0$ in the introductory paragraph of this section.
\par
We check that this map continuously extends to a map~$K_2$ the tame locus. By
Corollary~\ref{cor:EGB2} any two neighboring chambers~$C,C'$ in $B_0$ are
related
by a forward flip of the partial triangulations~$\AS \to \AS'$ at some arc~$\gamma$.
Consider a differential~$q$ on the wall between~$C$ and~$C'$.  We need to show
the continuity of the extension of~$K_0$ to~$0$ along the arc
\be
\rho: t \mapsto e^{it} q \in \FQuad^\bullet(\colsur)  \quad \text{for}
\quad t \in (-\epsilon,\epsilon) \setminus \{0\}\,.
\ee
By the isomorphism of exchange graphs from Theorem~\ref{thm:EGiso} there
is a tilt at some simple~$S$ that relates $K_0(C)$ to a neighboring chamber.
We moreover want to show that this chamber is indeed $K_0(C')$. This follows
since the assignment of stability conditions to partial triangulations is defined
using a refinement to a triangulation and the exchange graph isomorphism is derived
from the isomorphism~\eqref{EGpvdEG} at the level of triangulations, i.e., there
are refinements so that the forward tilt lifts to a tilt  $\TT \xrightarrow{\gamma}
 \TT'$. Next we check the compatibility of the lifted hat-homology classes
$\widehat{\eta_i}$ at this wall-crossing, i.e.\ that the lifted classes of saddle
connections crossing cylinders satisfy the same base change relation as the
corresponding simples, which is~\eqref{newsimples}. We checked this in the proof
of Theorem~\ref{thm:BS15_iso} along with
Figure~\ref{doublecoverflip} from the refining triangulation and this continues
to hold after setting to zero the classes of dual complementary arcs and the
corresponding simples. Now we are in position to use the periods of the
dual arcs in~$\TT$ and the corresponding simples as a coordinate system on a full
neighborhood of the wall, see Definition~\ref{def:wallsandchambers} and
\cite[Lemma~7.9]{BS15}. We conclude, since the periods of all dual arcs
(and hence all simples) vary continuously along the arc~$\rho$.
\par
\medskip
\paragraph{\textbf{Extension to non-tame differentials}}
We now construct $K_{p}$ defined on~$B_p$ inductively, assuming the
existence of $K_{p-1}$ to eventually obtain $K  = K_{k}$, where~$k$
was defined as the maximal number of horizontal saddle connections.
Just as in \cite[Proposition~5.5]{BS15} for any differential~$q \in F_p$
any small rotation $e^{it} q \in B_{p-1}$ for $0<|t|<\ep$ and~$\ep$ small
enough. By induction on~$p$ and $\bC$-equivariance the limiting
generalized stability conditions
\be
\sigma_{\pm}(q) \= \lim_{t \to 0^+} K_{\rho,p-1}(e^{\pm i t} q)
\ee
exist and we need to show that they coincide. The continuity of the
$\bC$-action on framed differentials and $\Stab^\bullet(\Dcol)$ implies that the
agreement of the limits has a locally constant answer. Since walls have ends on
any connected component of~$F_p$ (Proposition~\ref{prop:wallhaveends}),
this agreement locus is all of~$F_p$. We thus obtain eventually a map~$K$
defined everywhere.
\par
\medskip
\paragraph{\textbf{Injectivity}}
Suppose that the differentials~$q_1$ and~$q_2$ have the same image
$K(q_1) = K(q_2)$. Using the $\bC$-action we may assume that both differentials
lie in the interior of chambers. Using the injectivity of the exchange graph
isomorphism in Theorem~\ref{thm:EGiso} shows that the chambers agree and
since periods of crossing saddle connections are coordinates we conclude~$q_1=q_2$.
\par
\par
\medskip
\paragraph{\textbf{The surjectivity}} of~$K$ is obvious from the surjectvity
of the exchange graph isomorphism in Theorem~\ref{thm:EGiso} and the compatibility
with the $\bC$-action.
\par
\medskip
\paragraph{\textbf{The maps~$\pi_0$}} has been defined locally on each
chamber of~$B_0$ by using the lift to hat-homology with periods in~$\bH$.
On each wall crossing this assignment has been verified along with the continuity
of~$K_2$ to be compatible
with the base change in the Grothendieck group $K(\Dcol)$. By definition of
spaces of stability conditions simple objects are labeled globally and
in particular a basis of $K(\Dcol)$ can be chosen globally over
$\Stab^\bullet(\Dcol)$ (and in fact over $\Stab^\bullet(\Dcol)
/\mathpzc{Aut}^\bullet_K(\Dcol)$). This implies that the base change
formula~\eqref{newsimples} is consistent over loops in $\EGb(\Dcol)$, i.e.,
the product of the wall-crossing base changes along a closed loop is the
identity. Since
hat-homology is isomorphic to the Grothendieck group (say in the initial chamber)
this implies that the corresponding base change in hat-homology is consistent
over loops in $\EGb(\colsur)$. This shows that the local definition of~$\pi_0$
gives well-defined function on~$B_2$. We extend~$\pi_0$ to a global function
over higher~$B_k$ using the $\bC$-action just as we did with~$K_2$.
\par
\medskip
\paragraph{\textbf{The compatibility of~$K$ with the projections}}
This compatibility with~$\pi_2$ and $\pi_1 \circ \pi_0$ holds on the
initial chamber by definition,
on all the other chambers by construction of~$\pi_0$ and globally, since
all these maps are equivariant with respect to the $\bC$-action.
\par
\medskip
\paragraph{\textbf{Quotient orbifolds}}
In order to show that $\Stab^\bullet(\Dcol)/\mathpzc{Aut}^\bullet(\Dcol)$
is an orbifold we need to show the properness of the action
of $\mathpzc{Aut}^\bullet(\Dcol)$ and that this group acts with finite
stabilizers. For properness we use that the $\bC$-action, which commutes with
automorphisms, to assume that the two points whose orbits we have to separate
lie in the interior of a chamber. Since $\mathpzc{Aut}^\bullet(\Dcol)$ maps (open)
chambers to chambers, properness is obvious if the two orbits are never in a common
chamber. Otherwise we use that $\mathpzc{Aut}^\bullet(\Dcol)$ preserves the
integral lattice~$\Gamma^\vee$ in $\Hom_{\Z}(\Gamma, \bC)$
and the fact that there is no infinite sequence
in $\GL_d(\bZ)$, where $d = \rk(\Gamma)$, that fails to displace a small
ball. This argument shows moreover the finiteness of stabilizers.
(Compare with \cite[Lemma~3.3]{Smith18}.)
\par
\medskip
\paragraph{\textbf{Descending to $K^\Gamma$ and to $\ol{K}$}} Recall from
Proposition~\ref{prop:AutMCGgen} the existence of an isomorphism
$\aut^\bullet_\lift(\Dcol) \to \MCG^\bullet_\lift(\colsur)$. The equivalence
of~$K$ with respect to this isomorphism follows from the construction
in Proposition~\ref{prop:AutMCGclassical} (using the same initial triangulation),
since flipping arcs commutes with the mapping class group action.
\end{proof}
\par
The following proof adapts the argument of Bridgeland-Smith in a way so
that the complete description of the moduli space of stable objects of
given class (see \cite[Theorem~11.6]{BS15}) can be avoided. We expect the
analogue of this theorem to have more and more complicated case distinctions as
the entries of~$\w$ grow.
\par
\begin{proof}[Proof of Corollary~\ref{cor:stabbtame}]
The image of the holomorphic map~$K$ is obviously open, so we need to check that
it is closed.
Suppose that we have a one-parameter family
$\sigma(t)$ in $\Stab^\bullet(\Dcol)$ with $\sigma(t) = K(q(t))$ in the image
of the comparison isomorphism for $t \neq 0$. To show that $\sigma = \sigma(0)$
is also in the image we need by \cite[Proposition~6.8]{BS15} a lower bound
for the lengths of saddle connections of $q(t)$ as $t \to 0$. We claim that
for each $t \neq 0$ and for each saddle connection $\gamma$ on $q(t)$  there
is stable object $E \in \sigma(t)$ with mass $m(E) = |Z(\gamma)|$. We then obtain
the lower bound of lengths as $t \to 0$, since the masses of stable objects
in the limit $\sigma(0)$ are bounded thanks to the support property of the
stability condition.
\par
To prove the claim we may assume by rotation that $Z(\gamma) \in \bR$. After
a small rotation now $\gamma$ becomes a standard saddle connection crossing
a horizontal strip. (The nearby directions where this is not true have
a saddle connection or a spiral domain, hence a saddle connection, and this
exception set is countable.) The  stability condition corresponding to the
slightly rotated differential has a simple (and hence stable) object of
class~$\alpha$. This property persists after undoing the small rotation.
\end{proof}

\printbibliography

\end{document}